\documentclass[12pt,final]{amsart}
\usepackage[utf8]{inputenc}
\usepackage[T1]{fontenc}
\usepackage[letterpaper,centering]{geometry}
\usepackage{lmodern}
\usepackage{mathrsfs}
\usepackage{amsmath,amssymb,amsfonts}
\usepackage[unicode,pdfborder={0 0 0},final]{hyperref}
\usepackage{enumerate}
\usepackage{multicol}

\usepackage[all]{xy}
{\setbox0\hbox{$ $}}\fontdimen16\textfont2=\fontdimen17\textfont2
\entrymodifiers={+!!<0pt,\the\fontdimen22\textfont2>}
\SelectTips{cm}{12}

\usepackage{xr}
\theoremstyle{plain}
\newtheorem{thm}{Theorem}[section]

\newtheorem{sublem}[thm]{Sublemma}
\newtheorem{lem}[thm]{Lemma}

\newtheorem{question}[thm]{Question}
\newtheorem{cor}[thm]{Corollary}
\newtheorem{prop}[thm]{Proposition}
\theoremstyle{definition}
\newtheorem{defn}[thm]{Definition}
\newtheorem{rmk}[thm]{Remark}
\newtheorem{rmks}[thm]{Remarks}
\newtheorem{example}[thm]{Example}

\numberwithin{equation}{section}
\theoremstyle{remark}
\newtheorem{fact}[thm]{Fact}

\def\owrepositiontag{\mathpalette\owrepositiontaginternal}
\def\owrepositiontaginternal#1#2{\owrepositiontagaux{#1}#2}
\def\owrepositiontagaux#1#2#3{#2\hbox to 1pt{\hss$\mathsurround=0pt#1{#3}$\hss}}

\def\clap#1{\hbox to 0pt{\hss#1\hss}}

\def\mathrlap{\mathpalette\mathrlapinternal}
\def\mathclap{\mathpalette\mathclapinternal}

\def\mathrlapinternal#1#2{%
\rlap{$\mathsurround=0pt#1{#2}$}}
\def\mathclapinternal#1#2{%
\clap{$\mathsurround=0pt#1{#2}$}}

\newcommand{\congru}[3]{#1 \equiv #2 \!\!\mod #3}
\newcommand{\pascongru}[3]{#1 \not\equiv #2 \!\!\mod #3}
\newcommand{\ev}{{\mathrm{ev}}}
\newcommand{\Sq}{{\mathrm{Sq}}}
\newcommand{\alg}{{\mathrm{alg}}}
\newcommand{\cl}{{\mathrm{cl}}}
\newcommand{\orient}{{\mathrm{or}}}
\newcommand{\Gm}{\mathbf{G}_\mathrm{m}}
\newcommand{\isoto}{\myxrightarrow{\,\sim\,}}
\makeatletter
\def\myrightarrow{{\setbox\z@\hbox{$\rightarrow$}\dimen0\ht\z@\multiply\dimen0 6\divide\dimen0 10\ht\z@\dimen0\box\z@}}
\def\myrightarrowfill@{\arrowfill@\relbar\relbar\myrightarrow}
\newcommand{\myxrightarrow}[2][]{\ext@arrow 0359\myrightarrowfill@{#1}{#2}}
\makeatother
\newcommand{\petitoverline}[1]{\mkern1mu\overline{\mkern-1mu#1\mkern-2mu}\mkern2mu}
\newcommand{\Ztilde}{\widetilde\Z}
\newcommand{\gammatilde}{\widetilde\gamma}

\newcommand{\sH}{{\mathscr H}}
\newcommand{\sF}{{\mathscr F}}
\newcommand{\sG}{{\mathscr G}}
\newcommand{\sO}{{\mathscr O}}
\newcommand{\sA}{{\mathscr A}}
\newcommand{\sB}{{\mathscr B}}
\newcommand{\sY}{{\mathscr Y}}
\newcommand{\A}{{\mathbf A}}
\renewcommand{\C}{{\mathbf C}}
\newcommand{\F}{{\mathbf F}}
\renewcommand{\Im}{{\mathrm{Im}}}
\newcommand{\LL}{{\mathbf L}}
\newcommand{\Mpsi}{{\mathscr M}}
\renewcommand{\P}{{\mathbf P}}
\newcommand{\Q}{{\mathbf Q}}
\newcommand{\R}{{\mathbf R}}
\newcommand{\Z}{{\mathbf Z}}
\newcommand{\Zl}{{\Z_{\ell}}}
\newcommand{\Ql}{{\Q_{\ell}}}
\newcommand{\CH}{\mathrm{CH}}
\newcommand{\Hdg}{\mathrm{Hdg}}
\newcommand{\Gal}{\mathrm{Gal}}
\newcommand{\nr}{\mathrm{nr}}
\newcommand{\BM}{\mathrm{BM}}
\newcommand{\NS}{\mathrm{NS}}
\newcommand{\Pic}{\mathrm{Pic}}
\newcommand{\Br}{\mathrm{Br}}
\newcommand{\ci}{\mathscr{C}^{\infty}}
\renewcommand{\phi}{\varphi}
\renewcommand{\emptyset}{\varnothing}
\newcommand{\Ker}{{\mathrm{Ker}}}
\newcommand{\Coker}{{\mathrm{Coker}}}
\newcommand{\Hom}{{\mathrm{Hom}}}

\newcommand{\Homrond}{\mathscr{H}\mkern-4muom}
\newcommand{\et}{\text{ét}}
\newcommand{\tors}{{\mathrm{tors}}}
\newcommand{\elw}{\mathrm{ind}}
\newcommand{\RGamma}{{\mathrm{R\Gamma}}}
\newcommand{\RR}{{\mathrm{R}}}
\newcommand{\labelbetasubtilde}{{\mathrlap{\beta_{\Z}}\phantom{\beta}_{\mbox{\large\smash{\lower6.5pt\hbox{$\mathsurround=0pt\scriptstyle{\tilde{\kern2pt\hbox to 2.5pt{\kern-20pt\phantom{X}\hss}}}$}}}}}}

\advance\textheight 2mm
\advance\topmargin -1mm

\date{January 2nd, 2018; revised on April 30th, 2019}
\title[On the integral Hodge conjecture for real varieties, I]{On the integral Hodge conjecture for\\real varieties, I}

\author{Olivier Benoist}
\address{Institut de Recherche Math\'ematique Avanc\'ee,
UMR 7501, Universit\'e de Strasbourg et CNRS,
7 rue Ren\'e Descartes,
67000 Strasbourg, FRANCE}
\email{olivier.benoist@unistra.fr}

\author{Olivier Wittenberg}
\address{D\'epartement de math\'ematiques et applications, \'Ecole normale sup\'erieure, 45~rue d'Ulm, 75230 Paris Cedex 05, France}
\email{wittenberg@dma.ens.fr}

\begin{document}

\begin{abstract}
We formulate the ``real integral Hodge conjecture'',
a version of the integral Hodge conjecture for real varieties,
and raise the question of its validity
for cycles of dimension~$1$
on uniruled and Calabi--Yau
threefolds and on rationally connected varieties.
We relate it to
the problem of determining the image of the Borel--Haefliger cycle class map
for $1$\nobreakdash-cycles,
with the problem of deciding whether
a real variety with no real point contains a curve of even geometric genus
and with the problem of computing the torsion of
the Chow group of $1$\nobreakdash-cycles of real threefolds.
New results about these problems are obtained along the way.
\end{abstract}

\maketitle

\vspace*{-6.5mm}%
\section*{Introduction}
{\renewcommand*{\thethm}{\Alph{thm}}

One of the central problems in the study of algebraic cycles of
codimension~$k$ on a smooth proper complex algebraic variety~$X$ consists
in determining the subgroup
$H^{2k}_\alg(X(\C),\Z) \subseteq H^{2k}(X(\C),\Z)$
formed by their cycle classes.
Hodge theory provides a chain of inclusions
\begin{align}
\label{eq:intro chain of inclusions}
H^{2k}_\alg(X(\C),\Z) \subseteq \Hdg^{2k}(X(\C),\Z) \subseteq H^{2k}(X(\C),\Z)\rlap{,}
\end{align}
where $\Hdg^{2k}(X(\C),\Z)$ denotes the set of those classes whose images
in~$H^{2k}(X(\C),\C)$ have type~$(k,k)$ with respect to the Hodge
decomposition.
By the Hodge conjecture, the first inclusion should become an equality
after tensoring with~$\Q$.
It is customary to refer to the property that the first inclusion is itself
an equality as the \emph{integral Hodge conjecture}.
Despite its name, this property can fail.
Its study for specific~$X$ and~$k$ has nevertheless played a significant
role in recent years (see \cite{voisinthreefolds}, \cite{ctvoisin},
\cite[Chapter~6]{voisinweyl} and the references therein, and \S\ref{par:complexIHC} for a more detailed discussion).

Let now~$X$ denote a smooth proper real algebraic variety, by which we mean a smooth proper
scheme over~$\R$.  With any algebraic cycle of codimension~$k$ on~$X$,
Borel and Haefliger~\cite{borelhaefliger} have associated
a cycle class in $H^k(X(\R),\Z/2\Z)$.
The study of the subgroup $H^k_\alg(X(\R),\Z/2\Z) \subseteq H^k(X(\R),\Z/2\Z)$
formed by these classes is a classical topic in real algebraic geometry
(see \cite[Chapter~III]{silhol},
\cite[\textsection11.3]{bcr},
\cite{bochnakkucharzonhomologyclasses},
\cite[Chapitres~3--4]{mangoltelivre} and the references therein),
related to the problem of~$\ci$ approximation of submanifolds
of~$X(\R)$
by algebraic subvarieties.

Despite the formal similarity between these two settings, one
critical difference stands out:
being expressed with torsion coefficients,
the definition of the subgroup $H^k_\alg(X(\R),\Z/2\Z)$ misses any information that might come from the
Hodge theory of the underlying complex variety.
The latter, however, does have an influence on this
subgroup (see, \emph{e.g.}, \cite[Chapter~IV, Corollary~4.4]{vanhamelthese} or \cite{MangolteK3}).

The main aim of the present work is to put forward and examine the \emph{real
integral Hodge conjecture}, a statement for real algebraic varieties which
is analogous to the (complex) integral Hodge conjecture recalled above and
whose study refines, at the same time, that of $H^k_\alg(X(\R),\Z/2\Z)$.
In part~I, we formulate it (\textsection\ref{sec:realIHC}) and, focusing on the case
of $1$\nobreakdash-cycles, study its consequences
(\textsection\ref{sec:onecycles}, \textsection\ref{sec:blochogus})
while part~II (that is, \cite{bwpartie2}) establishes particular cases of it, again for $1$\nobreakdash-cycles
(see \cite[\textsection\ref*{BW2-sec:conicbundles}, \textsection\ref*{BW2-sec:Fano}, 
\textsection\ref*{BW2-sec:dPfibrations},
\textsection\textsection\ref*{BW2-subsec:conic bundles non arch}--\ref*{BW2-subsec:cubichypersurfaces}]{bwpartie2}).

More specifically, we define, in~\textsection\ref{sec:realIHC}, a subgroup
\begin{align}
\label{eq:intro hdgG0}
\Hdg^{2k}_G(X(\C),\Z(k))_0 \subseteq H^{2k}_G(X(\C),\Z(k))\rlap{,}
\end{align}
where $G=\Gal(\C/\R)$ acts on the space~$X(\C)$ and on
the group $\Z(k)=(\sqrt{-1}\mkern2mu)^k\Z$
and where $H^{2k}_G(X(\C),\Z(k))$ denotes $G$\nobreakdash-equivariant
cohomology in the sense of Borel, by combining 
the Hodge condition in $H^{2k}(X(\C),\C)$ with a topological condition in $H^{2k}_G(X(\R),\Z(k))$
discovered by Kahn and Krasnov \cite{kahnchern}, \cite{krasnovequivariant}.
When~$k\leq 1$, this topological condition is trivial;
when $\dim(X)\leq 3$ and $k=2$, it is simply the requirement that the
pull-back to any real point of~$X$ of the equivariant cohomology class
under consideration should vanish.

We prove that the subgroup~\eqref{eq:intro hdgG0}
is compatible with cup
products, pull-backs and proper push-forwards
(see~\textsection\ref{subsubsec:compatibility with pushforwards} and
Theorem~\ref{th:stability of topological constraints}).
It contains the subgroup of $H^{2k}_G(X(\C),\Z(k))$ formed by
the equivariant cycle classes of algebraic cycles of codimension~$k$.
The \emph{real integral Hodge conjecture}
refers to the property that every element of~\eqref{eq:intro hdgG0}
is the equivariant cycle class of some algebraic cycle of codimension~$k$.
Just as in the complex situation, the real integral Hodge conjecture sometimes fails,
though it always holds when $k=1$ (an observation due to Krasnov) or $k=\dim(X)$ (see~\textsection\ref{subsubsec:zerocycles}), it is a birational invariant when $k=2$ or $k=\dim(X)-1$
(see~\textsection\ref{subsubsec:birational invariance}), and it holds for projective spaces (see \S\ref{subsubsecps}).

As was noted by Krasnov~\cite{krasnovequivariant} and by van~Hamel~\cite{vanhamelthese},
there exists a canonical map
\begin{align}
\label{eq:intro psiwithpoint}
H^{2k}_G(X(\C),\Z(k)) \to H^k(X(\R),\Z/2\Z)
\end{align}
which sends the equivariant cycle class of any codimension~$k$ algebraic cycle
to its Borel--Haefliger cycle class.
Considering the image~$H$ of $\Hdg^{2k}_G(X(\C),\Z(k))_0$ by this map now leads to
a chain of inclusions
analogous to~\eqref{eq:intro chain of inclusions}:
\begin{align}
H^k_\alg(X(\R),\Z/2\Z) \subseteq H \subseteq H^k(X(\R),\Z/2\Z)\rlap{.}
\end{align}
Obviously, if $H \neq H^k(X(\R),\Z/2\Z)$,
then $H^k_\alg(X(\R),\Z/2\Z)\neq H^k(X(\R),\Z/2\Z)$.
This implication already explains all of the known examples of real varieties~$X$
 such that $H^k_\alg(X(\R),\Z/2\Z)\neq H^k(X(\R),\Z/2\Z)$
for some~$k$
(see Remark~\ref{rk:covers hknotalg}~(ii)).
In~\textsection\ref{subsec:cycle theoretic obs},
we provide an example that cannot be explained by this mechanism.
It is based on a degeneration argument to positive characteristic,
as in~\cite{totarocontreexemples},
to contradict the real integral Hodge conjecture.

Unlike the equality $H^k_\alg(X(\R),\Z/2\Z)=H^k(X(\R),\Z/2\Z)$,
the real integral Hodge conjecture turns out to be an interesting property
when $X(\R)=\emptyset$ as well.  When $X(\R)=\emptyset$ and~$X$ has pure dimension~$d$,
we construct a canonical map
\begin{align}
\label{eq:intro psiwithoutpoint}
H^{2d-2}_G(X(\C),\Z(d-1)) \to \Z/2\Z
\end{align}
which
sends the equivariant cycle class of a reduced curve $Z \subseteq X$ with normalisation~$Z'$ to $\chi(Z',\sO_{Z'})$ modulo~$2$ (see Theorem~\ref{th:phi}).  In particular, the equivariant cycle class
of a geometrically irreducible curve on~$X$ determines its geometric genus modulo~$2$, whereas
the other cycle classes we have mentioned do not.
Considering the image of $\Hdg^{2d-2}_G(X(\C),\Z(d-1))_0$ by this map now provides a possible
obstruction to the existence of a geometrically irreducible curve of even geometric genus
on~$X$,
and, in particular, to the existence of a geometrically rational curve.
We give examples of this in~\textsection\textsection\ref{subsec:topological
examples}--\ref{subsec:hodge-theoretic examples}.
We note that any smooth proper and geometrically irreducible
variety of dimension~$\geq 2$ contains a geometrically irreducible
curve of odd geometric genus (see Proposition~\ref{prop:parity of genus in irrelevant situations});
the existence of a geometrically irreducible curve of even geometric
genus is a property which does not seem to have been considered systematically before
(though see \cite[Example~41, Question~42]{kollarelw}).

Let~$\psi$ denote the restriction
of the map~\eqref{eq:intro psiwithpoint} for $k=d-1$, if $X(\R)\neq\emptyset$,
or of the map~\eqref{eq:intro psiwithoutpoint}, if $X(\R)=\emptyset$,
to the subgroup
\begin{align*}
H^{2d-2}_G(X(\C),\Z(d-1))_0 \subseteq H^{2d-2}_G(X(\C),\Z(d-1))
\end{align*}
cut out by the topological condition which enters the definition of~\eqref{eq:intro hdgG0}
(disregarding the Hodge condition).
To draw consequences of the real integral Hodge conjecture for $1$\nobreakdash-cycles,
one quickly faces the problem of determining the image of~$\psi$.
Using a new result of a purely topological nature established in~\textsection\ref{sec:cohomology}
(see Theorem~\ref{th:selfduality}),
we solve it completely in~\textsection\ref{sec:onecycles} (see Theorem~\ref{th:image psi}),
thus providing an answer to a question of van~Hamel
(see \cite[p.~93]{vanhamelthese}, where
the map~$\rho_1$ is our~\eqref{eq:intro psiwithpoint} for $k=d-1$).

When the $2$\nobreakdash-torsion subgroup of $\Pic(X_\C)$ is trivial
(\emph{e.g.}, when $X(\C)$ is simply connected), the map~$\psi$ is surjective and we show that its
kernel consists of norms from~$\C$ to~$\R$ of classes in $H^{2d-2}(X(\C),\Z)$.
A direct relation follows,
in this case, between four motifs:
the real integral Hodge conjecture, the complex
integral Hodge conjecture, the surjectivity of the Borel--Haefliger cycle class map and the existence of geometrically irreducible curves of even geometric genus (see Theorem~\ref{thm:relation ihc phi}).

Two theorems that we obtain as consequences of these results---in conjunction, in the case of Theorem~\ref{th:B},
with Voisin's theorem according to which complex uniruled or Calabi--Yau
threefolds satisfy the integral Hodge
conjecture (see~\cite{voisinthreefolds})---are the following.
If~$M$ is an abelian group, we denote by $M[2^\infty]$ its $2$\nobreakdash-primary torsion subgroup.

\begin{thm}[see Theorem~\ref{th:nohodgetheoreticob}]
\label{th:A}
Let~$X$ be a smooth, proper and geometrically irreducible real variety,
of dimension~$d\geq 1$.
Assume that~$X$ satisfies the real integral Hodge conjecture for $1$\nobreakdash-cycles
and that $H^2(X,\sO_X)=0$.
\begin{enumerate}[(i)]
\item
The subgroup $H^{d-1}_\alg(X(\R),\Z/2\Z) \subseteq H^{d-1}(X(\R),\Z/2\Z)$
is the exact orthogonal complement,
under the Poincar\'e duality pairing, of the image of $\Pic(X)[2^\infty]$
by the Borel--Haefliger cycle class map $\Pic(X) \to H^1(X(\R),\Z/2\Z)$.
\item
There exists a geometrically irreducible curve of even geometric
genus in~$X$ if and only if the natural map $\Pic(X)[2^{\infty}]\to
\Pic(X_\C)^G[2^{\infty}]$ is onto.
\end{enumerate}
\end{thm}

Theorem~\ref{th:A} applies, in particular, to surfaces of geometric genus zero.
Even in this case, its conclusions are new, except for~(i) when~$X_\C$ is
a surface of geometric genus zero such that $\Pic(X_\C)[2]=0$ (Silhol, van~Hamel;
see \cite[Th\'eor\`eme~3.7.18]{mangoltelivre}),
an Enriques surface (Mangolte and van Hamel~\cite{mangoltevanhamel}) or
a birationally ruled surface (Kucharz \cite{kucharzalgeq}).
Among the new corollaries of Theorem~\ref{th:A},
we find that
any real Enriques surface contains a geometrically irreducible curve of
even geometric genus
and that any real surface of geometric genus zero such that $H^1(X(\R),\Z/2\Z)\neq 0$
satisfies $H^1_\alg(X(\R),\Z/2\Z)\neq 0$
(see~\textsection\ref{subsec:varieties with h20=0}).

\begin{thm}[see Corollary~\ref{cor:IHC pour solides RC ou CY}]
\label{th:B}
Let $X$ be a smooth and proper real threefold.
Assume that~$X_\C$ is rationally connected or is simply connected Calabi--Yau.
Then the real integral Hodge conjecture for~$X$
is equivalent to the equality
$H^2_\alg(X(\R),\Z/2\Z)=H^2(X(\R),\Z/2\Z)$, if $X(\R)\neq\emptyset$,
or to the existence of a geometrically irreducible curve of even geometric genus on~$X$,
if $X(\R)=\emptyset$.
\end{thm}

As Theorem~\ref{th:B} clearly illustrates,
the existence of a geometrically irreducible curve of even geometric genus
must be considered as the analogue, in the absence of real points,
of the equality $H^{d-1}_\alg(X(\R),\Z/2\Z)=H^{d-1}(X(\R),\Z/2\Z)$.

Over~$\C$, Voisin has proved the integral Hodge conjecture for $1$\nobreakdash-cycles
on uniruled or Calabi--Yau threefolds and, conditionally on the Tate conjecture for surfaces
over finite fields, on rationally connected varieties of any dimension (see \cite{voisinthreefolds},
\cite{voisinremarks}).  The analogy between the real and complex integral Hodge conjectures,
on the one hand,
and the good properties of the real integral Hodge conjecture,
on the other hand,
prompt the following question, which serves as a guiding problem for \cite{bwpartie2}:
if~$X_\C$ is a uniruled threefold, a Calabi--Yau threefold, or a rationally connected variety,
does~$X$ satisfy the real integral Hodge conjecture for $1$\nobreakdash-cycles?

By Theorem~\ref{th:A}, a positive answer would imply that
rationally connected varieties (of positive dimension) over~$\R$
satisfy $H^{d-1}_\alg(X(\R),\Z/2\Z)=H^{d-1}(X(\R),\Z/2\Z)$
and contain geometrically irreducible curves of even geometric genus.
A conjecture of Kollár predicts that such varieties should even contain
geometrically rational curves (see \cite[Remarks~20]{araujokollar}, \cite[Question~42]{kollarelw}).

In \cite{bwpartie2}, we provide evidence towards a positive answer to the above question
by establishing the real integral Hodge conjecture for $1$\nobreakdash-cycles on~$X$
under any of the following assumptions:
\begin{enumerate}
\item $X$ is a conic bundle over a variety which itself satisfies the real integral Hodge conjecture for $1$\nobreakdash-cycles (\emph{e.g.}, $X$ can be any conic bundle threefold);
\item $X$ is a Fano threefold with no real point;
\item $X$ is a threefold fibred over a curve into del~Pezzo surfaces of degree~$\delta$,
when $\delta \notin \{1,2,4\}$, as well as in some cases for which $\delta\in\{1,2,4\}$
(see \cite{bwpartie2} for precise statements).
\end{enumerate}
In view of Theorem~\ref{th:A}, these results have concrete consequences.
Among them:
\begin{enumerate}[(i)]
\item the existence of a geometrically irreducible curve of even geometric genus
in any smooth real quartic threefold;
\item the equality $H^2_\alg(X(\R),\Z/2\Z)=H^2(X(\R),\Z/2\Z)$
when~$X$ is the total space of a fibration into cubic surfaces over a real
curve~$B$ such that $B(\R)$ is connected.
\end{enumerate}

The determination of the subgroup $H^k_\alg(X(\R),\Z/2\Z)$ also has
consequences on the problem of~$\ci$ approximation of submanifolds of~$X(\R)$
by algebraic subvarieties.  As an example, combining~(ii) with
a result of Akbulut and King~\cite{AkbulutKing}
(see also~\cite{bochnakkucharz} and \cite[\textsection\ref*{BW2-subsec:algebraic approximation}]{bwpartie2})
shows that for any threefold~$X$ as in~(ii), any $\ci$ loop in~$X(\R)$
can be approximated arbitrarily well by the
real locus of an algebraic curve.

Finally, we examine, in~\textsection\ref{sec:blochogus}, the implications of the real integral
Hodge conjecture for the study of torsion $1$\nobreakdash-cycles on real varieties.
We obtain, in particular, the following theorem.
Its proof relies, on the one hand, on Bloch--Ogus theory, which we develop
further in~\textsection\ref{sec:blochogus} for real varieties, and, on the
other hand, on the topological result already mentioned above
(Theorem~\ref{th:selfduality}).

\begin{thm}[see Corollary~\ref{cor:ch1torsduality}]
\label{th:C}
Let~$X$ be a smooth, proper and geometrically irreducible real variety,
of dimension $d$.
Assume that $\CH_0(X_\C)$ is supported on a surface
(such is the case, for instance, if~$X_\C$ is a uniruled threefold)
and that~$H^2(X,\sO_X)=0$.
If~$X$ satisfies the real integral Hodge conjecture
for $1$\nobreakdash-cycles,
then
the image of $\CH_1(X)[2^\infty]$
by the Borel--Haefliger cycle class map
\begin{align*}
\CH_1(X) \to H^{d-1}(X(\R),\Z/2\Z)
\end{align*}
is the exact orthogonal complement,
under the Poincar\'e duality pairing, of
the subgroup $H^1_\alg(X(\R),\Z/2\Z) \subseteq H^1(X(\R),\Z/2\Z)$.
\end{thm}

Theorem~\ref{th:A}~(i) and Theorem~\ref{th:C} coincide in the case of surfaces.

For real threefolds~$X$ which satisfy the real integral Hodge conjecture,
Bloch--Ogus theory also enables us to control
the torsion subgroup
of the kernel of the equivariant cycle class map $\CH_1(X) \to H^{2d-2}_G(X(\C),\Z(d-1))$.
As an example, we prove in~\textsection\ref{sec:blochogus}
that for any smooth real quartic threefold~$X$ with no real point,
the abelian group $\CH_1(X)_\tors$
is isomorphic to $\Z/2\Z \oplus (\Q/\Z)^{30}$
(see Proposition~\ref{prop:example quartic threefold torsion};
one can even determine the full structure of~$\CH_1(X)$,
see Remark~\ref{rmks:chow group of quartic threefold}~(i)).
This relies on the real integral Hodge conjecture for such~$X$,
which we establish in~\cite{bwpartie2}.

In this article as well as in~\cite{bwpartie2},
we work over an arbitrary real closed field,
except when we use the specific archimedean properties of~$\R$,
via the Stone--Weierstrass theorem as in \cite[\textsection\ref*{BW2-sec:conicbundles}]{bwpartie2}
or via Hodge theory as in \cite[\textsection\ref*{BW2-sec:Fano}]{bwpartie2}.
The statements of Theorems~\ref{th:A}, \ref{th:B}, \ref{th:C}, established in the present article, remain true in this generality,
\emph{mutatis mutandis}, while some of the results proved in \cite{bwpartie2} do not.
As is well known, and as we recall in \cite[\textsection\ref*{BW2-subsec:curves of bounded degree}]{bwpartie2},
the truth,
over an arbitrary real closed field, of an assertion such as the equality
$H^{d-1}_\alg(X(\R),\Z/2\Z)=H^{d-1}(X(\R),\Z/2\Z)$ or the existence of a
geometrically irreducible curve of even geometric genus in~$X$ is
equivalent to the truth of the same assertion over the reals together with
a bound, in any bounded family of real varieties, on the degree of the
curves whose existence it predicts.  Thus, for instance, it follows from
Theorem~\ref{th:nohodgetheoreticob} (which is Theorem~\ref{th:A} over a real closed field)
that in any bounded family of geometrically rational surfaces~$X$, every~$\ci$ loop
in~$X(\R)$ is homologically equivalent to the real locus of an algebraic curve of bounded degree.
For rationally connected threefolds, however, the same is not true, as
the real integral Hodge conjecture can fail for them over non-archimedean real closed fields
(see \cite[\textsection\ref*{BW2-nonarchimedeansection}]{bwpartie2}).
This is in marked contrast with the situation over algebraically closed fields of characteristic~$0$:
by the Lefschetz principle, Voisin's theorem on the integral
Hodge conjecture for rationally connected threefolds readily extends to such fields.

The text is organised as follows.
We devote~\textsection\ref{sec:cohomology} to the cohomological
tools that are used throughout the
article (both reminders and new results); we refer the reader to
the introduction of~\textsection\ref{sec:cohomology} for more details.
In~\textsection\ref{sec:realIHC},
we formulate the real integral Hodge conjecture, prove a few basic results about 
it (\emph{e.g.}, its birational invariance for cycles of dimension~$1$ or of codimension~$2$)
and raise the question of its validity for rationally connected varieties, uniruled threefolds
and Calabi--Yau threefolds over~$\R$.
We proceed,
in~\textsection\ref{sec:onecycles},
to relate
the real integral Hodge conjecture for $1$\nobreakdash-cycles
to the study of the group $H^{d-1}_\alg(X(\R),\Z/2\Z)$ and of geometrically irreducible
curves of even geometric genus.  This leads us, in particular, to Theorems~\ref{th:A} and~\ref{th:B}.
A number of examples of smooth, proper and geometrically irreducible varieties~$X$ over~$\R$ such that $H^{d-1}_\alg(X(\R),\Z/2\Z)\neq H^{d-1}(X(\R),\Z/2\Z)$ or such that~$X$ does not contain any geometrically irreducible curve of even geometric genus are presented in~\textsection\ref{section:examples}.
Finally, we develop Bloch--Ogus theory and apply the real integral Hodge conjecture
to the study of torsion $1$\nobreakdash-cycles in~\textsection\ref{sec:blochogus},
where we prove, in particular, Theorem~\ref{th:C}.

\bigskip
\emph{Acknowledgements.}
Krasnov and van Hamel were the first to approach algebraic cycles on real
varieties in a systematic way through the study of the cycle class map into the equivariant
integral singular cohomology of the complex locus.
We would like to emphasise the importance of their work to the development of the subject
considered in this article.
In addition, we thank the referee for their careful work and for many suggestions
which helped improve the exposition.
}

\section{Cohomology of real algebraic varieties}
\label{sec:cohomology}

We introduce, in this section, the cohomological tools on which
this article and its sequel~\cite{bwpartie2} heavily depend.
Some of these tools are standard (at least for varieties over the field~$\R$ of real numbers), while some are new.
Let us describe the organisation of~\textsection\ref{sec:cohomology}.

For lack of an adequate reference to the existing literature,
we first recall, in~\textsection\ref{subsec:sheaf cohomology},
the formalism of sheaf cohomology and equivariant
sheaf cohomology for algebraic varieties over a real closed field,
together with some of the standard properties that we shall use throughout:
the two spectral sequences of equivariant cohomology;
Poincar\'e duality \`a la Verdier; purity, equivariant
purity; covariant functoriality.
Over the field of real numbers,
these topics are discussed in \cite[Chapter~II and Chapter~III]{vanhamelthese}.
Over a real closed field,
one has to replace singular cohomology with semi-algebraic cohomology,
a theory first developed by Delfs and Knebusch (see \cite{delfsknebuschonthehomology}, \cite{delfshomology}).

Letting $G=\Z/2\Z$, the $G$\nobreakdash-equivariant cohomology groups of
a space endowed with the trivial action of~$G$, with coefficients in~$\Z/2\Z$, in~$\Z$,
or in the twisted integers~$\Z(1)$, canonically decompose as direct sums of non-equivariant cohomology groups.
These decompositions appear in~\cite{kahnchern}, \cite{krasnovequivariant},
\cite[Chapter~III, \textsection\textsection6--7]{vanhamelthese}.
They play an essential role in the formulation of the real integral Hodge conjecture
(see~\textsection\ref{subsubsec:topological constraints} and~\textsection\ref{par:realIHC}).
We discuss them in~\textsection\ref{subsec:canonical decompositions},
in the setting of semi-algebraic cohomology and with locally constant sheaves as coefficients.

In~\textsection\ref{subsec:on the normal bundle}, we show
that in the case of the complex locus of a smooth real algebraic variety with
support in the real locus,
the long exact sequence of equivariant cohomology with support
decomposes into canonically split short exact sequences,
when one works with $\Z/2\Z$ coefficients (Proposition~\ref{prop:short exact sequences})
or with appropriately twisted integer coefficients (Proposition~\ref{prop:short exact sequences integral}).
Proposition~\ref{prop:short exact sequences} is an improvement on
 \cite[\textsection2.2]{vanhamelabeljacobi};
Proposition~\ref{prop:short exact sequences integral} and its
companion Proposition~\ref{prop:truncated projection integral coeff},
however,
seem to be entirely new.
The results of~\textsection\ref{subsec:on the normal bundle}
are used in~\textsection\ref{subsec:two dualities}, in~\textsection\ref{sec:blochogus}
and in \cite[proof of Theorem~\ref*{BW2-thm:fibres en coniques}, Step~\ref*{BW2-step:notrace}]{bwpartie2}.

The goal of~\textsection\ref{subsec:two dualities} is to formulate and prove
Theorem~\ref{th:selfduality}, a duality result which combines,
for any smooth and proper real variety, Poincar\'e duality for the real
locus with Lefschetz duality for the complement of the real locus in the
complex locus.
Theorem~\ref{th:selfduality} is new, and is key to the proofs
of Theorem~\ref{th:image psi}
and Theorem~\ref{th:ch1torsion}.

In~\textsection\ref{subsec:real lefschetz},
we establish
a Lefschetz hyperplane theorem for equivariant cohomology with twisted integral
coefficients
(Proposition~\ref{prop:weak lefschetz surjectivity}).
This will be used in \cite[proof of Theorem~\ref*{BW2-thm:solides fano}, \textsection\ref*{BW2-par:systeme anticanonique}]{bwpartie2}.

Finally, we devote~\textsection\ref{subsec:topologicalconstraints} to
the equivariant cycle class map associated with a smooth real algebraic variety,
and to the topological constraint
discovered by
Kahn~\cite{kahnchern} and Krasnov~\cite{krasnovequivariant}
that all algebraic cycle classes must satisfy (see Theorem~\ref{th:conditions de krasnov}).
A new result here is Theorem~\ref{th:stability of topological constraints},
which asserts the compatibility of this topological constraint with proper push-forwards.
Its proof rests on the contents of~\textsection\ref{subsec:on the normal bundle}
and on a relative version of Wu's theorem due to Atiyah and Hirzebruch.
Theorem~\ref{th:stability of topological constraints} is used
in~\textsection\ref{subsubsec:birational invariance} and in
\cite[proofs of Theorem~\ref*{BW2-thm:fibres en coniques} and Theorem~\ref*{BW2-thm:solides fibres en del Pezzo}~(iv)]{bwpartie2}.

In the whole article,
in an effort to keep the notation as tidy as possible
(especially in the equivariant setting),
we stick to cohomology
and do not introduce Borel--Moore homology.
For smooth equidimensional varieties, this makes no difference:
if~$\sF$ is a locally constant sheaf of abelian groups on an equidimensional
semi-algebraic space~$V$
subject to the assumptions of~\textsection\ref{subsubsec:verdier} below,
one could regard $H_i^{\BM}(V,\sF)$ as shorthand for $H^{\dim(V)-i}(V,\sF \otimes_\Z \orient_V)$,
where~$\orient_V$ denotes the orientation sheaf of~$V$
(see~\cite[Chapter~III, \textsection9, Theorem~9.3]{delfshomology}).
Some of the results below
(\emph{e.g.}, Theorem~\ref{th:phi})
would extend to singular varieties if one replaced
cohomology with Borel--Moore homology.

\subsection{Sheaf cohomology over real closed fields}
\label{subsec:sheaf cohomology}

We fix a real closed field~$R$
and set $C=R(\sqrt{-1}\mkern2mu)$,
$G=\Gal(C/R)$
and $\Z(j)=(\sqrt{-1}\mkern2mu)^j\Z \subset C$
for $j\in \Z$.
Thus~$C$ is algebraically closed,
the group~$G$ has order~$2$ and~$\Z(j)$ is canonically isomorphic, as a $G$\nobreakdash-module,
to~$\Z$ or to~$\Z(1)$, depending on the parity of~$j$.
If~$M$ is a $G$\nobreakdash-module,
we let $M(j)=M \otimes_\Z \Z(j)$ and $M[G]=M\otimes_\Z \Z[G]$; these $G$\nobreakdash-modules fit into canonical short exact sequences
\begin{align}
\label{eq:real-complex sequence 01}
0 \to M \to M[G] \to M(1) \to 0
\end{align}
and
\begin{align}
\label{eq:real-complex sequence 10}
0 \to M(1) \to M[G] \to M \to 0\rlap{\text{,}}
\end{align}
which we shall refer to as the \emph{real-complex exact sequences}.
We denote the field of real (resp.~complex) numbers by~$\R$
(resp.~$\C$).
We use the term \emph{variety} (over~$R$) as a synonym for \emph{separated scheme of finite type}
(over~$R$).

For the whole of~\textsection\ref{sec:cohomology}, we fix a variety~$X$ over~$R$.

\subsubsection{Semi-algebraic spaces and their cohomology}
\label{subsubsec:sacohomology}

The set~$X(R)$ of rational points of~$X$ is
a locally complete (hence affine) semi-algebraic space
in the sense of Delfs and Knebusch (see
\cite{delfsknebuschbasictheory1},
\cite{delfsknebuschbasictheory2},
\cite{delfsknebuschsurvey}, \cite{robson}, \cite[Chapter~I, Example~7.1]{delfsknebuschbook},
\cite{delfshomology}).
We shall always consider it as such;
thus, a \emph{sheaf of abelian groups on~$X(R)$} will refer to a sheaf on the semi-algebraic site of~$X(R)$
 (see \cite[\textsection1, \textsection5]{delfsknebuschsurvey}).
We denote by $H^*(X(R),\sF)$ the cohomology groups of
such a sheaf.
If~$\sF$ is the constant sheaf associated with~$M$, we simply write $H^*(X(R),M)$.
The semi-algebraic cohomology groups $H^*(X(R),M)$ are finitely generated if~$M$ is a finitely generated abelian
group;
when~$R=\R$,
they coincide with
the singular cohomology of the naive topological space~$X(\R)$,
with coefficients in~$M$
(see~\cite[Proposition~6.2]{delfsknebuschsurvey}).

We also view~$X(C)$ as a locally complete
semi-algebraic space, namely as the space of $R$\nobreakdash-points
of the Weil restriction from~$C$ to~$R$ of $X_C = X\otimes_R C$.
(To be precise, the Weil restriction makes sense when~$X$ is quasi-projective.
In general, one proceeds by choosing an affine open cover and then gluing;
see \cite[Lemma~5.6.1]{scheiderer}.)
This semi-algebraic space carries a natural action of~$G$.
If~$\sF$ is a $G$\nobreakdash-equivariant sheaf of abelian groups on~$X(C)$, we denote by $H^*_G(X(C),\sF)$ its
equivariant cohomology groups
(see~\cite[(6.1.3)]{scheiderer}).
We simply write $H^*_G(X(C),M)$
if~$\sF$ is the constant sheaf associated with a $G$\nobreakdash-module~$M$.
When~$R=\R$,
this coincides with equivariant Betti cohomology of~$X(\C)$ with coefficients in~$M$.
Finally, we recall that when~$M$ is torsion,
there are canonical
isomorphisms
$H^*(X(C),M)=H^*_\et(X_C,M)$
and
$H^*_G(X(C),M)=H^*_\et(X,M)$
(see~\cite[Corollary~15.3.1]{scheiderer}).
It follows, as these groups are finite when~$M$ is finite
and as the $G$\nobreakdash-module $\Q/\Z(1)$ is canonically isomorphic
to the group of roots
of unity of~$C$,
that there are canonical isomorphisms
$H^*_G(X(C),\Z(j)) \otimes_\Z \Zl = H^*_\et(X,\Zl(j))$
for all primes~$\ell$ and all~$j$.

More generally,
for any $G$\nobreakdash-invariant locally closed semi-algebraic subset
$V \subseteq X(C)$,
any $G$\nobreakdash-equivariant sheaf of abelian groups~$\sF$ on~$V$,
any integer~$i$
and any
$G$\nobreakdash-invariant closed semi-algebraic subset
$Z \subseteq V$,
we denote by
$H^i_{G,Z}(V,\sF)$ the value on~$\sF$ of the $i$th right derived functor
of the functor of $G$\nobreakdash-invariant global sections supported on~$Z$.
Let $\pi:V \to V/G$ denote the quotient map
(see \cite[Corollary~1.6]{brumfielquotient}).
We denote by $\sH^i(G,\sF)$
the sheaf, on~$V/G$, defined as the sheafification
of the presheaf $U \mapsto H^i(G,\sF(\pi^{-1}(U)))$.
We recall that
the two spectral sequences of equivariant cohomology with support in~$Z$ take the shape
\begin{align}
\label{eq:first spectral sequence}
E_2^{p,q}=H^p_{Z/G}(V/G,\sH^q(G,\sF)) \Rightarrow H^{p+q}_{G,Z}(V,\sF)
\end{align}
and
\begin{align}
\label{eq:hochschild-serre}
E_2^{p,q}=H^p(G,H^q_Z(V,\sF)) \Rightarrow H^{p+q}_{G,Z}(V,\sF)
\end{align}
(see \cite[Th\'eor\`eme~5.2.1]{tohoku});
the latter is
the \emph{Hochschild--Serre spectral sequence}.

Viewing $\sF[G]=\sF\otimes_\Z \Z[G]$
as a $G$\nobreakdash-equivariant sheaf with the diagonal action of~$G$,
let us now consider the spectral sequence~\eqref{eq:hochschild-serre}
associated with~$\sF[G]$
(rather than with~$\sF$).
As $H^q_Z(V,\sF[G])=H^q_Z(V,\sF)[G]$, we have $E_2^{p,q}=0$ for $p>0$ and $E_2^{0,q}=H^q_Z(V,\sF)$
(see \cite[Chapter~III, Corollary~5.7, Proposition~5.9, Corollary~6.6]{brown}),
hence
\begin{align}
\label{eq:cohoeqnoneq}
H^i_{G,Z}(V,\sF[G])=H^i_Z(V,\sF)
\end{align}
for all~$i$.
Setting $\sF(1)=\sF \otimes_\Z \Z(1)$,
the sequences~\eqref{eq:real-complex sequence 01} and~\eqref{eq:real-complex sequence 10} with $M=\Z$, tensored by~$\sF$, therefore
induce long exact sequences
\begin{align}
\label{eq:real-complex long 01}
\cdots \to H^i_{G,Z}(V,\sF) \to H^i_Z(V,\sF) \to H^i_{G,Z}(V,\sF(1)) \to H^{i+1}_{G,Z}(V,\sF) \to \cdots
\end{align}
and
\begin{align}
\label{eq:real-complex long 10}
\cdots \to H^i_{G,Z}(V,\sF(1)) \to H^i_Z(V,\sF) \to H^i_{G,Z}(V,\sF) \to H^{i+1}_{G,Z}(V,\sF(1)) \to \cdots
\end{align}
for any $G$\nobreakdash-equivariant sheaf of abelian groups~$\sF$ on~$V$.
We will refer to the map $H^i_Z(V,\sF)\to H^i_{G,Z}(V,\sF)$
appearing in~\eqref{eq:real-complex long 10} as the \emph{norm map}.

\subsubsection{Notation: the class $\omega$ and its variants}
\label{subsubsec:omega}

For any integer $i \geq 1$,
we shall denote by $\omega^i_V \in H^i_G(V,\Z(i))$ the image
of $1\in \Z/2\Z=H^i(G,\Z(i))=H^i_G(\mathrm{pt},\Z(i))$
by pull-back with respect to the map from~$V$ to the point,
and by $\omega^i_{V,\Z/2\Z} \in H^i_G(V,\Z/2\Z)$ the image of~$\omega^i_V$
by the map $H^i_G(V,\Z(i))\to H^i_G(V,\Z/2\Z)$ induced by the surjection $\Z(i)\to \Z/2\Z$.
The subscript~$V$ will be omitted when no confusion can arise.
We shall write $\omega$ (resp., $\omega_{\Z/2\Z}$) for $\omega^1$ (resp., $\omega^1_{\Z/2\Z}$).
The notation is justified by the remark that $\omega^i$ (resp.,
$\omega^i_{\Z/2\Z}$) coincides with the $i$\nobreakdash-fold cup product
of~$\omega$ (resp., $\omega_{\Z/2\Z}$) with itself.
Finally, we note that the map $H^i_{G,Z}(V,\sF) \to H^{i+1}_{G,Z}(V,\sF(1))$
which appears in~\eqref{eq:real-complex long 10}
can be interpreted as the cup product with~$\omega$
(see \cite[\textsection A3]{kahndeuxtheoremes}).

\subsubsection{Cohomological dimension}
\label{subsubsec:cohomological dimension}

For a $G$\nobreakdash-module~$M$,
we shall consider the relative equivariant cohomology groups $H^*_G(X(C),X(R),M)$,
defined as $H^*_G(X(C),j_!M)$, where $j:X(C) \setminus X(R) \hookrightarrow X(C)$ denotes the inclusion and~$j_!M$ is the extension by zero.
These groups fit into the localisation long exact sequence
\begin{align*}
\xymatrix@C=1.3em{
\cdots \ar[r] & H^i_G(X(C),X(R),M) \ar[r] & H^i_G(X(C),M) \ar[r] & H^i_G(X(R),M) \ar[r] & \cdots
}
\end{align*}
(induced by the exact sequence of $G$\nobreakdash-equivariant sheaves
$0 \to j_!M \to M \to \iota_*M \to 0$, where $\iota:X(R) \hookrightarrow X(C)$ denotes the inclusion).
As the action of~$G$ on~$X(C)$ is discontinuous (see \cite[\textsection5.3]{tohoku}),
it follows from the spectral sequence~\eqref{eq:first spectral sequence}
that
\begin{align}
\label{eq:cohomological dimension iso relative}
H^i_G(X(C),X(R),M)=H^i(X(C)/G,X(R),M)
\end{align}
and
\begin{align}
\label{eq:cohomological dimension iso open}
H^i_G(X(C)\setminus X(R),M)=H^i((X(C)\setminus X(R))/G,M)\rlap{\text{,}}
\end{align}
where~$M$ now also denotes the locally constant sheaf $\sH^0(G,M)$ on
$(X(C)\setminus X(R))/G$.
By \cite[Chapter~II, Lemma~9.1]{delfshomology},
these two groups vanish for $i>2\dim(X)$.
In particular, the restriction map $H^i_G(X(C),M) \to H^i_G(X(R),M)$ is an isomorphism
if $i>2\dim(X)$.

\subsubsection{Semi-algebraic Verdier duality and purity}
\label{subsubsec:verdier}

Let $f:V\to W$ denote a continuous semi-algebraic map between locally complete
semi-algebraic spaces over~$R$.
We assume that~$V$ and~$W$ are homology manifolds
(see~\cite[Chapter~III, \textsection3, Definition~1]{delfshomology})
and denote by $\orient_V$, $\orient_W$ their orientation sheaves;
these are locally constant sheaves with stalks isomorphic to~$\Z$.
We let $\orient_{V/W}=\Homrond(f^*\orient_W,\orient_V)$.
The main examples of such manifolds, in this article, will be the spaces $X(R)$, $X(C)$ and $(X(C)\setminus X(R))/G$
for a smooth variety~$X$ over~$R$
(\emph{loc.\ cit.}, Example~3.3).
For any noetherian ring~$\Lambda$ and
any $T \in \{V,W\}$,
we let $D^+(T,\Lambda)$ denote the derived category of bounded below complexes of sheaves of $\Lambda$\nobreakdash-modules on the semi-algebraic site of~$T$.
According to~\cite[Theorem~4.1]{edmundoprelli},
the derived direct image functor with proper support
$\RR f_!:D^+(V,\Lambda) \to D^+(W,\Lambda)$,
defined in \cite[Chapter~II, \textsection8]{delfshomology},
admits a right adjoint
$\RR f^!:D^+(W,\Lambda) \to D^+(V,\Lambda)$.

The following statement combines Poincar\'e duality and a version of the
Thom isomorphism in the context of semi-algebraic spaces over real closed fields.

\begin{prop}
\label{prop:verdier}
For any noetherian ring~$\Lambda$
and any bounded complex~$\sF$ of sheaves of $\Lambda$\nobreakdash-modules on~$W$
whose cohomology sheaves are locally constant, there is a canonical isomorphism
\begin{align}
\label{eq:verdier}
\RR f^!\sF = \big(\orient_{V/W}\otimes_\Z f^* \sF\big)[\dim(V)-\dim(W)]
\end{align}
in $D^+(V,\Lambda)$.
\end{prop}

\begin{proof}
Applying \cite[Theorem~4.10]{edmundoprelli} to~$V$ and to~$W$
yields a canonical isomorphism
\begin{align}
\RR f^! \orient_W = \orient_V[\dim(V)-\dim(W)]
\end{align}
in $D^+(V,\Z)$.
On the other hand,
there is
a canonical morphism
\begin{align}
\RR f^!\orient_W \otimes^\LL_\Z f^*\Homrond(\orient_W,\sF) \to \RR f^!\sF
\end{align}
(see \cite[Proposition~3.1.11]{kashiwaraschapira}).
It suffices to prove that the latter is an isomorphism.
By the triangulated five lemma, we may assume that~$\sF$ is concentrated in degree~$0$.
After shrinking~$W$, we may assume that $\sF=M \otimes_\Z \orient_W$
for some $\Lambda$\nobreakdash-module~$M$.
The $\Lambda$\nobreakdash-module structure is now irrelevant;
we may therefore assume that $\Lambda=\Z$.
The question being compatible with filtered direct limits, we may also
assume that~$M$ is a finitely generated abelian group, and then that $M=\Z$
or that $M=\Z/N\Z$ for some $N\geq 1$.
In this case, the assertion follows from \cite[Theorem~4.10]{edmundoprelli} applied four
times (to~$V$ and to~$W$, with coefficients~$\Z$ and~$\Z/N\Z$).
\end{proof}

When~$f$ is a closed embedding of pure codimension~$c$,
the isomorphism~\eqref{eq:verdier} with $\sF=\Lambda=\Z/2\Z$
induces a canonical isomorphism
\begin{align}
\label{eq:purity mod 2}
H^{i-c}(V,\Z/2\Z)=H^i_V(W,\Z/2\Z)
\end{align}
for any $i\in\Z$
(see \cite[Proposition~3.1.12]{kashiwaraschapira}), and hence,
by forgetting the support, a Gysin map
$H^{i-c}(V,\Z/2\Z) \to H^i(W,\Z/2\Z)$.
By definition, the
\emph{fundamental class}
$s_{V/W} \in H^c_V(W,\Z/2\Z)$
of~$V$ in~$W$ is the image
of the constant section $1 \in H^0(V,\Z/2\Z)$
by the isomorphism~\eqref{eq:purity mod 2}
for $i=c$.

More generally, if~$f$ is a proper semi-algebraic map
(see \cite[Chapter~II, Remark~7.6]{delfshomology})
and if we let $c=\dim(W)-\dim(V)$,
the isomorphism~\eqref{eq:verdier}
yields, by adjunction,
a canonical morphism
$\RR f_*(\orient_{V/W}\otimes_\Z f^* \sF) \to \sF[c]$,
since $\RR f_*=\RR f_!$.
This morphism induces, in turn, a push-forward homomorphism
\begin{align}
\label{eq:pushforward}
f_*:H^{i-c}(V,\orient_{V/W}\otimes_\Z f^* \sF) \to H^i(W,\sF)
\end{align}
for any $i\in\Z$.  Pull-back and proper push-forward are related by the formula
\begin{align}
\label{eq:projection formula}
f_*(\alpha \smile f^*\beta)=f_*\alpha\smile\beta\rlap{,}
\end{align}
valid
in $H^{i+j+c}(W,\sF \otimes^\LL_\Z \sG)$
for all $\alpha \in H^i(V,\orient_{V/W}\otimes_\Z f^* \sF)$, $\beta \in H^j(W,\sG)$, all~$i$, $j$,
and all bounded complexes of sheaves of abelian groups~$\sF$ and~$\sG$ on~$W$ with locally constant
cohomology sheaves (and therefore also in
$H^{i+j+c}(W,\sF \otimes_\Z \sG)$ if~$\sF$ and~$\sG$ are themselves just locally constant sheaves).
To prove~\eqref{eq:projection formula}, we note
that the map $\beta \mapsto f_*(\alpha \smile f^*\beta)$
is induced by a morphism $\sG \to \sF \otimes^\LL_\Z \sG[i+c]$ in $D^+(W,\Z)$
that depends functorially on~$\sG$.
In view of the general fact stated below
(which could also be taken as a definition for the cup product), this map can be interpreted, for all~$j$ and all~$\sG$, as the cup product with the class in $H^{i+c}(W,\sF)$ obtained by taking $j=0$, $\sG=\Z$, $\beta=1$, \emph{i.e.}, with the class~$f_*\alpha$.

\begin{fact}
\label{fact:cupproduct}
Let~$\sA_1$, $\sA_2$ be bounded above complexes of sheaves of abelian groups on~$W$
and let $x_1 \in H^0(W,\sA_1)$,
$x_2 \in H^0(W,\sA_2)$.
We set $\sA_3 = \sA_1 \otimes^\LL_\Z \sA_2$
and $x_3=x_1\smile x_2 \in H^0(W,\sA_3)$.
Letting $\phi_i:\Z \to \sA_i$ denote the morphism in $D^-(W,\Z)$
corresponding to~$x_i$,
we have $\phi_3 = (\mathrm{Id}_{\sA_1} \otimes \phi_2) \circ \phi_1$.
\end{fact}

As explained in
\cite[\textsection3.2.6]{delignedualite},
when~$V$ has pure dimension~$d$ and~$W$ is a point,
the isomorphism~\eqref{eq:verdier}
with $\sF=\Lambda=\Z/N\Z$
for a divisible enough $N \geq 1$
induces, for any
locally constant sheaf~$\sG$ of abelian groups of finite exponent on~$V$, a perfect Poincar\'e duality pairing
\begin{align}
\label{eq:semi-algebraic poincare duality}
H^i_c(V,\sG) \times H^{d-i}(V,\sG\mkern3mu\check{ }\otimes_\Z\orient_V) \to \Q/\Z
\end{align}
for every $i\in \Z$, where
 $H^i_c(V,\sG)$ denotes semi-algebraic cohomology with complete
supports (see \cite[Chapter~II, \textsection1]{delfshomology})
and $\sG\mkern3mu\check{ }=\Homrond(\sG,\Q/\Z)$.

\subsubsection{Equivariant purity}
\label{subsubsec:equivariant purity}

We keep the notation of~\textsection\ref{subsubsec:verdier}
and assume, in addition, that the group $G=\Gal(C/R)$ acts on~$V$ and~$W$ and that~$f$ is $G$\nobreakdash-equivariant; thus, the sheaves~$\orient_V$, $\orient_W$, and $\orient_{V/W}$ are $G$\nobreakdash-equivariant.
Using a finite-dimensional approximation of the Borel construction, we now transfer the results of~\textsection\ref{subsubsec:verdier}
to the setting
of equivariant cohomology.
Let~$G$ act on the
 $n$\nobreakdash-dimensional semi-algebraic unit sphere $S^n \subset R^{n+1}$
by the antipodal involution. For $T \in \{V,W\}$, the diagonal action of~$G$ on $T \times S^n$ is free
and discontinuous, hence the quotient semi-algebraic space $(T\times S^n)/G$ is again a homology manifold.
Moreover, if $p:T \times S^n \to T$
denotes the first projection,
then
for any $G$\nobreakdash-equivariant sheaf~$\sF$ of abelian groups on~$T$, the $G$\nobreakdash-equivariant
sheaf $p^*\sF$ uniquely descends to a sheaf~$\sF'$
on $(T\times S^n)/G$ and for any $G$\nobreakdash-invariant closed semi-algebraic subset $Z \subseteq T$,
there are canonical isomorphisms
\begin{align}
\label{eq:finite dimensional borel}
H^i_{G,Z}(T,\sF)=H^i_{G,Z \times S^n}(T \times S^n,p^*\sF)=H^i_{(Z \times S^n)/G}((T \times S^n)/G,\sF')
\end{align}
for every $i<n$, as follows from~\eqref{eq:first spectral sequence}
and from the Leray spectral sequence for~$p$.
If~$\sF$ is a locally constant $G$\nobreakdash-equivariant sheaf of abelian groups (by which we mean that it is
a $G$\nobreakdash-equivariant sheaf of abelian groups which, as a sheaf of abelian groups,
is locally constant), then~$\sF'$ is locally constant.  We note, however, that~$\sF$
being a constant sheaf does not imply that~$\sF'$ is constant.

As a consequence, when~$f$ is a closed embedding of pure codimension~$c$,
applying Proposition~\ref{prop:verdier} to the inclusion of $(V \times S^n)/G$ in $(W \times S^n)/G$
for a large enough~$n$ yields
a canonical isomorphism
\begin{align}
\label{eq:equivariant purity}
H^{i-c}_G(V,\orient_{V/W} \otimes_\Z f^*\sF)=H^i_{G,V}(W,\sF)
\end{align}
for any $i \in \Z$
and any locally constant $G$\nobreakdash-equivariant sheaf~$\sF$ of abelian groups on~$W$.
Similarly, when~$f$ is a proper semi-algebraic map,
considering~\eqref{eq:pushforward}
for the
map $(V \times S^n)/G \to (W \times S^n)/G$ induced by~$f$
for a large enough~$n$ leads to
a push-forward homomorphism
\begin{align}
\label{eq:pushforward equivariant}
f_*:H^{i-c}_G(V,\orient_{V/W}\otimes_\Z f^* \sF) \to H^i_G(W,\sF)
\end{align}
for any $i \in \Z$ and any locally constant
$G$\nobreakdash-equivariant sheaf~$\sF$ of abelian groups on~$W$,
where $c=\dim(W)-\dim(V)$.

In particular,
if~$X$ is a smooth variety over~$R$, purely of dimension~$d$,
we obtain a canonical isomorphism
\begin{align}
\label{eq:equivariant purity example}
H^{i-d}_G(X(R),\Z/2\Z)=H^i_{G,X(R)}(X(C),\Z/2\Z)
\end{align}
for any $i \in \Z$
by taking $\sF=\Lambda=\Z/2\Z$ and $V=X(R)$, $W=X(C)$.
The
\emph{equivariant fundamental class}
$s_{G,X(R)/X(C)} \in H^d_{G,X(R)}(X(C),\Z/2\Z)$
of~$X(R)$ in~$X(C)$ is the image
of the constant section $1 \in H^0_G(X(R),\Z/2\Z)$
by the isomorphism~\eqref{eq:equivariant purity example} for $i=d$.

We note that
for any smooth variety~$X$ of pure dimension~$d$ over~$R$,
 there is a canonical
 isomorphism
of $G$\nobreakdash-equivariant sheaves
  $\orient_{X(C)}=\Z(d)$
(see \cite[Chapter~IV, \textsection1, Example~1.7]{delfshomology}).
As a consequence,
for any smooth variety~$X$ over~$R$,
any smooth subvariety $Y \subseteq X$ of pure codimension~$k$
and any $G$\nobreakdash-module~$M$,
we find that $\orient_{Y(C)/X(C)}=\Z(-k)$ and
we obtain
a canonical isomorphism
\begin{align}
\label{eq:equivariant purity subvariety}
H^{i-2k}_G(Y(C),M(-k))=H^i_{G,Y(C)}(X(C),M)
\end{align}
for any $i \in \Z$,
by taking $\Lambda=\Z$, $\sF=M$, $V=Y(C)$, $W=X(C)$
in~\eqref{eq:equivariant purity};
and hence, by forgetting the support, also a Gysin map
$H^{i-2k}_G(Y(C),M(-k)) \to H^i_G(X(C),M)$.

More generally, for any proper morphism $f:Y\to X$
of smooth equidimensional varieties over~$R$
and for any $G$\nobreakdash-module~$M$,
if we let $k=\dim(X)-\dim(Y)$,
we obtain, in view of~\eqref{eq:pushforward equivariant},
a push-forward homomorphism
\begin{align}
\label{eq:pushforward equivariant complex}
f_*:H^{i-2k}_G(Y(C),M(-k)) \to H^i_G(X(C),M)
\end{align}
for any $i\in \Z$, since $\orient_{Y(C)/X(C)}=\Z(-k)$.
The projection formula
\begin{align}
\label{eq:projection formula equivariant}
f_*(\alpha \smile f^*\beta)=f_*\alpha\smile\beta
\end{align}
now holds in $H^{i+j+2k}_G(X(C),(M \otimes_\Z M')(k))$
for all $G$\nobreakdash-modules~$M$ and~$M'$,
all $\alpha \in H^i_G(Y(C),M)$, $\beta \in H^j_G(X(C),M')$,
and all~$i$, $j$ (see~\eqref{eq:projection formula}).

\subsection{Canonical decompositions}
\label{subsec:canonical decompositions}

Let~$\sF$ be a locally constant sheaf of abelian groups
on~$X(R)$, with stalks
isomorphic to~$\Z$,
viewed as a $G$\nobreakdash-equivariant sheaf with the trivial action.
For any $G$\nobreakdash-module~$M$ and any $i \geq 0$,
as~$G$ acts trivially on~$X(R)$, there is a canonical isomorphism
\begin{align}
\label{eq:general canonical decomposition}
H^i_G(X(R),\sF \otimes_\Z M)=H^i(X(R), \sF \otimes^\LL_\Z \RGamma(G,M))\rlap{\text{.}}
\end{align}
Indeed, for any $G$\nobreakdash-equivariant sheaf of abelian groups~$\sG$
on~$X(R)$, there is a canonical isomorphism $H^i_G(X(R),\sG)=H^i(X(R),\RR\sH^0(G,\sG))$,
where~$\RR\sH^0(G,-)$ denotes the total right derived functor of the
functor~$\sH^0(G,-)$ introduced just before~\eqref{eq:first spectral sequence};
as
$\RR\sH^0(G,\sF \otimes_\Z M)=\sF \otimes^\LL_\Z \RGamma(G,M)$,
this yields~\eqref{eq:general canonical decomposition}.

\subsubsection{With $\Z/2\Z$ coefficients}
\label{subsubsec:canonical decomposition mod 2}

Let us apply~\eqref{eq:general canonical decomposition}
to $M=\Z/2\Z$.
As $\Z/2\Z$ is a field, there exists, in
the derived
category of $\Z/2\Z$\nobreakdash-modules,
a canonical isomorphism
\begin{align}
\label{eq:derived decomposition modulo 2}
\RGamma(G,\Z/2\Z)=\bigoplus_{q \geq 0} H^q(G,\Z/2\Z)[-q]\rlap{\text{.}}
\end{align}
(There is, in fact, a unique
isomorphism inducing the identity on the cohomology groups.)
As $H^q(G,\Z/2\Z)=\Z/2\Z$ for all $q\geq 0$, a canonical decomposition
\begin{align}
\label{eq:canonical decomposition mod 2}
H^i_G(X(R),\Z/2\Z) = \bigoplus_{0\leq p\leq i} H^p(X(R), \Z/2\Z)
\end{align}
results, for any~$i$.
We note that as the category of $\Z/2\Z$\nobreakdash-modules is semisimple,
the derived cup product map
$\RGamma(G,\Z/2\Z) \otimes^\LL_{\Z/2\Z} \RGamma(G,\Z/2\Z) \to \RGamma(G,\Z/2\Z)$
has to coincide, via~\eqref{eq:derived decomposition modulo 2},
with the direct sum, over all~$r$, of the cup product maps
\begin{align}
\bigoplus_{p+q=r} H^p(G,\Z/2\Z) \otimes_{\Z/2\Z} H^q(G,\Z/2\Z) \to H^r(G,\Z/2\Z)
\end{align}
shifted by~$-r$.
The cup product of $x\in H^k_G(X(R),\Z/2\Z)$, $y \in H^\ell_G(X(R),\Z/2\Z)$
in $H^{k+\ell}_G(X(R),\Z/2\Z)$ can therefore be written,
in terms of~\eqref{eq:canonical decomposition mod 2},
as
\begin{align}
\label{eq:cup product}
x\smile y = \bigg(\sum_{p+q=r}x_p \smile y_q\bigg)_{0 \leq r \leq k+\ell}
\end{align}
if $x=(x_p)_{0\leq p\leq k}$ and $y=(y_q)_{0\leq q\leq \ell}$.

As an alternative way to obtain~\eqref{eq:canonical decomposition mod 2},
one can identify $H^i_G(X(R),\Z/2\Z)$ with $H^i(X(R)\times (S^n/G),\Z/2\Z)$
for some $n>i$, via the Borel construction
(see~\textsection\ref{subsubsec:equivariant purity}),
and then apply the K\"unneth formula.
This leads to the same decomposition,
as we have $\tau_{\leq i}\RGamma(G,\Z/2\Z)=\tau_{\leq i}\RR\pi_*\Z/2\Z$
in $D^+(X(R),\Z/2\Z)$
if $\pi:X(R) \times (S^n/G) \to X(R)$ denotes the first projection.

\subsubsection{With integral coefficients}
\label{subsubsec:decomposition with integral coeff}

Let us now apply~\eqref{eq:general canonical decomposition} to $M=\Z(j)$ for $j \in \Z$.
In the derived category of abelian groups,
there exists a canonical
isomorphism
\begin{align}
\label{eq:decomposition of RGamma integral coeff}
\RGamma(G,\Z(j))=\bigoplus_{q \geq 0} H^q(G,\Z(j))[-q]\rlap{\text{,}}
\end{align}
as can be seen
from the standard explicit complex representing the total cohomology of the cyclic group~$G$.
(There is, in fact, a unique
isomorphism inducing the identity on the cohomology groups.)
Letting $\sF(j) = \sF \otimes_\Z \Z(j)$, a canonical decomposition
\begin{align}
\label{eq:canonical decomposition}
H^i_G(X(R),\sF(j)) = \bigoplus_{p+q=i} H^p(X(R), \sF \otimes_\Z H^q(G,\Z(j)))
\end{align}
results, for any~$i$.  We note that $H^q(G,\Z(j))=0$ if $\pascongru{q}{j}{2}$,
$H^q(G,\Z(j))=\Z/2\Z$
if $\congru{q}{j}{2}$ and $q>0$,
and
$H^q(G,\Z(j))=\Z$
if $q=0$ and~$j$ is even.
By mapping $H^i(X(R),\sF)$
to $H^i(X(R),\sF/2\sF)=H^i(X(R),\Z/2\Z)$ in the obvious way, we thus obtain a natural map
\begin{align}
\label{eq:natural projection}
H^i_G(X(R),\sF(j)) \mkern3mu\to \mkern-20mu\bigoplus_{\substack{0\leq p\leq i \\ p\mkern1mu\equiv\mkern1mu i-j \text{ mod } 2}}\mkern-20muH^p(X(R),\Z/2\Z)\rlap{\text{,}}
\end{align}
which is an isomorphism if $i>\dim(X)$ (see \cite[\textsection5]{delfsknebuschsurvey}) or if~$j$ is odd.

If $Y \subset X$ is a closed subvariety of~$X$, the exact same reasoning yields a canonical decomposition
\begin{align}
\label{eq:canonical decomposition with support}
H^i_{G,Y(R)}(X(R),\sF(j)) = \bigoplus_{p+q=i} H^p_{Y(R)}(X(R), \sF\otimes_\Z H^q(G,\Z(j)))\rlap{\text{.}}
\end{align}

We warn the reader that the formula (\ref{eq:cup product})
does \emph{not} describe
cup products in equivariant cohomology with integral coefficients
in terms of the decompositions~\eqref{eq:canonical decomposition}.

\subsubsection{Change of coefficients}
\label{subsubsec:reduction modulo 2}

Krasnov~\cite[Theorem~1.2, Theorem~1.3]{krasnovequivariant}
has shown that the above canonical decompositions of cohomology with coefficients in~$\sF(j)$ and in~$\Z/2\Z$ are compatible in the following sense:
for any~$i$ and~$j$,
the reduction map
\begin{align*}
H^i_G(X(R),\sF(j)) \to H^i_G(X(R),\sF(j) \otimes_\Z \Z/2\Z)=H^i_G(X(R),\Z/2\Z)\rlap{\text{,}}
\end{align*}
the natural map~\eqref{eq:natural projection}
and the decomposition~\eqref{eq:canonical decomposition mod 2} fit into a commutative square
\begin{align}
\label{eq:commutative square reduction modulo 2}
\begin{aligned}
\xymatrix{
H^i_G(X(R),\sF(j)) \ar[r] \ar[d] & \displaystyle
\mkern-20mu{\smash{{\bigoplus_{\substack{0\leq p\leq i \\ p\mkern1mu\equiv\mkern1mu i-j \text{ mod } 2}}}}}\phantom{\bigoplus}\mkern-40mu H^p(X(R),\Z/2\Z) \ar[d]^{1+\beta_\sF} \\
H^i_G(X(R),\Z/2\Z) \ar@{=}[r] & \displaystyle \bigoplus_{0\leq p\leq i} H^p(X(R), \Z/2\Z)\rlap{\text{,}}
}
\end{aligned}
\end{align}
where the right-hand vertical arrow is the sum of the twisted Bockstein homomorphisms $\beta_\sF:H^p(X(R),\Z/2\Z) \to H^{p+1}(X(R),\Z/2\Z)$, defined as the boundary maps of the short exact sequence
$0 \to \Z/2\Z \to \sF \otimes_\Z \Z/4\Z \to \Z/2\Z \to 0$, for $0\leq p\leq i-1$, and of the identity maps  $H^p(X(R),\Z/2\Z) \to H^{p}(X(R),\Z/2\Z)$, for $0\leq p\leq i$.

\subsubsection{Effect on the real-complex short exact sequence}
\label{subsubsec:effect}

Let us consider the effect of the decomposition~\eqref{eq:canonical decomposition mod 2} on the real-complex exact sequence~\eqref{eq:real-complex sequence 01}
for $M=\Z/2\Z$.
Applying the functor $\RGamma(G,-)$ to this short exact sequence yields,
via the canonical isomorphisms~\eqref{eq:derived decomposition modulo 2}
and $\RGamma(G,\Z/2\Z[G])=\Z/2\Z$,
a distinguished triangle
\begin{align}
\label{eq:dt effect of decomposition}
\bigoplus_{q\geq 0} \Z/2\Z[-q] \to \Z/2\Z \to \bigoplus_{q\geq 0} \Z/2\Z[-q] \to
 \bigoplus_{q\geq -1} \Z/2\Z[-q]\rlap{\text{.}}
\end{align}
The first map of~\eqref{eq:dt effect of decomposition} is the natural projection, the second map vanishes,
and the third is the natural inclusion. (Indeed, these induce the correct maps after
passing to cohomology, and the category of $\Z/2\Z$\nobreakdash-modules is semisimple.)
Thus,
the decomposition~\eqref{eq:canonical decomposition mod 2} identifies the canonical map
$H^n_G(X(R),\Z/2\Z) \to H^n(X(R),\Z/2\Z)$ with the projection map
$H^0(X(R),\Z/2\Z) \oplus\dots\oplus H^n(X(R),\Z/2\Z) \to H^n(X(R),\Z/2\Z)$
and it identifies the map
$H^n_G(X(R),\Z/2\Z) \to H^{n+1}_G(X(R),\Z/2\Z)$
induced by~\eqref{eq:real-complex sequence 01} with the inclusion
$\bigoplus_{0 \leq p \leq n} H^p(X(R),\Z/2\Z)
\subseteq
\bigoplus_{0 \leq p \leq n+1} H^p(X(R),\Z/2\Z)$.

\subsubsection{Decomposition of $\omega$}
\label{subsubsec:decomposition omega}

The class $\omega^i_{\Z/2\Z} \in H^i_G(X(R),\Z/2\Z)$
(defined in~\textsection\ref{subsubsec:omega})
is mapped,
by the canonical decomposition~\eqref{eq:canonical decomposition mod 2},
to the unit element of the graded ring
$\bigoplus_{0\leq p \leq i}H^p(X(R),\Z/2\Z)$, since this is so when~$X$ is a point
and since~\eqref{eq:canonical decomposition mod 2}
is compatible with pull-backs.

\subsubsection{Covariant functoriality}
\label{subsubsec:covariant functoriality decomposition mod 2}

Let~$X$ and~$Y$ be smooth equidimensional varieties over~$R$.  Let $c=\dim(X)-\dim(Y)$.
Let $f:Y\to X$ be a proper morphism.
The push-forward map $f_*:H^{i-c}_G(Y(R),\Z/2\Z) \to H^i_G(X(R),\Z/2\Z)$
constructed in~\eqref{eq:pushforward equivariant} can be regarded,
via the decompositions~\eqref{eq:canonical decomposition mod 2},
as a map
\begin{align}
\label{eq:pushforward in decomposition}
f_*:\bigoplus_{0\leq p\leq i-c} H^p(Y(R), \Z/2\Z)\to
\bigoplus_{0\leq p\leq i} H^p(X(R), \Z/2\Z)\rlap{.}
\end{align}
It follows at once from the interpretation of~\eqref{eq:canonical decomposition mod 2}
in terms of the Borel construction and of the K\"unneth formula
(see~\textsection\ref{subsubsec:canonical decomposition mod 2}),
and from the projection formula (see~\eqref{eq:projection formula}),
that~\eqref{eq:pushforward in decomposition}
coincides with the direct sum, over $p \in \{0,\dots,i-c\}$,
of the push-forward map $f_*:H^p(X(R),\Z/2\Z) \to H^{p+c}(X(R),\Z/2\Z)$
constructed in~\eqref{eq:pushforward}.

\subsection{On the normal bundle of \texorpdfstring{$X(R)$}{X(R)} in \texorpdfstring{$X(C)$}{X(C)}}
\label{subsec:on the normal bundle}

We fix a smooth variety~$X$ over~$R$.
Making use of the invertibility of
the equivariant Euler class of the normal bundle of~$X(R)$
in~$X(C)$, viewed as an element of the non-equivariant cohomology ring of~$X(R)$
via the
canonical decompositions of~\textsection\ref{subsec:canonical decompositions},
we investigate, in this section,
the long exact sequence of equivariant cohomology
of~$X(C)$ with support in~$X(R)$.

\subsubsection{With $\Z/2\Z$ coefficients}
\label{subsubsec:normal bundle with mod 2 coefficients}

With $\Z/2\Z$ coefficients, we show that
this long exact sequence
decomposes into canonically split short exact sequences.
\begin{prop}
\label{prop:short exact sequences}
Let~$X$ be a smooth variety over~$R$.
For every $i \geq 0$, the sequence
\begin{align*}
0 \to H^i_{G,X(R)}(X(C),\Z/2\Z) \to H^i_G(X(C),\Z/2\Z) \to H^i_G(X(C) \setminus X(R),\Z/2\Z) \to 0
\end{align*}
is exact and canonically split.
\end{prop}

The exactness of this sequence and the role of the equivariant Euler class appear
in \cite[Proposition~5.3.7]{alldaypuppe} when $i\gg 0$ (in the setting of Poincar\'e duality spaces)
and in \cite[Lemma~2.6]{vanhamelabeljacobi} for all $i\geq 0$.
 The description of a canonical splitting is new and will be of importance in the sequel.
 The remainder of~\textsection\ref{subsubsec:normal bundle with mod 2 coefficients}
is devoted to constructing a canonical retraction
of the left arrow of the above sequence,
for any~$i$, thus proving
Proposition~\ref{prop:short exact sequences}.
From now on, we assume, as we may, that~$X$ is irreducible, and we let $d=\dim(X)$.

\begin{defn}
\label{def:gamma}
For the sake of simplicity, let us write $H^p=H^p(X(R),\Z/2\Z)$.
We denote by $\gamma \in H^0 \oplus \dots \oplus H^d$
the image of
the
equivariant fundamental class of~$X(R)$ in~$X(C)$,
introduced in~\textsection\ref{subsubsec:equivariant purity},
by the composition
\begin{align*}
H^d_{G,X(R)}(X(C),\Z/2\Z) \to
H^d_G(X(C),\Z/2\Z) \to
H^d_G(X(R),\Z/2\Z)=
H^0 \oplus \cdots \oplus H^d
\end{align*}
of the forgetful map, the restriction map, and
the canonical decomposition~\eqref{eq:canonical decomposition mod 2}.
\end{defn}

\begin{lem}
\label{lem:gamma invertible}
The class~$\gamma$ is invertible in the graded ring $H^0 \oplus\dots\oplus H^d$.
In other words, its degree~$0$
component
is the constant section $1 \in H^0$.
\end{lem}

\begin{proof}
Let us denote the equivariant fundamental class
by $s \in H^d_{G,X(R)}(X(C),\Z/2\Z)$.
For any~$i$, the composition
of
the equivariant purity
isomorphism
\begin{align}
\label{eq:the equivariant purity isomorphism}
H^{i-d}_G(X(R),\Z/2\Z)\isoto H^i_{G,X(R)}(X(C),\Z/2\Z)
\end{align}
(see~\eqref{eq:equivariant purity example})
with
the forgetful map
\begin{align}
\label{eq:forgetful map}
f_i:H^i_{G,X(R)}(X(C),\Z/2\Z)\to H^i_G(X(C),\Z/2\Z)
\end{align}
and the restriction map
\begin{align}
\label{eq:map gi}
g_i:H^i_G(X(C),\Z/2\Z) \to H^i_G(X(R),\Z/2\Z)
\end{align}
sends $x \in H^{i-d}_G(X(R),\Z/2\Z)$
to $x \smile g_d(f_d(s)) \in H^i_G(X(R),\Z/2\Z)$.
(Indeed, the two maps~\eqref{eq:the equivariant purity isomorphism}
and $g_i\circ f_i$ are induced by maps in the derived category
of sheaves on $(X(R) \times S^n)/G$ that do not depend on $i$.
Hence so is their composition, which can therefore be interpreted,
for all~$i$, as the
cup product with a fixed class in $H^d_G(X(R),\Z/2\Z)$
(see Fact~\ref{fact:cupproduct}). To compute this class,
one simply takes $i=d$.)
By~\textsection\ref{subsubsec:canonical decomposition mod 2},
this map can therefore be identified,
via the canonical decompositions~\eqref{eq:canonical decomposition mod 2},
with the map
\begin{align}
\label{eq:cup product with gamma}
H^0\oplus\dots\oplus H^{i-d} \to H^0 \oplus\dots\oplus H^i
\end{align}
given by multiplication by~$\gamma$
(see~\eqref{eq:cup product}).
On the other hand,
as the groups $H^i_G(X(C) \setminus X(R),\Z/2\Z)$ and $H^i_G(X(C),X(R),\Z/2\Z)$
vanish for $i\gg 0$
(see~\textsection\ref{subsubsec:cohomological dimension}),
the maps~$f_i$ and~$g_i$ are isomorphisms for $i\gg 0$.
Hence~$\gamma$ is invertible (note that $H^0\oplus\dots\oplus H^{i-d}=H^0\oplus\dots\oplus H^d$
for $i\gg 0$,
by \cite[Chapter~II, Lemma~9.1]{delfshomology}).
\end{proof}

Thanks to Lemma~\ref{lem:gamma invertible}, we may now consider,
for any~$i$,
the composition
\begin{align}
\label{eq:canonical retraction mod 2}
H^i_G(X(C),\Z/2\Z) \to H^i_{G,X(R)}(X(C),\Z/2\Z)
\end{align}
of the following maps: first the restriction map~$g_i$,
then the canonical decomposition
$H^i_G(X(R),\Z/2\Z) = H^0 \oplus\dots\oplus H^i$,
then the map $H^0\oplus\dots\oplus H^i \to H^0\oplus\dots\oplus H^{i-d}$
given by multiplication by~$\gamma^{-1}$ followed by projection onto the first $i-d+1$
summands,
then the canonical decomposition
$H^0 \oplus\dots\oplus H^{i-d}=H^{i-d}_G(X(R),\Z/2\Z)$
and finally the isomorphism~\eqref{eq:the equivariant purity isomorphism}.
According to the proof of Lemma~\ref{lem:gamma invertible}, the map~\eqref{eq:canonical retraction mod 2}
is indeed a retraction of~$f_i$.
Thus, Proposition~\ref{prop:short exact sequences} is established.

\begin{rmks}
\label{rks:short exact sequences mod 2}
(i) If $R=\R$, Krasnov has checked that~$\gamma$ coincides
with the total Stiefel--Whitney class of the tangent bundle of~$X(R)$ (see \cite[Theorem~2.1]{krasnovequivariant}).

(ii) Let $\pi_0(X(R))$ denote the set of semi-algebraic connected components of~$X(R)$.
For $V \in \pi_0(X(R))$,
let~$s_V$ denote the image, in $H^d_\et(X,\Z/2\Z)=H^d_G(X(C),\Z/2\Z)$, of the equivariant
fundamental class of~$V$ in~$X(C)$.
For $x \in X(R)$, the image of~$s_V$ by the evaluation map
$\ev_x:H^d_\et(X,\Z/2\Z) \to H^d_\et(x,\Z/2\Z)=\Z/2\Z$
coincides with
the evaluation at~$x$ of the image,
by the map~\eqref{eq:cup product with gamma} for $i=d$,
of the non-zero element of~$H^0$ supported on~$V$.
In view of Lemma~\ref{lem:gamma invertible}, it follows
that $\ev_x(s_V)=0$ if and only if $x\notin V$:
the~$s_V$'s form a canonical family of elements of $H^d_\et(X,\Z/2\Z)$ which separate the
semi-algebraic connected components of~$X(R)$.
This gives a further (and entirely canonical) answer to a question of Colliot-Th\'el\`ene and Parimala~\cite[Remark~2.4.4]{ctparimala}.
At least for $R=\R$, answers to this question had already been given
by Krasnov and by van~Hamel, see \cite[\textsection2.3]{vanhamelabeljacobi}.

(iii)
We have constructed not only a retraction of~$f_i$ but also one of $g_i \circ f_i$
since the map~\eqref{eq:canonical retraction mod 2} factors,
by definition,
through~$g_i$.
The kernel of this retraction
of $g_i \circ f_i$,
and therefore also the cokernel of $g_i \circ f_i$,
is canonically isomorphic to $H^{i-d+1}\oplus\dots\oplus H^i$.
\end{rmks}

\subsubsection{With integral coefficients}
\label{subsubsec:with integral coeff}

The assertion obtained by replacing~$\Z/2\Z$ with $\Z(j)$
in the statement of Proposition~\ref{prop:short exact sequences}
fails: already when $X=\A^1_R$ and $i=j=1$, the group
$H^i_{G,X(R)}(X(C),\Z(j))$ is infinite while $H^i_G(X(C),\Z(j))$ has order~$2$.
It does hold,
however,
when~$j$ has the correct parity, as we now show.

\begin{prop}
\label{prop:short exact sequences integral}
Let~$X$ be a smooth, irreducible variety over~$R$, of dimension~$d$.
Let $\congru{j}{d-1}{2}$.
For every $i \geq 0$, the sequence
\begin{align*}
0 \to H^i_{G,X(R)}(X(C),\Z(j)) \to H^i_G(X(C),\Z(j)) \to H^i_G(X(C) \setminus X(R),\Z(j)) \to 0
\end{align*}
is exact and canonically split.
\end{prop}

As in~\textsection\ref{subsubsec:normal bundle with mod 2 coefficients},
we shall prove Proposition~\ref{prop:short exact sequences integral} by describing
a canonical retraction
of the left arrow of the above sequence.
Let us denote $\Ztilde=\orient_{X(R)}$ and $\Ztilde(j) = \Ztilde \otimes_\Z \Z(j)$.

\begin{proof}
In view of
the canonical isomorphism
of $G$\nobreakdash-equivariant sheaves
 $\orient_{X(C)}=\Z(d)$,
putting together the equivariant purity isomorphisms~\eqref{eq:equivariant purity}
and the square~\eqref{eq:commutative square reduction modulo 2} yields a commutative diagram
\begin{align*}
\xymatrix{
\displaystyle\bigoplus_{\substack{0\leq p\leq i-d \\ p\mkern1mu\equiv\mkern1mu i-j \text{ mod } 2}} \mkern-10muH^p
\ar[d] &
H^{i-d}_G(X(R),\Ztilde(j-d)) \ar[d] \ar[r]^\sim \ar[l]_(.55)\sim &
H^i_{G,X(R)}(X(C),\Z(j)) \ar[d] \\
H^0\oplus\dots\oplus H^{i-d} &
H^{i-d}_G(X(R),\Z/2\Z) \ar@{=}[l] \ar[r]^\sim &
H^i_{G,X(R)}(X(C),\Z/2\Z)
}
\end{align*}
in which the leftmost vertical map is $1+\beta_{\Ztilde}$ (see~\textsection\ref{subsubsec:reduction modulo 2})
and in which the top left horizontal arrow is an isomorphism because
$j-d$ is odd (see~\textsection\ref{subsubsec:decomposition with integral coeff}).
The map $1+\beta_{\Ztilde}$ admits an obvious canonical retraction, namely the projection
map.
Following the isomorphisms in the above diagram, this yields a canonical retraction
of the rightmost vertical map.
Composing it with the map~\eqref{eq:canonical retraction mod 2}
constructed in~\textsection\ref{subsubsec:normal bundle with mod 2 coefficients}
and with the natural map
$H^i_G(X(C),\Z(j)) \to H^i_G(X(C),\Z/2\Z)$,
we obtain the desired retraction.
\end{proof}

Refining these ideas leads to the following integral variant of
\cite[Theorem~2.8]{vanhamelabeljacobi}
and \cite[Corollary~3.2]{krasnovequivariant}.
We shall use it in~\textsection\ref{sec:blochogus} and in~\cite{bwpartie2}.

\begin{prop}
\label{prop:truncated projection integral coeff}
Let~$X$ be a smooth, irreducible variety over~$R$, of dimension~$d$.
For any $i\geq 0$ and any $j \in \Z$, the composition
\begin{align}
\label{eq:truncated projection integral coeff}
H^i_G(X(C),\Z(j))
 \mkern3mu\to \mkern-20mu\bigoplus_{\substack{0\leq p<i-d \\ p\mkern1mu\equiv\mkern1mu i-j \text{ mod } 2}}\mkern-20muH^p(X(R),\Z/2\Z)
\end{align}
of the restriction map
$H^i_G(X(C),\Z(j))\to H^i_G(X(R),\Z(j))$,
of the map~\eqref{eq:natural projection},
and of the natural projection
is surjective.
\end{prop}

\begin{proof}
We shall in fact prove the surjectivity
(and even produce a canonical section)
of
the map obtained by composing~\eqref{eq:truncated projection integral coeff} with
the forgetful map
\begin{align}
\label{eq:forgetful map integral coeff}
H^i_{G,X(R)}(X(C),\Z(j)) \to H^i_G(X(C),\Z(j))
\end{align}
and with the equivariant purity isomorphism
\begin{align}
\label{eq:equiv iso integral coeff}
H^{i-d}_G(X(R),\Ztilde(j-d)) \isoto H^i_{G,X(R)}(X(C),\Z(j))
\end{align}
(see~\eqref{eq:equivariant purity}).
As in the proof of Lemma~\ref{lem:gamma invertible},
we remark
that the
composition
of the maps~\eqref{eq:equiv iso integral coeff}
and~\eqref{eq:forgetful map integral coeff}
with the restriction map
$H^i_G(X(C),\Z(j))\to H^i_G(X(R),\Z(j))$
exists at the derived level: this map can
be written as $H^{i-d}(X(R),e)$
for some map
\begin{align*}
e:\Ztilde \otimes_\Z^\LL \RGamma(G,\Z(j-d)) \to \RGamma(G,\Z(j))[d]
\end{align*}
in the derived category of sheaves of abelian groups on~$X(R)$.
In view of the decomposition~\eqref{eq:decomposition of RGamma integral coeff},
in view of~\textsection\ref{subsubsec:reduction modulo 2},
and in view of the proof of Lemma~\ref{lem:gamma invertible},
this map fits into the following commutative diagram,
in which $\tau_{>0}$ denotes the truncation functor:
\begin{align*}
\xymatrix@R=3ex{
\ar@{=}[d] H^{i-d}_G(X(R),\Ztilde(j-d)) \ar[r]^{H^{i-d}(X(R),e)} & H^i_G(X(R),\Z(j)) \ar@{=}[d] \\
\ar@{->>}[d] \displaystyle \vphantom{\bigoplus}\smash[b]{\bigoplus_{p+q=i-d}} H^p(X(R), \Ztilde \otimes_\Z H^q(G,\Z(j-d))) \ar[r] &
\displaystyle \vphantom{\bigoplus}\smash[b]{\bigoplus_{p+q=i}} H^p(X(R), H^q(G,\Z(j))) \ar@{->>}[d] \\
\ar@{^{(}->}[d]^(.62){1+\labelbetasubtilde}\displaystyle\vphantom{\bigoplus}\smash[b]{\bigoplus_{\substack{0\leq p<i-d \\ p\mkern1mu\equiv\mkern1mu i-j \text{ mod } 2}}}
\mkern-11mu H^p(X(R),\Z/2\Z) \ar[r]^{H^{i-d}(X(R),\tau_{>0}(e))} &
\mkern-11mu\ar@{^{(}->}[d]^(.62){1+\beta_{\Z}}\displaystyle
  \vphantom{\bigoplus}\smash[b]{\bigoplus_{\substack{0\leq p<i-d \\ p\mkern1mu\equiv\mkern1mu i-j \text{ mod } 2}}}
\mkern-11mu H^p(X(R),\Z/2\Z) \\
*!<-.85em,1ex>\entrybox{H^0 \oplus\cdots\oplus H^{i-d-1}} \ar@<-1ex>[r]^{x\mapsto x\gamma} &
*!<-.85em,1ex>\entrybox{H^0 \oplus\cdots\oplus H^{i-d-1}\rlap{.}}
}
\end{align*}
Here, for $\sF=\Ztilde$ or $\sF=\Z$,
the label $1+\beta_{\sF}$ denotes
the sum of the twisted Bockstein homomorphisms
$\beta_\sF:H^p\to H^{p+1}$
(see~\textsection\ref{subsubsec:reduction modulo 2})
for $p<i-d-1$ and of the identity maps  $H^p\to H^p$
for $p<i-d$.

By Lemma~\ref{lem:gamma invertible},
the bottom horizontal map is injective;
therefore, so is the map $H^{i-d}(X(R),\tau_{>0}(e))$.
As the latter is an endomorphism of a finite-dimensional vector space,
it must then be surjective.  The proposition follows.
\end{proof}

\begin{rmk}
One can prove,
although we shall not use this fact,
that the map $H^{i-d}(X(R),\tau_{>0}(e))$ in the above diagram
is given by $x \mapsto x\gammatilde + \beta_{\Ztilde}(x)\beta_{\Ztilde}(\gammatilde)$,
where $\gammatilde \in H^0 \oplus H^2 \oplus \cdots$
is such that $\gamma=\gammatilde + \beta_{\Ztilde}(\gammatilde)$.
In particular,
the derived cup product
with integral coefficients,
in terms of the decomposition~\eqref{eq:decomposition of RGamma integral coeff},
is not just the sum of the individual cup product maps
(in contrast with~$\Z/2\Z$ coefficients; see~\eqref{eq:cup product}).
\end{rmk}

\subsection{Two Poincar\'e dualities}
\label{subsec:two dualities}

If~$X$ is an irreducible, smooth and proper variety over~$R$, of dimension~$d$,
there cannot be, in general, a duality between the equivariant cohomology groups
$H^i_G(X(C),\Z/2\Z)$
and $H^{2d-i}_G(X(C),\Z/2\Z)$,
as the former group need not vanish when $i\gg 0$.
The following proposition
provides
a substitute.

\begin{prop}
\label{prop:poincare duality equivariant}
Let~$X$ be a smooth and proper variety over~$R$, of pure dimension~$d$.
There is a canonical ``trace'' homomorphism
\begin{align}
H^{2d}_G(X(C),X(R),\Q/\Z(d))\to\Q/\Z\rlap{;}
\end{align}
together with cup product, it induces,
for any finite $G$\nobreakdash-module~$M$,
a perfect pairing
of finite abelian groups
\begin{align}
\label{eq:poincare duality equivariant}
H^i_G(X(C),X(R),M) \times H^{2d-i}_G(X(C) \setminus X(R),\Hom(M,\Q/\Z(d))) \to \Q/\Z
\end{align}
for any $i\in \Z$.
\end{prop}

\begin{proof}
This is a reformulation of the semi-algebraic Poincar\'e duality~\eqref{eq:semi-algebraic poincare duality} applied to $V=(X(C)\setminus X(R))/G$
and to $\sG=\sH^0(G,M)$,
in view of the isomorphisms~\eqref{eq:cohomological dimension iso relative}
and~\eqref{eq:cohomological dimension iso open},
of the isomorphism $\orient_V=\sH^0(G,\Z(d))$
and of the remark that by the properness of~$X$,
the cohomology of~$V$ with complete supports coincides with
the cohomology of~$X(C)/G$ relative to~$X(R)$
(see \cite[Chapter~II, Theorem~5.1]{delfshomology}, \cite[Theorem~2.3]{delfsknebuschsurvey}).
\end{proof}

\begin{rmks}
\label{rks:equivariant poincare}
(i)
When $X(R)=\emptyset$, Proposition~\ref{prop:poincare duality equivariant} produces
a canonical duality between
the \'etale cohomology groups
$H^i_\et(X,M)$ and $H^{2d-i}_\et(X,\Hom(M,\Q/\Z(d)))$.
Contrary to what happens for the \'etale cohomology of varieties over a $p$\nobreakdash-adic field
(for which see~\cite[Lemma~2.9]{saitolakelouise}),
this duality does not result from the formal combination of Poincar\'e
duality for the \'etale cohomology of~$X_C$ with a duality in the Galois cohomology of~$R$.
Such a combination only yields a duality for the modified (\`a la Tate) \'etale hypercohomology of~$X$
(\emph{loc.\ cit.}, Lemma~2.10).

(ii)
If~$M$ is a finitely generated $G$\nobreakdash-module,
applying Proposition~\ref{prop:poincare duality equivariant}
to $M/nM$ for all $n\geq 1$
yields, for any $i\in \Z$, a pairing that identifies
the profinite completion
of the finitely generated abelian group
$H^i_G(X(C),X(R),M)$
with the Pontrjagin dual of the torsion abelian group $H^{2d-i}_G(X(C) \setminus X(R),\Hom(M,\Q/\Z(d)))$.
\end{rmks}

Let us fix an irreducible, smooth and proper variety~$X$ over~$R$, of dimension~$d$.
We now have two Poincar\'e dualities at our disposal:
one for the equivariant cohomology of $X(C)\setminus X(R)$,
between
the cohomological degrees~$i$ and~$2d-i$ (see Proposition~\ref{prop:poincare duality equivariant}),
and one for the cohomology of~$X(R)$,
between
the cohomological degrees~$i$ and~$d-i$
(see~\eqref{eq:semi-algebraic poincare duality}).
In the remainder of~\textsection\ref{subsec:two dualities},
we show how to reconcile these two seemingly incompatible dualities in a single self-dual long exact sequence
of cohomology groups.

Using the notation of \S\ref{subsubsec:normal bundle with mod 2 coefficients}, we first consider the localisation long exact sequence
\begin{align*}
\xymatrix@C=1.3em{
\cdots \ar[r]^(.2){\delta_i} & H^i_G(X(C),X(R),\Z/2\Z) \ar[r]^(.563){\smash[t]{h_i}} & H^i_G(X(C),\Z/2\Z) \ar[r]^{g_i} & H^i_G(X(R),\Z/2\Z) \ar[r]^(.73){\delta_{i+1}} & \cdots\rlap{\text{.}}
}
\end{align*}
The maps
$f_i:H^i_{G,X(R)}(X(C),\Z/2\Z) \to H^i_G(X(C),\Z/2\Z)$
and $g_i\circ f_i$ are injective,
by Proposition~\ref{prop:short exact sequences}
and Remark~\ref{rks:short exact sequences mod 2}~(iii).
The above sequence therefore remains exact if we
replace $H^i_G(X(C),\Z/2\Z)$ and $H^i_G(X(R),\Z/2\Z)$
with $\Coker(f_i)$ and $\Coker(g_i\circ f_i)$, respectively.
In view of Proposition~\ref{prop:short exact sequences}
and Remark~\ref{rks:short exact sequences mod 2}~(iii),
we obtain, in this way, a long exact sequence
\begin{align}
\label{eq:selfdual sequence}
\begin{aligned}
\xymatrix@R=0.5ex@C=1em{
&*!<8.18em,0ex>\entrybox{
\cdots \xrightarrow{w_{i-1}}
H^{i-d}\oplus\dots\oplus H^{i-1}
\xrightarrow{u_i}
 H^i_G(X(C), X(R),\Z/2\Z)
}\ar`r/6pt[d]`[l]`[dl]`[d]_{v_i}[d] \\
&*!<.3em,0ex>\entrybox{
H^i_G(X(C)\setminus X(R),\Z/2\Z)
\xrightarrow{w_i}
H^{i-d+1}\oplus\dots\oplus H^i \xrightarrow{u_{i+1}} \cdots
}
}
\end{aligned}
\end{align}
in which the arrows may be described as follows.
Let us denote by~$e_i$ the restriction map $e_i:H^i_G(X(C),\Z/2\Z) \to H^i_G(X(C)\setminus X(R),\Z/2\Z)$.
We have $v_i=e_i \circ h_i$.
The map~$w_i$ is characterised by the property that~$w_i(e_i(x'))$,
for $x' \in H^i_G(X(C),\Z/2\Z)$, is obtained by
decomposing~$g_i(x')$
according to~\eqref{eq:canonical decomposition mod 2},
then multiplying by~$\gamma^{-1}$
and projecting onto $H^{i-d+1}\oplus\dots\oplus H^i$.
Finally, to compute $u_i(y)$ for $y \in H^{i-d}\oplus\dots\oplus H^{i-1}$,
we multiply~$y$ by~$\gamma$, project onto $H^{i-d} \oplus\dots\oplus H^{i-1}$,
view the result as an element of $H^{i-1}_G(X(R),\Z/2\Z)$ via~\eqref{eq:canonical decomposition mod 2}
and apply the connecting homomorphism~$\delta_i$ of the localisation exact sequence.

Poincar\'e duality for~$X(R)$, in the form of~\eqref{eq:semi-algebraic poincare duality}, yields
a trace map $H^d \to \Z/2\Z$ (which is an isomorphism if and only if~$X(R)$ is semi-algebraically connected).
Let us denote by $\deg:\bigoplus_{p \geq 0} H^p \to \Z/2\Z$ its composition with the projection map onto~$H^d$.
For every~$i$, let us consider the pairing
\begin{align}
\label{eq:modified pairing}
\left( H^{i-d+1} \oplus \dots \oplus H^i \right) \times \left( H^{d-i} \oplus \dots \oplus H^{2d-i-1} \right) \to \Z/2\Z
\end{align}
defined by $(x,y) \mapsto \deg(xy\gamma)$, where the products take place in the ring
$\bigoplus_{p \geq 0} H^p$.
This is a perfect pairing since the map $H^p \times H^{d-p} \to \Z/2\Z$ given by $(x,y) \mapsto \deg(xy)$ is a
perfect pairing
for every~$p$
(see~\eqref{eq:semi-algebraic poincare duality})
and~$\gamma$ is invertible.

\begin{thm}
\label{th:selfduality}
For any irreducible, smooth and proper variety~$X$ over~$R$,
the long exact sequence~\eqref{eq:selfdual sequence}
is self-dual with respect to the perfect pairings~\eqref{eq:poincare duality equivariant}
and~\eqref{eq:modified pairing}.
In other words, if~$d$ denotes the dimension of~$X$, then,
for every integer~$i$,
the duals of $u_i$, $v_i$, $w_i$
with respect to these pairings
are $w_{2d-i}$, $v_{2d-i}$, $u_{2d-i}$, respectively.
\end{thm}

\begin{proof}
For $x \in H^i_G(X(C),X(R),\Z/2\Z)$ and $y \in H^{2d-i}_G(X(C),X(R),\Z/2\Z)$,
we have $v_i(x) \smile y = x \smile y = x \smile v_{2d-i}(y)$
in $H^{2d}_G(X(C),X(R),\Z/2\Z)$,
hence~$v_{2d-i}$ is dual to~$v_i$.
Thus, to prove the theorem, we only have to show that
\begin{align}
\label{eq:adjunction}
x \smile u_{2d-i}(y) = \deg(w_i(x)y\gamma)
\end{align}
for any $i\geq 0$, any $x \in H^i_G(X(C)\setminus X(R),\Z/2\Z)$ and any $y \in H^{d-i}\oplus\dots\oplus H^{2d-i-1}$.
Let us fix~$i$, $x$ and~$y$.
Let $x' \in H^i_G(X(C),\Z/2\Z)$ be the image of~$x$ by the canonical section of~$e_i$
given by Proposition~\ref{prop:short exact sequences}.
Let $x'' \in H^0 \oplus\dots\oplus H^i$
denote the decomposition
of~$g_i(x')$
via~\eqref{eq:canonical decomposition mod 2}.
As~$x'$ is annihilated
by the retraction~\eqref{eq:canonical retraction mod 2},
the product $x'' \gamma^{-1}$
belongs to
$H^{i-d+1} \oplus \dots \oplus H^i \subseteq H^0\oplus\dots\oplus H^i$.
It follows that $w_i(x)=x'' \gamma^{-1}$
and hence that
\begin{align}
\label{eq:adjoint computation 1}
\deg(w_i(x)y\gamma)=\deg(x''y)\rlap{\text{.}}
\end{align}
On the other hand, if $t'' \in H^{d-i} \oplus\dots\oplus H^{2d-i-1}$ denotes the projection
of~$y\gamma$
and $t' \in H^{2d-i-1}_G(X(R),\Z/2\Z)$ the element
which corresponds to~$t''$ via~\eqref{eq:canonical decomposition mod 2},
we have
\begin{align}
\label{eq:adjoint computation 2}
x\smile u_{2d-i}(y) &= x' \smile u_{2d-i}(y) = x' \smile \delta_{2d-i}(t') \\\notag&= \delta_{2d}(g_i(x') \smile t')=\deg(x''t''\gamma^{-1})\rlap{\text{,}}
\end{align}
where the last equality follows from Lemma \ref{lem:boundary map and degree} below
(and the penultimate one from \cite[Chapter~II, Theorem~7.1~(b)]{bredon}).
Finally, as $x''\gamma^{-1}$
belongs to $H^{i-d+1} \oplus \dots \oplus H^i$,
we have $x''t''\gamma^{-1}=x''\gamma^{-1}t''=x''\gamma^{-1}y\gamma=x''y$;
in view of~\eqref{eq:adjoint computation 1} and~\eqref{eq:adjoint computation 2},
this completes the proof of~\eqref{eq:adjunction} and hence of Theorem~\ref{th:selfduality}.
\end{proof}

\begin{lem}
\label{lem:boundary map and degree}
Let $z \in H^0\oplus\dots\oplus H^{2d-1}$.
The image of~$z$ by the
composed map
\begin{align*}
H^0 \oplus\dots\oplus H^{2d-1}=H^{2d-1}_G(X(R),\Z/2\Z) \xrightarrow{\delta_{2d}} H^{2d}_G(X(C),X(R),\Z/2\Z)=\Z/2\Z\rlap{\text{,}}
\end{align*}
where the two canonical isomorphisms come
from~\eqref{eq:canonical decomposition mod 2}
and from Proposition~\ref{prop:poincare duality equivariant},
is equal to $\deg(z \gamma^{-1})$.
\end{lem}

\begin{proof}
We may assume that $X(R)\neq\emptyset$.
Let us consider the commutative diagram
\begin{equation*}
\vspace*{-1.5ex}
\xymatrix@C=1.05em@R=3ex{
H^0\oplus\dots\oplus H^d \ar[d] \ar@{=}[r] & H^d_G(X(R),\Z/2\Z) \ar[d] \ar[r]^(.45)\sim & H^{2d}_{G,X(R)}(X(C),\Z/2\Z) \ar[d] \ar[r]^(.53){f_{2d}} & H^{2d}_G(X(C),\Z/2\Z) \ar[d]^(.45){\varepsilon} \\
H^d \ar`d/6pt[r][rrrd]!<0em,2ex>^(.75){\deg} \ar@{=}[r] & H^d(X(R),\Z/2\Z) \ar[r]^(.45)\sim & H^{2d}_{X(R)}(X(C),\Z/2\Z) \ar[r] & H^{2d}(X(C),\Z/2\Z)\rlap{\text{,}} \ar@{=}[d] \\
&&&
*!<0em,-2ex>\entrybox{\Z/2\Z}
}
\end{equation*}
in which the leftmost vertical arrow is the projection map
and the horizontal isomorphisms are the
decomposition~\eqref{eq:canonical decomposition mod 2}
and the equivariant
purity isomorphism~\eqref{eq:equivariant purity},
with coefficients $\Z/2\Z$ or $\Z/2\Z[G]$.
The leftmost square
commutes by~\textsection\ref{subsubsec:effect};
the bottom triangle commutes as a consequence of the fact that~$H^d$ is generated
by classes supported on points.
The hypothesis that $X(R)\neq\emptyset$
implies, by Proposition~\ref{prop:poincare duality equivariant},
that $H^{2d}_G(X(C)\setminus X(R),\Z/2\Z)=0$.
It follows, by Proposition~\ref{prop:short exact sequences},
that~$f_{2d}$ is an isomorphism
and that its inverse is the map~\eqref{eq:canonical retraction mod 2}.
From the definition of~\eqref{eq:canonical retraction mod 2} and from the above diagram,
we deduce the commutativity of the bottom square of the following diagram,
whose vertical arrows are induced by~\eqref{eq:real-complex sequence 01}:
\begin{align*}
\xymatrix@R=1.5ex@C=1em{
H^{2d}_G(X(C),X(R),\Z/2\Z) \ar@{=}[d] && H^{2d-1}_G(X(R),\Z/2\Z) \ar[ll]_(.45){\delta_{2d}} \save+<-4.7em,-1.7ex>\ar@{=}[dl]!<0em,1.6ex>\restore \ar[dd] && H^{2d-1}_G(X(C),\Z/2\Z) \ar[ll]_(.49){g_{2d-1}} \ar[dd] \\
\Z/2\Z
& \hbox to 0pt{\hss\rlap{$H^0\oplus\dots\oplus H^{2d-1}$}\phantom{$H^0\oplus\dots\oplus H^{2d}$}\hss} \ar@<-.17em>@{^{(}->}[dd]^(.45)\iota \\
&& H^{2d}_G(X(R),\Z/2\Z) \save+<-4.3em,-1.7ex>\ar@{=}[dl]!<.7em,2.05ex>\restore && \smash[b]{H^{2d}_G(X(C),\Z/2\Z)}\vphantom{\lower.3ex\hbox{x}} \ar@{->>}[ll]_(.49){g_{2d}} \ar[d]^(.42)\varepsilon \\
& \hbox to 0pt{\hss$H^0\oplus\dots\oplus H^{2d}$\hss} \save+<3.7em,0ex>\ar[rr]!(-8,0)^{z \mapsto \deg(z\gamma^{-1})}\restore &&& \mkern-66mu\Z/2\Z=H^{2d}(X(C),\Z/2\Z)\rlap{\text{.}}
}
\end{align*}
The map~$g_{2d}$ is surjective, as $H^{2d+1}_G(X(C),X(R),\Z/2\Z)=0$
(see~\textsection\ref{subsubsec:cohomological dimension}).
By~\textsection\ref{subsubsec:effect},
the map~$\iota$ is the inclusion.
As the top row and the right column are exact,
we see that an element of $H^0 \oplus\dots\oplus H^{2d-1}$ dies in the top left corner
if and only if it dies in the bottom right corner.  This proves the lemma.
\end{proof}

\subsection{A real Lefschetz hyperplane theorem}
\label{subsec:real lefschetz}

We shall need,
in~\cite[\textsection\ref*{BW2-par:systeme anticanonique}]{bwpartie2},
a Lefschetz hyperplane theorem
for the equivariant Gysin map
\begin{align}
\label{eq:gysin in lefschetz}
H^{i-2}_G(Y(C),\Z(j-1))\to H^i_G(X(C),\Z(j))
\end{align}
of a smooth ample hypersurface $Y \subset X$ in a smooth and projective variety~$X$ over~$R$
(see~\eqref{eq:equivariant purity subvariety}).
The usual formulation is correct when $X(R)=\emptyset$,
as we now check.

\begin{prop}
\label{prop:weak lefschetz no point}
Let~$X$ be a smooth and projective variety over~$R$, of dimension~$d$.
Let $Y\subset X$ be a smooth ample hypersurface.
Let $j \in \Z$.
If $X(R)=\emptyset$,
the map~\eqref{eq:gysin in lefschetz}
is bijective for any $i>d+1$ and is surjective for $i=d+1$.
\end{prop}

\begin{proof}
Let $U = X \setminus Y$.
It suffices to check that $H^i_G(U(C),M)=0$ for any $i>d$ and any
finitely generated $G$\nobreakdash-module~$M$.
This is true for $i\gg 0$, as $U(R)=\emptyset$
(see~\textsection\ref{subsubsec:cohomological dimension}).
On the other hand, if $i>d$,
we have $H^i(U(C),M) \otimes_\Z \Zl=H^i_\et(U_C,M \otimes_\Z \Zl)=0$
for every~$\ell$, as~$U_C$ is affine and~$C$ is algebraically closed
(see \cite[Chapter~VI, Theorem~7.2]{milneet}), hence $H^i(U(C),M)=0$.
The desired vanishing then follows,
thanks to the real-complex exact sequence~\eqref{eq:real-complex long 10},
by a descending induction on~$i$
in which~$M$ is allowed to vary.
\end{proof}

When $X(R)\neq\emptyset$,
some condition on the real Gysin maps
\begin{align}
\label{eq:real gysin in lefschetz}
H^{p-1}(Y(R),\Z/2\Z)\to H^p(X(R),\Z/2\Z)
\end{align}
(see~\eqref{eq:purity mod 2})
must appear in the formulation
of a Lefschetz hyperplane theorem
as
the ampleness of~$Y$
has no effect on
the injectivity or surjectivity of these maps.
For example,
the conclusion of Proposition~\ref{prop:weak lefschetz no point} fails whenever $X(R)\neq\emptyset$
and $Y(R)=\emptyset$, since in this case
$H^{i-2}_G(Y(C),\Z(i-1))=0$ and $H^i_G(X(C),\Z(i))\neq 0$ for $i>2d$,
by \S\ref{subsubsec:cohomological dimension} and~\eqref{eq:canonical decomposition}.
The next proposition generalises the surjectivity half of Proposition~\ref{prop:weak lefschetz no point} when $X(R)\neq\emptyset$.

\begin{prop}
\label{prop:weak lefschetz surjectivity}
Let~$X$ be a smooth and projective variety over~$R$, of dimension~$d$.
Let $Y\subset X$ be a smooth ample hypersurface.  Let $i,j\in \Z$
with $i\geq d+1$. If~\eqref{eq:real gysin in lefschetz}
is surjective for every $p\geq 1$ such that $\congru{p}{i-j}{2}$,
then the image of~\eqref{eq:gysin in lefschetz}
is
\begin{align*}
\left\{\alpha\in H_G^i(X(C),\Z(j));\; \alpha|_x=0 \text{ for all }x\in X(R)\right\}\rlap{\text{,}}
\end{align*}
where $\alpha|_x \in H^i_G(x,\Z(j))=H^i(G,\Z(j))$ denotes the restriction of~$\alpha$ to~$x$.
\end{prop}

We start with a lemma.

\begin{lem}
\label{lem:restrmapinjective}
Let~$U$ be an affine variety over~$R$, of dimension~$d$.
For any $i>d$
and any $G$\nobreakdash-module~$M$,
the restriction map
$H^i_G(U(C),M) \to H^i_G(U(R),M)$ is an isomorphism.
\end{lem}

\begin{proof}
Writing~$M$ as the union of its finitely generated sub-$G$\nobreakdash-modules,
we see that we may assume~$M$ to be finitely generated.
The conclusion of the lemma, which holds for $i\gg 0$
(see~\textsection\ref{subsubsec:cohomological dimension}),
then follows for $i>d$ by a descending induction, just as in the proof of Proposition~\ref{prop:weak lefschetz no point}, thanks to the real-complex exact sequence~\eqref{eq:real-complex long 10}
and to the vanishing, for $i>d$, of the groups $H^i(U(C),M)$ (see the proof of Proposition~\ref{prop:weak lefschetz no point})
and $H^i(U(R),M)$ (see \cite[Chapter~II, Lemma~9.1]{delfshomology}).
\end{proof}

\begin{proof}[Proof of Proposition~\ref{prop:weak lefschetz surjectivity}]
Let $U=X \setminus Y$.
The commutative diagram
\begin{align}
\label{eq:proof of weak lefschetz forget support}
\begin{aligned}
\xymatrix@R=3ex{
H^i_{G,Y(C)}(X(C),\Z(j)) \ar[r]^(.54){\theta^i_{G,C}} \ar[d] & H^i_G(X(C),\Z(j)) \ar[d] \ar[r] & H^i_G(U(C),\Z(j)) \ar[d]^(.45)\wr \\
H^i_{G,Y(R)}(X(R),\Z(j)) \ar[r]^(.54){\theta^i_{G,R}} & H^i_G(X(R),\Z(j)) \ar[r] & H^i_G(U(R),\Z(j))
}
\end{aligned}
\end{align}
has exact rows and its rightmost vertical map is an
isomorphism by Lemma~\ref{lem:restrmapinjective}.

Let $\theta^p_R:H^p_{Y(R)}(X(R),\Z/2\Z) \to H^p(X(R),\Z/2\Z)$
denote the forgetful map.
Using the decompositions~\eqref{eq:canonical decomposition}
and~\eqref{eq:canonical decomposition with support} (with $\sF=\Z$), we see that
when $i>d$, the map~$\theta^i_{G,R}$ can be identified with the direct sum,
over
all $p \geq 0$
with $\congru{p}{i-j}{2}$,
of the maps~$\theta^p_R$.
Thus, the maps $\theta^p_R$
for $p \geq 1$
such that $\congru{p}{i-j}{2}$
are all surjective
if and only if
the
image
of~$\theta^i_{G,R}$
coincides with the set of $\alpha \in H^i_G(X(R),\Z(j))$
whose component in $H^0(X(R),H^i(G,\Z(j)))$
vanishes, \emph{i.e.}, with the set of~$\alpha$
such that $\alpha|_x=0$ for all $x\in X(R)$.
The proposition then follows by a chase in the diagram~\eqref{eq:proof of weak lefschetz forget support}, as~$\theta^p_R$
and~$\theta^i_{G,C}$
can be identified, respectively, with~\eqref{eq:real gysin in lefschetz}
and~\eqref{eq:gysin in lefschetz}, via~\eqref{eq:purity mod 2}
and~\eqref{eq:equivariant purity subvariety}.
\end{proof}

\begin{rmks}
\label{rmk:lefschetzfaible}
(i) One can also check
that for $i>d+1$,
if~\eqref{eq:real gysin in lefschetz}
is injective for every $p\geq 1$ such that $\congru{p}{i-j}{2}$ and is surjective for every $p\geq 1$ such that $\pascongru{p}{i-j}{2}$,
then~\eqref{eq:gysin in lefschetz} is injective.
However, the proof is more involved and we shall not use this fact.

(ii)
In the setting of Lemma~\ref{lem:restrmapinjective},
Scheiderer \cite[Corollary~18.11]{scheiderer} has shown,
at least when~$M$ is torsion,
the stronger
(and significantly more delicate)
fact
that $H^i_G(U(C),U(R),M)=0$ for all $i>d$.
\end{rmks}

\subsection{Cycle classes and topological constraints}
\label{subsec:topologicalconstraints}

We fix, in this section,
an integer~$k$ and a smooth variety~$X$ over~$R$.

\subsubsection{The equivariant cycle class map}
\label{subsubsec:eqcl}

For any irreducible closed subvariety
$Y \subset X$ of codimension~$k$,
the \emph{equivariant fundamental class of~$Y(C)$ in~$X(C)$}
is the image of~$1$
by the canonical isomorphism
\begin{align}
\Z=H^{2k}_{G,Y(C)}(X(C),\Z(k))
\end{align}
given by~\eqref{eq:equivariant purity subvariety}
when~$Y$ is smooth,
and which stems from~\eqref{eq:equivariant purity subvariety} by d\'evissage
in general
(see \cite[Chapter~VI, \textsection9]{milneet}).
We let $\cl(Y)$ denote its image in $H^{2k}_G(X(C),\Z(k))$.
For a codimension~$k$ cycle $Y=\sum n_i Y_i$ on~$X$,
we let $\cl(Y)=\sum n_i \cl(Y_i)$.
This assignment induces a homomorphism
\begin{align}
\cl:\CH^k(X) \to H^{2k}_G(X(C),\Z(k))
\end{align}
from the Chow group of codimension~$k$ cycles on~$X$,
as was verified by Krasnov \cite[Proposition~2.1.1]{krasnovcharacteristicclasses}
when $R=\R$.  The proof given in \emph{loc.\ cit.} carries over to an arbitrary real closed field~$R$
in view of the homotopy invariance of semi-algebraic cohomology
(see \cite[\textsection6]{delfsknebuschsurvey}, \cite{delfshomotopyaxiom}).
The map~$\cl$ is compatible with proper push-forwards (by functoriality of push-forward maps),
with products
(see \cite[Proposition~2.1.3]{krasnovcharacteristicclasses})
and with pull-backs
(\emph{loc.\ cit.}, Proposition~2.3.3).

\subsubsection{The Borel--Haefliger cycle class map}
\label{subsubsec:borel-haefliger}

For any irreducible closed subvariety $Y \subset X$ of codimension~$k$,
Borel and Haefliger~\cite{borelhaefliger}
(and Delfs \cite[Chapter~III, Theorem~3.7]{delfshomology}
over an arbitrary real closed field)
have proved the existence
of a unique class in $H^k_{Y(R)}(X(R),\Z/2\Z)$,
called the \emph{fundamental class of $Y(R)$ in $X(R)$},
whose restriction
to~$H^k_{Y^0(R)}(X^0(R),\Z/2\Z)$
is the
fundamental class of~$Y^0(R)$
in~$X^0(R)$
in the sense of~\eqref{eq:purity mod 2},
where~$Y^0$ (resp.~$X^0$) denotes the complement, in~$Y$ (resp.~$X$), of the singular locus of~$Y$.
We denote by $\cl_R(Y)$ its image in $H^k(X(R),\Z/2\Z)$.
For a codimension~$k$ cycle $Y=\sum n_i Y_i$ on~$X$,
we let $\cl_R(Y)=\sum n_i \cl_R(Y_i)$.
This assignment induces a homomorphism
\begin{align}
\label{eq:borel-haefliger}
\cl_R:\CH^k(X) \to H^k(X(R),\Z/2\Z)
\end{align}
(see \cite[Proposition~3.4, Remark~3.5]{scheidererpurity}),
which is compatible with proper push-forwards, products and pull-backs
(\emph{e.g.}, as a consequence of
the same property for the equivariant cycle class and of~\textsection\ref{subsubsec:compatibility with pushforwards} below). We define $H^k_{\alg}(X(R),\Z/2\Z)$ to be the image of the map (\ref{eq:borel-haefliger}).

\subsubsection{Topological constraints}
\label{subsubsec:topological constraints}

Let us consider the composition
\begin{align}
\label{eq:composition restriction decomposition}
H^{2k}_G(X(C),\Z(k)) \to H^{2k}_G(X(R),\Z(k))\to \mkern-15mu\bigoplus_{\substack{0\leq p\leq 2k \\ \vphantom{X^X}p \mkern1mu\equiv\mkern1mu k \text{ mod } 2}}\mkern-15muH^p(X(R),\Z/2\Z)
\end{align}
of the restriction map from $X(C)$ to the $G$\nobreakdash-invariant
semi-algebraic subspace $X(R)$
with the map~\eqref{eq:natural projection}
induced by the decomposition~\eqref{eq:canonical decomposition}.
For $\alpha \in H^{2k}_G(X(C),\Z(k))$, we let $\alpha_p \in H^p(X(R),\Z/2\Z)$ denote the $p$th coordinate
of the image of~$\alpha$ by~\eqref{eq:composition restriction decomposition}.

The next theorem spells out the relationship between the equivariant cycle class
and the Borel--Haefliger cycle class, as well as a topological constraint,
expressed in terms of the Steenrod squares
$\Sq^i : H^k(X(R),\Z/2\Z) \to H^{k+i}(X(R),\Z/2\Z)$
(see \cite{epstein}, \cite[Proposition~4.2]{raynaudmodulesproj}),
which equivariant cycle classes must satisfy.
We take the convention that
$\Sq^i=0$ for $i<0$
and recall that $\Sq^0$ is the identity.

\begin{thm}
\label{th:conditions de krasnov}
Let~$X$ be a smooth variety over~$R$.   Let~$k$ be an integer.  Let~$Y$ be a cycle of codimension~$k$
on~$X$ and let $\alpha = \cl(Y) \in H^{2k}_G(X(C),\Z(k))$.
Then
$$\alpha_{k+i} = \Sq^i(\cl_R(Y))$$ for every $i \in 2\Z$.
\end{thm}

This theorem is due to Kahn~\cite{kahnchern}
for Chern classes and to Krasnov~\cite{krasnovequivariant}
in general (at least when $R=\R$).
A more general result will be proved
in~\textsection\ref{subsubsec:compatibility with pushforwards}.

The statement of Theorem~\ref{th:conditions de krasnov} motivates the following definition.

\begin{defn}
\label{def:condition topologique}
We denote by $H^{2k}_G(X(C),\Z(k))_0 \subseteq H^{2k}_G(X(C),\Z(k))$ the subgroup consisting
of those classes~$\alpha$ which satisfy
$\alpha_{k+i} = \Sq^i(\alpha_k)$ for every $i\in 2\Z$.
\end{defn}

Thanks to Theorem~\ref{th:conditions de krasnov},
we may now view the equivariant cycle class map
as a map $\cl:\CH^k(X) \to H^{2k}_G(X(C),\Z(k))_0$.
As a notable consequence of Theorem~\ref{th:conditions de krasnov},
the Borel--Haefliger cycle class map factors through the equivariant cycle class map;
namely, it coincides with the composition of~$\cl$ with the map
\begin{align}
\label{eq:psi with point}
H^{2k}_G(X(C),\Z(k))_0 \to H^k(X(R),\Z/2\Z)
\end{align}
defined by $\alpha\mapsto \alpha_k$.

\begin{rmks}
\label{rk:topological condition 1-cycles and ak}
(i) Let $d=\dim(X)$.
The condition appearing in Definition~\ref{def:condition topologique}
takes a particularly simple form
for $k=d-1$ since $H^p(X(R),\Z/2\Z)=0$ for $p\geq d+1$:
the group $H^{2d-2}_G(X(C),\Z(d-1))_0$
consists of those $\alpha \in H^{2d-2}_G(X(C),\Z(d-1))$
such that $\alpha_p=0$ for all $p<d-1$.

(ii)
Let us fix $\alpha \in H^{2k}_G(X(C),\Z(k))$ and
denote by $\bar\alpha \in H^{2k}_G(X(R),\Z/2\Z)$
the class obtained by restricting to~$X(R)$ and by
reducing the coefficients modulo~$2$.
For $p\in \Z$, let $\bar\alpha_p \in H^p(X(R),\Z/2\Z)$ be the $p$th coordinate
of~$\bar\alpha$ in the decomposition~\eqref{eq:canonical decomposition mod 2}.
Then $\alpha \in H^{2k}_G(X(C),\Z(k))_0$
if and only if
$\bar\alpha_{k+i} = \Sq^i(\bar\alpha_k)$ for every $i\in \Z$.
This follows from~\textsection\ref{subsubsec:reduction modulo 2}
and from the Adem relation $\Sq^{2m+1}=\beta_{\Z}\circ\Sq^{2m}$
(for which we refer the reader to \cite[Proposition~4.2]{raynaudmodulesproj}).

(iii)
For any even~$k$ and
any $\alpha \in H^{2k}_G(X(C),\Z(k))_0$,
the class $\alpha_k^2$ belongs to the image of the natural map
$H^{2k}(X(R),\Z) \to H^{2k}(X(R),\Z/2\Z)$.
Indeed, so does~$\alpha_{2k}$ by its very definition,
and we have
$\alpha_k^2 = \Sq^k(\alpha_k)=\alpha_{2k}$.
As noted by Krasnov~\cite[Remark~4.8]{krasnovequivariant},
this observation, in the case of equivariant cycle classes,
was known to Akbulut and King
(see \cite[Theorem~A~(b)]{akbulutkingtranscendental}).
\end{rmks}

\subsubsection{Compatibility with cup products, pull-backs, push-forwards}
\label{subsubsec:compatibility with pushforwards}

Using~\textsection\ref{subsubsec:reduction modulo 2}, the formula~\eqref{eq:cup product}
and the Cartan formula (see~\cite[p.~91]{milnorstasheff}),
it is easy to check that
Definition~\ref{def:condition topologique} is compatible with cup products, in the sense that
for any~$k$ and~$\ell$,
cup product induces the horizontal arrows of a commutative square
\begin{align*}
\xymatrix@R=3ex{
\ar@<.055em>[d] \mkern-2muH^{2k}_G(X(C),\Z(k))_0 \times H^{2\ell}_G(X(C),\Z(\ell))_0 \ar[r] & H^{2k+2\ell}_G(X(C),\Z(k+\ell))_0 \ar[d] \\
H^k(X(R),\Z/2\Z) \times H^\ell(X(R),\Z/2\Z) \ar[r] & H^{k+\ell}(X(R),\Z/2\Z)
}
\end{align*}
whose vertical arrows are the maps~\eqref{eq:psi with point}.
In addition, the topological constraints of
Definition~\ref{def:condition topologique}
and the map~\eqref{eq:psi with point}
are obviously compatible with pull-backs.
We verify, in Theorem~\ref{th:stability of topological constraints} below,
that they are also
compatible with push-forwards along proper maps.
The proof of this fact
does not depend on Theorem~\ref{th:conditions de krasnov}.
Applying Theorem~\ref{th:stability of topological constraints}
with $\ell=0$
therefore gives an independent proof of Theorem~\ref{th:conditions de krasnov}.

\begin{thm}
\label{th:stability of topological constraints}
Let~$X,Y$ be smooth, irreducible varieties over~$R$.
Let $f:Y\to X$ be a proper morphism.
Let~$\ell$ be an integer. Let $k=\ell+\dim(X)-\dim(Y)$.
Then
\begin{align*}
f_*\big(H^{2\ell}_G(Y(C),\Z(\ell))_0\big) \subseteq H^{2k}_G(X(C),\Z(k))_0\rlap{,}
\end{align*}
where
$f_*:H^{2\ell}_G(Y(C),\Z(\ell)) \to H^{2k}_G(X(C),\Z(k))$
is the push-forward homomorphism defined in~\eqref{eq:pushforward equivariant complex}.
Moreover, the square
\begin{align*}
\xymatrix@R=3ex{
\ar[d]H^{2\ell}_G(Y(C),\Z(\ell))_0 \ar[r]^{\smash[t]{f_*}} & H^{2k}_G(X(C),\Z(k))_0\ar[d] \\
H^\ell(Y(R),\Z/2\Z) \ar[r]^{f_*} & H^k(X(R),\Z/2\Z)\rlap{,}
}
\end{align*}
whose lower horizontal arrow is the push-forward map defined in~\eqref{eq:pushforward}
and whose vertical arrows are
the maps~\eqref{eq:psi with point} associated with~$Y$ and with~$X$,
is commutative.
\end{thm}

In view of
Remark~\ref{rk:topological condition 1-cycles and ak}~(ii),
Theorem~\ref{th:stability of topological constraints} follows
from Proposition~\ref{prop:differentiable rr} below.

\begin{prop}
\label{prop:differentiable rr}
Let $f:Y\to X$ be a proper morphism between smooth, irreducible varieties
over~$R$.
Let~$n$ be an integer. Let $m=n+2\dim(X)-2\dim(Y)$.
The square
\begin{align*}
\xymatrix@R=3ex@C=5em{
H^n_G(Y(C),\Z/2\Z) \ar[r]^{f_*} \ar[d] & H^m_G(X(C),\Z/2\Z) \ar[d] \\
H^n_G(Y(R),\Z/2\Z) \ar@{=}[d] & H^m_G(X(R),\Z/2\Z) \ar@{=}[d] \\
*!<-0em,-1.8ex>\entrybox{\displaystyle\bigoplus_{0\leq p \leq n}\mkern-5muH^p(Y(R),\Z/2\Z)}
\ar@<2.72ex>[r]^{\Sq \circ f_* \circ \Sq^{-1}} &
*!<-0em,-1.8ex>\entrybox{\displaystyle\bigoplus_{0\leq p \leq m}\mkern-5muH^p(X(R),\Z/2\Z)\rlap{,}}
}
\end{align*}
whose vertical maps are the restriction maps composed with the
decompositions~\eqref{eq:canonical decomposition mod 2}, commutes.
In the lower row of this diagram,
the symbol~$f_*$
denotes the direct sum of the push-forward maps
\begin{align}
\label{eq:real push-forward maps}
f_*:H^p(Y(R),\Z/2\Z) \to H^{p+\dim(X)-\dim(Y)}(X(R),\Z/2\Z)
\end{align}
over all $p\geq 0$ (see~\eqref{eq:pushforward}),
while $\Sq=\Sq^0 + \Sq^1 + \Sq^2 + \dots$ is viewed as an automorphism of
the graded rings $\bigoplus_{0\leq p \leq n}H^p(Y(R),\Z/2\Z)$
and $\bigoplus_{0\leq p \leq m}H^p(X(R),\Z/2\Z)$,
and~$\Sq^{-1}$ stands for its inverse.
\end{prop}

\begin{proof}
As a consequence of~\textsection\ref{subsubsec:decomposition omega}
and of the formula~\eqref{eq:cup product},
the diagram
\begin{align*}
\xymatrix@R=3ex@C=7em{
H^n_G(Y(C),\Z/2\Z) \ar[d] \ar[r]^{y\mapsto y \smile \omega^i_{Y(C),\Z/2\Z}} & H^{n+i}_G(Y(C),\Z/2\Z) \ar[d] \\
H^n_G(Y(R),\Z/2\Z) \ar@{=}[d]
\ar[r]^{y\mapsto y \smile \omega^i_{Y(R),\Z/2\Z}}
& H^{n+i}_G(Y(R),\Z/2\Z) \ar@{=}[d] \\
*!<-0em,-1.8ex>\entrybox{\displaystyle\bigoplus_{0\leq p \leq n}\mkern-5muH^p(Y(R),\Z/2\Z)}
&
*!<-0em,-1.8ex>\entrybox{\displaystyle\bigoplus_{0\leq p \leq n+i}\mkern-10muH^p(Y(R),\Z/2\Z)}
\ar@<-2.72ex>[l]_{\mathrm{projection}}
}
\end{align*}
commutes
for any $i\geq 0$ (and similarly with~$X$ instead of~$Y$).
In view of the projection formula
$f_*(y\smile \omega^i_{Y(C),\Z/2\Z})=f_*(y\smile f^* \omega^i_{X(C),\Z/2\Z})=f_*y\smile \omega^i_{X(C),\Z/2\Z}$
(see~\eqref{eq:projection formula equivariant}),
we deduce that in order to prove Proposition~\ref{prop:differentiable rr},
we may replace~$n$ with $n+i$.  In particular,
we may, and will, assume that $n > 2\dim(Y)+1$.

Let us then consider the commutative diagram (without the dotted arrows)
\begin{align*}
\xymatrix@R=2.5ex@C=5em{
*!<-0em,.2ex>\entrybox{\vphantom{(}\displaystyle\smash[b]{\bigoplus_{\mathclap{p\geq 0}}\mkern6muH^p(Y(R),\Z/2\Z)}}
\ar@<-.35ex>[r]^{f_*}
\ar@<-4.25em>@{..>}@/_4pc/[]!(4,2);[dddd]!(5,2)
&
*!<-0em,.2ex>\entrybox{\vphantom{(}\displaystyle\smash[b]{\bigoplus_{\mathclap{p\geq 0}}\mkern6muH^p(X(R),\Z/2\Z)}}
\ar@<4.25em>@{..>}@/^4pc/[]!(-4,2);[dddd]!(-5,2)
\\
H^{n-\dim(Y)}_G(Y(R),\Z/2\Z) \ar@{=}[u] \ar[d]^(.45)\wr \ar[r]^{f_*} & H^{m-\dim(X)}_G(X(R),\Z/2\Z) \ar@{=}[u] \ar[d]^(.45)\wr  \\
H^n_G(Y(C),\Z/2\Z) \ar[d] \ar[r]^{f_*} & H^m_G(X(C),\Z/2\Z) \ar[d] \\
H^n_G(Y(R),\Z/2\Z) \ar@{=}[d] & H^m_G(X(R),\Z/2\Z) \ar@{=}[d] \\
*!<-0em,-1.5ex>\entrybox{\displaystyle\bigoplus_{\mathclap{p\geq 0}}\mkern6muH^p(Y(R),\Z/2\Z)}
\ar@<2.5ex>@{..>}[r]^{\Sq \circ f_* \circ \Sq^{-1}} &
*!<-0em,-1.5ex>\entrybox{\displaystyle\bigoplus_{\mathclap{p\geq 0}}\mkern6muH^p(X(R),\Z/2\Z)\rlap{,}}
}
\end{align*}
in which the upper horizontal~$f_*$ denotes the direct sum of the maps~\eqref{eq:real push-forward maps}
over all $p\geq 0$
(see~\textsection\ref{subsubsec:covariant functoriality decomposition mod 2}
for the commutativity of the top square)
and the vertical arrows of the middle square are the Gysin maps
associated with the inclusions $Y(R)\subseteq Y(C)$ and $X(R)\subseteq X(C)$
(see~\eqref{eq:equivariant purity example}).
As we have seen in the proof of Lemma~\ref{lem:gamma invertible},
these two Gysin maps are isomorphisms since $n>2\dim(Y)+1$ and $m>2\dim(X)+1$,
and the bent arrows which make the diagram commute
are given by multiplication by the classes
$\gamma_Y \in \bigoplus_{p\geq 0} H^p(Y(R),\Z/2\Z)$
and
$\gamma_X \in \bigoplus_{p\geq 0} H^p(X(R),\Z/2\Z)$
that one obtains by applying Definition~\ref{def:gamma} to~$Y$ and to~$X$.
To prove that the diagram remains commutative with the horizontal dotted arrow,
it therefore suffices to check that
\begin{align}
\label{eq:diffrr 1st form}
\Sq(f_*(\Sq^{-1}(y))) = \gamma_X  f_*(\gamma_Y^{-1} y)
\end{align}
for any $y \in \bigoplus_{p\geq 0} H^p(Y(R),\Z/2\Z)$,
or, equivalently, that
\begin{align}
\label{eq:diffrr 2nd form}
f_*\big(\Sq^{-1}(y)\mkern2mu\Sq^{-1}(\gamma_Y)\big)=\Sq^{-1}(f_*y)\mkern2mu\Sq^{-1}(\gamma_X)
\end{align}
for any $y \in \bigoplus_{p\geq 0} H^p(Y(R),\Z/2\Z)$.

This last equality is a particular case of the relative variant of Wu's theorem
due to Atiyah and Hirzebruch~\cite{atiyahhirzebruchrrdiff}.
Namely, when $R=\R$ and~$X$ and~$Y$ are proper,
this is \emph{op.\ cit.}, Satz~3.2
applied to $\lambda=\Sq^{-1}$;
indeed,
with the notation of~\emph{loc.\ cit.},
one has $\mathrm{Wu}(\Sq,Y)=\Sq^{-1}(\gamma_Y)$
and $\mathrm{Wu}(\Sq,X)=\Sq^{-1}(\gamma_X)$,
according to Thom's formula \cite[p.~91]{milnorstasheff}
and to
Remark~\ref{rks:short exact sequences mod 2}~(i)
(see \cite[Theorem~2.1]{krasnovequivariant}).
When~$R=\R$ but~$X$ and~$Y$ need not be proper,
the proof given in \cite[\textsection3.4]{atiyahhirzebruchrrdiff}
goes through verbatim
once one remarks
that if~$i$ denotes the composition of a closed embedding $Y(\R) \subset \R^N$
given by Whitney's theorem (see \cite[Chapter~2, Theorem~2.14]{Hirsch})
with an embedding of~$\R^N$ into
the $N$\nobreakdash-dimensional sphere~$\mathbf{S}^N$,
the map $(f,i):Y(\R) \to X(\R) \times \mathbf{S}^N$
is a closed embedding.

By a spreading out argument
entirely similar to the one used in~\cite[\textsection7]{delfsknebuschonthehomology},
one can deduce
the validity of~\eqref{eq:diffrr 1st form}
over an arbitrary real closed field~$R$
from the validity of~\eqref{eq:diffrr 1st form}
for all~$X$, $Y$, $f$, $y$ defined over~$\R$.
(The point is that~\eqref{eq:diffrr 1st form} is of a purely cohomological nature.
Spreading out allows one to deduce cohomological statements over arbitrary real closed fields
from the same statements over~$\R$. For statements on algebraic cycles, the situation
is quite different, as we will see in \cite[\S\ref*{BW2-nonarchimedeansection}]{bwpartie2}.)

To establish~\eqref{eq:diffrr 1st form} without spreading out,
one can also, in the projective case,
apply the Riemann--Roch theorem of
Panin and Smirnov \cite[Theorem~2.5.4]{paninrroriented}
to semi-algebraic cohomology with~$\Z/2\Z$ coefficients.
\end{proof}

\section{The real integral Hodge conjecture}
\label{sec:realIHC}

We formulate, in this section, a real analogue of the integral Hodge conjecture.

\subsection{Reminders on the complex integral Hodge conjecture}
\label{par:complexIHC}

Let $k\geq 0$.  Let~$X$ be a smooth and proper variety of pure dimension~$d$ over~$\C$.

A class $\alpha \in H^{2k}(X(\C),\Z(k))$ is \emph{Hodge} if its image in $H^{2k}(X(\C),\C)$
is of type~$(k,k)$ in the Hodge decomposition.
Let $\Hdg^{2k}(X(\C),\Z(k)) \subseteq H^{2k}(X(\C),\Z(k))$
denote the subgroup of Hodge classes.
Classes of algebraic cycles belong to
$\Hdg^{2k}(X(\C),\Z(k))$.
By definition, the \emph{integral Hodge conjecture for codimension~$k$ cycles on~$X$} holds if and only if
the induced map $\CH^k(X)\to \Hdg^{2k}(X(\C),\Z(k))$ is surjective.

The integral Hodge conjecture holds for $k=0$ or $k\geq d$ (trivial)
and for
$k=1$ (this is the Lefschetz~$(1,1)$ theorem).
For all other values of~$k$ and~$d$, it fails in general.
The first counterexamples were discovered by Atiyah and Hirzebruch~\cite{atiyahhirzebruch}.
These are counterexamples for $k=2$ and $d\geq 7$.
In the case of $1$\nobreakdash-cycles,
Koll\'ar~\cite{trentoexamples} has shown that the integral Hodge conjecture
fails for
very general
hypersurfaces of degree~$\delta$ in~$\P^4_\C$
for some~$\delta$.
The smallest value of~$\delta$ known to yield
a counterexample is~$48$
(see~\cite[\textsection5]{totarocontreexemples}),
though it is expected that any $\delta\geq 6$ should yield one:
indeed, Griffiths and Harris~\cite[p.~32]{griffithsharrisnl}
conjecture that if $\delta\geq 6$,
the degree of any curve on
a very general hypersurface of degree~$\delta$ in~$\P^4_\C$
is a multiple of~$\delta$.

By blowing up~$\P^d_\C$ along a very general hypersurface of degree~$\delta$
in $\P^4_\C \subset \P^d_\C$,
one finds counterexamples to the integral Hodge conjecture
among rational varieties for all $k \in \{3, \dots, d-2\}$
(see~\cite[p.~113]{soulevoisin}).
On the other hand, for $k \in \{2,d-1\}$,
the integral Hodge conjecture for codimension~$k$ cycles on~$X$
is a birational invariant (see~\cite[Lemma~15]{voisinsomeaspects}) and in particular it holds for rational
varieties.
The integral Hodge conjecture can nevertheless fail for $k=2$ 
among rationally connected varieties
\cite[Th\'eor\`eme~1.3]{ctvoisin}, even for rationally connected fourfolds \cite[Corollary~1.6]{Stefan}, although it holds for smooth cubic fourfolds \cite[Theorem 18]{voisinsomeaspects}.
For $k=d-1$, the following question is open:

\newcommand{\citequestionvoisin}{\cite[Question~16]{voisinsomeaspects}}
\begin{question}[Voisin~\citequestionvoisin]
\label{mainquestionC}
Let~$X$ be a smooth, proper, rationally connected variety over~$\C$.
Does~$X$ satisfy the integral Hodge conjecture for $1$\nobreakdash-cycles?
\end{question}

Voisin~\cite[Theorem~2]{voisinthreefolds} proves that a complex projective threefold satisfies the integral Hodge conjecture if it is uniruled or Calabi--Yau\footnote{We mean Calabi--Yau in the sense that $K_X\simeq \sO_X$ and $H^1(X,\sO_X)=H^2(X,\sO_X)=0$. Totaro has announced a proof of the integral Hodge conjecture
for complex projective threefolds~$X$ such that $K_X \simeq \sO_X$.  Abelian threefolds are dealt with in~\cite[Corollary~3.1.9]{Grabowski}. The hypothesis that $K_X \simeq \sO_X$ cannot be weakened to $K_X$ being torsion in view of \cite[Theorem 0.1]{John}.}.
This answers
Question~\ref{mainquestionC}
in the affirmative when $\dim(X)=3$.
In higher dimensions, there are partial results for Fano varieties
(see~\cite{HoeringVoisin, Floris}).
In addition,
using a theorem of Schoen~\cite{schoen},
Voisin~\cite[Theorem~1.6]{voisinremarks} shows
that the Tate conjecture for all surfaces over finite fields would
imply a positive answer to Question~\ref{mainquestionC}.

\subsection{The real formulation}
\label{par:realIHC}

Let $k\geq 0$.
Let~$X$ be a smooth and proper variety over~$\R$.
According to Theorem~\ref{th:conditions de krasnov},
the image of the equivariant cycle class map
$\CH^k(X) \to H^{2k}_G(X(\C),\Z(k))$
is contained in the subgroup
$H^{2k}_G(X(\C),\Z(k))_0$
introduced in
Definition~\ref{def:condition topologique}.
As recalled in~\textsection\ref{par:complexIHC},
it is also contained in the inverse image
$\Hdg^{2k}_G(X(\C),\Z(k)) \subseteq H^{2k}_G(X(\C),\Z(k))$
of $\Hdg^{2k}(X(\C),\Z(k))$
by the natural map
$H^{2k}_G(X(\C),\Z(k))\to H^{2k}(X(\C),\Z(k))$.
Hence we obtain a map
\begin{align}
\label{eq:cycle class map to hdgzero}
\cl:\CH^k(X) \to \Hdg^{2k}_G(X(\C),\Z(k))_0
\end{align}
to the subgroup
$\Hdg^{2k}_G(X(\C),\Z(k))_0=\Hdg^{2k}_G(X(\C),\Z(k)) \cap H^{2k}_G(X(\C),\Z(k))_0$
of those classes in $H^{2k}(X(\C),\Z(k))$ which satisfy both the Hodge condition and the topological condition.

\begin{defn}
Let $k\geq 0$.
Let~$X$ be a smooth and proper variety over~$\R$.
We say that
\emph{the real integral Hodge conjecture for codimension~$k$ cycles on~$X$} holds
if the map~\eqref{eq:cycle class map to hdgzero}
is surjective.
\end{defn}

Let us stress that this property can fail, just as the integral Hodge conjecture can fail for varieties
over~$\C$ (see~\textsection\ref{par:complexIHC}).  We shall provide examples of geometrically
connected varieties over~$\R$ for which it fails
in~\textsection\ref{subsec:cycle theoretic obs},
see also Example~\ref{ex:failure ihc omega} below.

We note that if $H^q(X,\Omega_X^p)=0$ for all $p,q$ such that $p+q=2k$, $(p,q)\neq (k,k)$, then
$\Hdg^{2k}_G(X(\C),\Z(k))_0=H^{2k}_G(X(\C),\Z(k))_0$ and the formulation of the real
integral Hodge conjecture makes sense
over an arbitrary real closed field (where Hodge theory is not readily available).

\begin{defn}
\label{def:ihc real closed field}
Let $k\geq 0$.  Let~$X$ be a smooth and proper variety over a real closed field~$R$.
Assume that $H^q(X,\Omega_X^p)=0$ whenever $p+q=2k$ and $(p,q)\neq (k,k)$.
We say that
\emph{the real integral Hodge conjecture for codimension~$k$ cycles on~$X$} holds
if the equivariant cycle class map $\cl:\CH^k(X) \to H^{2k}_G(X(C),\Z(k))_0$
is surjective.
\end{defn}

\begin{rmks}
(i)
Considering equivariant cohomology \`a la Bredon, rather than \`a la Borel, leads to
a factorisation
$\CH^k(X) \to H^{2k,k}_{\Br}(X(\C),\underline{\Z}) \to H^{2k}_G(X(\C),\Z(k))$
of the equivariant cycle class map~$\cl$
(see~\cite[\textsection\textsection1--2]{Bredonquadrics}).
It would be interesting to
determine what additional constraints on the image
of~$\cl$, if any, result from this.

(ii)
By a norm argument,
the map~\eqref{eq:cycle class map to hdgzero} tensored with~$\Q$
is surjective if~$X_\C$ satisfies the Hodge conjecture.
Thus,  there would be no point in formulating a real variant of the
Hodge conjecture with rational coefficients.
\end{rmks}

An intriguing feature of the real integral Hodge conjecture,
one with no analogue in the complex setting,
is the existence of canonical ``constant'' cohomology classes.
Namely, let~$X$ be a smooth and proper variety over a real closed field~$R$
and let~$k \geq 0$ be even,
so that $\Z(2k)=\Z(k)$.
We can view $\omega^{2k} \in H^{2k}_G(X(C),\Z(2k))$
(see~\textsection\ref{subsubsec:omega}) as an element of $H^{2k}_G(X(C),\Z(k))$.
This class is Hodge if $R=\R$, being torsion,
and it belongs to $H^{2k}_G(X(C),\Z(k))_0$ if (and only if) $X(R)=\emptyset$
(see~\textsection\ref{subsubsec:decomposition omega}).
If $X(R)=\emptyset$, the integral Hodge conjecture on~$X$
therefore implies that~$\omega^{2k}$ is algebraic, \emph{i.e.}, is the class
of an algebraic cycle.
Determining when~$\omega^{2k}$ is algebraic is an interesting problem in its own right.

\begin{example}
\label{ex:failure ihc omega}
Let~$X$ be the smooth projective anisotropic quadric of dimension~$n$ over~$R$.
By \cite[Proposition~3.3, Proposition~3.5]{Hilbert17}
and by comparison with $2$\nobreakdash-adic cohomology,
the class~$\omega^4$
is algebraic if and only if~$-1$ is a sum of~$7$ squares in the function field~$R(X)$.
By Pfister~\cite[Satz~5]{PfisterStufe},
such is the case if and only if $n\leq 6$.
\end{example}

In this example, the real integral Hodge conjecture for codimension~$2$ cycles on~$X$ fails
if $n \geq 7$.  This can be compared with~\cite[Th\'eor\`eme~1.3]{ctvoisin}. In constrast, the following question is open.

\begin{question}
Does there exist a smooth and proper variety
of odd dimension~$d$, over a real closed field~$R$,
such that $X(R)=\emptyset$
and $\omega^{2d-2}$ is not algebraic?
\end{question}

\begin{rmks}
\label{rk:covers hknotalg}
(i)
Let~$X$ be a smooth and proper variety over~$\R$
and let $k\geq 0$.
If the restriction of the
map~\eqref{eq:psi with point}
to the subgroup $\Hdg^{2k}_G(X(\C),\Z(k))_0$ is not surjective,
then $H^k_\alg(X(\R),\Z/2\Z)\neq H^k(X(\R),\Z/2\Z)$.
As far as we are aware,
this simple remark explains
all of the examples that appear in the literature
of real varieties~$X$
and integers~$k$ such that $H^k_\alg(X(\R),\Z/2\Z)\neq H^k(X(\R),\Z/2\Z)$.
We give an example of a different kind in~\textsection\ref{subsec:cycle theoretic obs}
(see Example~\ref{ex:a la totaro avec points}); such an example necessarily underlies
a defect of the real integral Hodge conjecture for codimension~$k$ cycles on~$X$.

(ii)
Let us justify the assertion made in~(i) about the existing literature.
Examples with $k=1$,
such as those of~\cite{rislerhomologie}
and~\cite{silholabound},
fall under the scope of Remark~\ref{rk:covers hknotalg}~(i)
since the map~\eqref{eq:cycle class map to hdgzero} is surjective
when $k=1$ according to Proposition~\ref{prop:real(1,1)} below.
Next, the examples
given in~\cite[Theorem~6.9]{akbulutkingsubmanifolds}
and in \cite[Theorem~3.1]{kucharzonhomology}
are examples in which even the map~\eqref{eq:psi with point} is not surjective.
Indeed, these examples only depend
on two properties of the subgroups
$H^k_{\alg}(X(\R),\Z/2\Z)$:
their stability
under cup products, pull-backs and proper push-forwards,
and the fact that
$H^0_\alg(X(\R),\Z/2\Z)\neq H^0(X(\R),\Z/2\Z)$
when~$X$ is connected while~$X(\R)$ is not;
as it turns out, the images of~\eqref{eq:psi with point} form a system of subgroups
that enjoy these properties as well
(see~\textsection\ref{subsubsec:compatibility with pushforwards}).
The example of~\cite{benedettidedo} is an example of a class in $H^2(X(\R),\Z/2\Z)$
whose square cannot be lifted to $H^4(X(\R),\Z)$
(see \cite{teichner}, especially the end of~\textsection2);
by Remark~\ref{rk:topological condition 1-cycles and ak}~(iii),
the map~\eqref{eq:psi with point} again fails to be surjective in this case.
Finally,
in the examples furnished by~\cite[Theorem~2.1]{kucharzalgeqhom},
the map~\eqref{eq:psi with point} fails once more to be surjective,
as we explain in Remark~\ref{rmk:kucharz alg explained} below.
\end{rmks}

\subsection{First positive results}
\label{subsec:first positive results}

On a smooth and proper variety~$X$ of dimension~$d$ over a real closed field~$R$,
the real integral Hodge conjecture for codimension~$k$ cycles
holds if $k=0$ or $k>d$.
For $k=0$, this is a trivial assertion;
for $k>d$, the restriction map
$H^{2k}_G(X(C),\Z(k)) \to H^{2k}_G(X(R),\Z(k))$
is injective
(see~\textsection\ref{subsubsec:cohomological dimension}) with kernel $H^{2k}_G(X(C),\Z(k))_0$, so that the target of (\ref{eq:cycle class map to hdgzero}) vanishes.
Let us now consider, in~\textsection\ref{subsubsec:divisors} and~\textsection\ref{subsubsec:zerocycles} below,
the more interesting cases where $k=1$ or $k=d$.

\subsubsection{Divisors}
\label{subsubsec:divisors}

It was observed by Krasnov that the Lefschetz~$(1,1)$ theorem holds in the real setting as well.
The topological constraints
provided by Theorem~\ref{th:conditions de krasnov}
do not play any role here,
as $H^2_G(X(C),\Z(1))_0=H^2_G(X(C),\Z(1))$.

\begin{prop}[Krasnov]
\label{prop:real(1,1)}
Any smooth and proper variety over~$\R$ satisfies
the real integral Hodge conjecture for codimension~$1$ cycles.
\end{prop}

This can be checked by mimicking the usual proof of the Lefschetz~$(1,1)$ theorem and noting
that the exponential short exact sequence on~$X(\C)$ is $G$\nobreakdash-equivariant.
We refer the reader to~\cite[Proposition~1.3.1]{krasnovcharacteristicclasses},
\cite[Proposition~3.2]{mangoltevanhamel}, \cite[Chapter~IV,
Theorem~4.1]{vanhamelthese}.
Over a real closed field, one has the following substitute:

\begin{prop}
\label{prop:real(1,1)nonarch}
Let~$X$ be a smooth variety over a real closed field~$R$. The cokernel of
the equivariant cycle class map $\cl:\Pic(X)\to H^2_G(X(C),\Z(1))$ is torsion-free.
Assume, moreover, that~$X$ is proper and that $H^2(X,\sO_X)=0$.
Then~$X$ satisfies the real integral Hodge conjecture for codimension~$1$ cycles in
the sense of Definition~\ref{def:ihc real closed field}.
\end{prop}

\begin{proof}
Let $E=\Coker\mkern2mu\big(\cl:\Pic(X)\to H^2_G(X(C),\Z(1))\big)$.
Let $\Br(X)=H^2_\et(X,\Gm)$ and $T(\Br(X))=\Hom(\Q/\Z,\Br(X))$.
Recall,
from~\textsection\ref{subsubsec:sacohomology},
the canonical identification
$H^2_G(X(C),\Z(1)) \otimes_\Z \widehat\Z=H^2_\et(X,\widehat\Z(1))$,
and
consider the commutative diagram
\begin{align*}
\xymatrix@R=3ex{
\Pic(X) \otimes_\Z \widehat\Z \ar[d] \ar[r] & H^2_G(X(C),\Z(1)) \otimes_\Z \widehat\Z \ar@{=}[d] \ar[r] & E \otimes_\Z \widehat\Z \ar[r] \ar[d] & 0 \\
\displaystyle\varprojlim_{n\geq 1} \big(\Pic(X)/n\Pic(X)\big) \ar[r] & H^2_\et(X,\widehat\Z(1)) \ar[r] & T(\Br(X)) \ar[r] & 0\rlap{,}
}
\end{align*}
in which the bottom row comes from the Kummer exact sequence in \'etale cohomology
and
the vertical map on the right is determined by the rest of the diagram.
The zigzag map $\varprojlim \Pic(X)/n\Pic(X) \to E \otimes_\Z \widehat \Z$
vanishes since its composition
with the projection $E \otimes_\Z \widehat\Z \to E/nE$ vanishes for every~$n$ while~$E$ is a finitely generated abelian group. The vertical map on the right is therefore an isomorphism.
As the abelian group $T(\Br(X))$ is torsion-free and as~$E$ is finitely generated,
it follows that~$E$ is torsion-free.
If~$X$ is proper and satisfies $H^2(X,\sO_X)=0$,
then $\NS(X_C)\otimes_\Z\Ql=H^2_\et(X_C,\Ql)$
for any~$\ell$
(by Hodge theory if $C=\C$, by the Lefschetz principle in general)
and hence $T(\Br(X_C))=0$
(see~\cite[II, Corollaire~3.4]{grothbrauer}).
As $T(\Br(X))$ is torsion-free, a norm argument
now implies that $T(\Br(X))=0$. It follows that $E\otimes_\Z\widehat\Z=0$, hence that $E=0$.
\end{proof}

\subsubsection{Zero-cycles}
\label{subsubsec:zerocycles}

On a smooth and proper variety of pure dimension $d$ over $R$, one may consider the real integral Hodge conjecture for zero-cycles in the sense of Definition \ref{def:ihc real closed field} since $H^q(X,\Omega^p_X)=0$ for $p>d$ or $q>d$. 
We show that it always holds.

\begin{prop}
\label{prop:hodgereel0cycles}
Any smooth and proper variety of pure dimension~$d$ over a real closed field satisfies
the real integral Hodge conjecture for codimension~$d$ cycles.
\end{prop}

\begin{proof}
Pick a point~$x_i$ in each semi-algebraic connected component of~$X(R)$
and a point~$y_j$ in each semi-algebraic connected component of~$X(C)$.
By Theorem~\ref{th:conditions de krasnov}
and the next lemma, the classes $\cl(x_i)$ and $\cl(N_{C/R}(y_j))$
generate~$H^{2d}_G(X(C),\Z(d))_0$.
\end{proof}

\begin{lem}
\label{lem:H2dreel}
Let~$X$ be a smooth and proper variety of pure dimension~$d$ over a real closed field~$R$.
\begin{enumerate}[(i)]
\item
Sending a connected component of~$X_C$ to the cycle class of one of its closed points
defines a canonical isomorphism $\Z^{\pi_0(X_C)}\isoto H^{2d}(X(C),\Z(d))$.
\item
The norm map $H^{2d}(X(C),\Z(d))\to H^{2d}_G(X(C),\Z(d))$ from~\eqref{eq:real-complex long 10}
takes its values in the subgroup $H^{2d}_G(X(C),\Z(d))_0$ and fits into an exact sequence
\begin{align*}
\xymatrix{
H^{2d}(X(C),\Z(d))\ar[r]& H^{2d}_G(X(C),\Z(d))_0\ar[r]& H^d(X(R),\Z/2\Z)\ar[r]&0\rlap{,}
}
\end{align*}
where
the map on the right
is the map~\eqref{eq:psi with point} for $k=d$
and where the map on the left is injective if~$X$ is geometrically connected.
\end{enumerate}
\end{lem}

\begin{proof}
Assertion~(i) is a part of Poincar\'e duality
and follows from \cite[Chapter~VI, Theorem~11.1~(a)]{milneet}
by comparison with $\ell$\nobreakdash-adic cohomology for every~$\ell$.
Alternatively, although we do not do it here,
it could be deduced from the statement of Poincar\'e
duality given in Proposition~\ref{prop:verdier}.

The composition of the norm map which appears in~(ii)
with the natural map $H^{2d}_G(X(C),\Z(d)) \to H^{2d}(X(C),\Z(d))$
coincides, if~$X$ is geometrically irreducible,
with multiplication by~$2$ on the torsion-free group $H^{2d}(X(C),\Z(d))=\Z$.
It is therefore injective in this case.
The rest of the assertion follows from the commutative diagram
\begin{align*}
\xymatrix@R=3ex{
H^{2d}(X(C),\Z(d)) \ar[r] & H^{2d}_G(X(C),\Z(d)) \ar[r] \ar[d] & H^{2d+1}_G(X(C),\Z(d+1)) \ar[d]^(.45)\wr \ar[r] & 0\\
& H^{2d}_G(X(R),\Z(d)) \ar[r] & H^{2d+1}_G(X(R),\Z(d+1))\rlap{,}
}
\end{align*}
whose first row is exact and comes from~\eqref{eq:real-complex long 10}.
Indeed,
the vertical map on the right is an isomorphism (see~\textsection\ref{subsubsec:cohomological dimension})
and the bottom horizontal map is compatible with the decompositions~\eqref{eq:canonical decomposition}
(see~\textsection\ref{subsubsec:reduction modulo 2} and~\textsection\ref{subsubsec:effect}).
\end{proof}

\begin{rmk}
\label{rmk:kucharz alg explained}
Thanks to Proposition~\ref{prop:hodgereel0cycles}, we can now
extend \cite[Theorem~2.1]{kucharzalgeqhom} to the following purely cohomological statement,
as promised in Remark~\ref{rk:covers hknotalg}~(ii):

Let~$X$ be a smooth and proper
variety of pure dimension~$d$ over a real closed field~$R$.
Let~$k$ be an integer. Let $\alpha_k \in H^k(X(R),\Z/2\Z)$
and $\beta_{d-k} \in H^{d-k}(X(R),\Z/2\Z)$.
If~$\alpha_k$ and~$\beta_{d-k}$ can be lifted,
by the corresponding maps~\eqref{eq:psi with point},
to $\alpha \in H^{2k}_G(X(C),\Z(k))_0$
and $\beta \in H^{2d-2k}_G(X(C),\Z(d-k))_0$
such that the image of~$\beta$ in $H^{2d-2k}(X(C),\Z(d-k))$ vanishes
(as is the case if $\beta=\cl(y)$ for a cycle~$y$ algebraically equivalent to~$0$),
then $\deg(\alpha_k \smile \beta_{d-k})=0$, where $\deg:H^d(X(R),\Z/2\Z) \to \Z/2\Z$
denotes the total degree.

Indeed, letting $\gamma = \alpha \smile \beta$,
we have $\gamma \in H^{2d}_G(X(C),\Z(d))_0$
and $\gamma_d = \alpha_k \smile \beta_{d-k}$
by~\textsection\ref{subsubsec:compatibility with pushforwards}.
By Proposition~\ref{prop:hodgereel0cycles},
there exists a zero-cycle~$z$
such that $\gamma=\cl(z)$.
It must have degree~$0$ since the image of~$\gamma$
in $H^{2d}(X(C),\Z(d))$ vanishes.
Hence $\deg(\gamma_d)=0$.
\end{rmk}

\subsubsection{Birational invariance}
\label{subsubsec:birational invariance}

The integral Hodge conjecture for codimension~$k$ cycles on a smooth and proper complex algebraic
variety of pure dimension~$d$
is a birational invariant when $k=2$ or $k=d-1$ (see~\cite[Lemma~15]{voisinsomeaspects}).
The real integral Hodge conjecture enjoys the same birational invariance property.

\begin{prop}
\label{prop:birinvIHC}
Let~$R$ be a real closed field.
Let~$X$ denote a smooth and proper variety of pure dimension~$d$ over~$R$.
Let $k \in \{2,d-1\}$.
If $R=\R$,
the group
\begin{align*}
\Coker\big(\CH^k(X) \to \Hdg^{2k}_G(X(\C),\Z(k))_0\big)
\end{align*}
is a birational invariant of~$X$.
If $k=d-1$,
so is the group
\begin{align*}
\Coker\big(\CH^k(X) \to H^{2k}_G(X(C),\Z(k))_0\big)\rlap{,}
\end{align*}
for any~$R$.
\end{prop}

\begin{proof}
By Hironaka's theorem \cite[\textsection5]{hironaka},
it suffices to show that these groups are invariant by any blowing-up with smooth centre.
This, in turn, follows from the next lemma and from the validity of the real integral
Hodge conjecture for divisors and for zero-cycles (Propositions~\ref{prop:real(1,1)},
\ref{prop:real(1,1)nonarch}, and~\ref{prop:hodgereel0cycles}) on the centre of the blowing-up.
\end{proof}

\begin{lem}
Let $\pi:X'\to X$ denote the blowing-up of
a smooth irreducible closed subvariety $Z \subset X$ of codimension~$r$.
Let $\iota:E \hookrightarrow X'$ denote the inclusion of the exceptional divisor
and $\tau:E\to Z$ the projection.
Let~$k$ be an integer.
Let~$\xi$
denote the image, by $\cl:\Pic(E)\to H^2_G(E(C),\Z(1))$,
of the class of~$\sO_E(1)$ (the tautological line bundle of the projective bundle~$\tau$).
 There is a natural isomorphism
\begin{align}
H^{2k}_G(X(C),\Z(k))_0 \oplus \bigoplus_{i=k-r+1}^{k-1} H^{2i}_G(Z(C),\Z(i))_0 \to
H^{2k}_G(X'(C),\Z(k))_0
\end{align}
and, if $R=\R$, a natural isomorphism
\begin{align}
\Hdg^{2k}_G(X(C),\Z(k))_0 \oplus \bigoplus_{i=k-r+1}^{k-1} \Hdg^{2i}_G(Z(C),\Z(i))_0 \to
\Hdg^{2k}_G(X'(C),\Z(k))_0\rlap{,}
\end{align}
both given by $\alpha \oplus \bigoplus (\beta_i)_i \mapsto \pi^*\alpha + \sum_i \iota_*(\xi^{k-i-1} \smile \tau^*\beta_i)$.
\end{lem}

\begin{proof}
As is well known, the above formula defines
an isomorphism of abelian groups
\begin{align}
\label{eq:cohomology of blowup}
H^{2k}_G(X(C),\Z(k)) \oplus \bigoplus_{i=k-r+1}^{k-1} H^{2i}_G(Z(C),\Z(i)) \to
H^{2k}_G(X'(C),\Z(k))
\end{align}
(see \cite[Th\'eor\`eme~2.2, Th\'eor\`eme~1.2]{katzsga72}).
When $R=\R$, this isomorphism respects the Hodge decompositions.
It only remains to be checked that
for any $\alpha \oplus \bigoplus (\beta_i)_i$
in the left-hand side of~\eqref{eq:cohomology of blowup},
if one sets $\gamma=\pi^*\alpha + \sum_i \iota_*(\xi^{k-i-1} \smile \tau^*\beta_i)$,
then $\gamma \in H^{2k}_G(X'(C),\Z(k))_0$ if and only if
$\alpha \in H^{2k}_G(X(C),\Z(k))_0$ and $\beta_i \in H^{2i}_G(Z(C),\Z(i))_0$ for every~$i$.
As cup products, push-forwards and pull-backs preserve
the condition of Definition~\ref{def:condition topologique}
(see~\textsection\ref{subsubsec:compatibility with pushforwards}
and Theorem~\ref{th:stability of topological constraints}),
the converse implication is clear.
For the direct implication, let
\begin{align*}
\delta=\sum_{i=k-r+1}^{k-1} \xi^{k-i}\smile \tau^*\beta_i \in H^{2k}_G(E(C),\Z(k))
\end{align*}
and let
us assume that
$\gamma \in H^{2k}_G(X'(C),\Z(k))_0$.
As $\alpha=\pi_*\gamma$,
Theorem~\ref{th:stability of topological constraints}
shows that
$\alpha \in H^{2k}_G(X(C),\Z(k))_0$.
As $\delta=\iota^*(\gamma-\pi^*\alpha)$,
it follows that $\delta \in H^{2k}_G(E(C),\Z(k))_0$.
To prove that $\beta_i \in H^{2i}_G(Z(C),\Z(i))_0$
for all~$i$, we now argue by induction on~$i$.
Assuming that $i_0\in \{k-r+1,\dots,k-1\}$
is such that $\beta_i \in H^{2i}_G(Z(C),\Z(i))_0$ for all $i<i_0$,
let us check that $\beta_{i_0} \in H^{2i_0}_G(Z(C),\Z(i_0))_0$.
By the projection formula~\eqref{eq:projection formula equivariant},
we have
\begin{align}
\label{eq:tauxi proof birational invariance}
\tau_*(\xi^{i_0-k+r-1}\smile\delta) = \beta_{i_0} + \sum_{i=k-r+1}^{i_0-1} \tau_*\xi^{i_0-i+r-1}\smile \beta_i
\end{align}
since $\tau_*\xi^m=0$ for $m<r-1$ and $\tau_*\xi^{r-1}=1\in H^0_G(Z(C),\Z)$.
By the compatibility of Definition~\ref{def:condition topologique} with
push-forwards and cup products
(see~\textsection\ref{subsubsec:compatibility with pushforwards}
and Theorem~\ref{th:stability of topological constraints}),
the left-hand side of~\eqref{eq:tauxi proof birational invariance}
belongs to $H^{2i_0}_G(Z(C),\Z(i_0))_0$.
By our assumption on~$i_0$,
so does the second term of the right-hand side;
hence $\beta_{i_0} \in H^{2i_0}_G(Z(C),\Z(i_0))_0$.
\end{proof}

\subsubsection{Projective spaces}
\label{subsubsecps}
We finally consider the case of projective spaces.

\begin{prop}
\label{propIHCprojectivespace}
For all integers $0\leq k\leq d$, the real integral Hodge conjecture holds for codimension $k$ cycles on $\P^d_R$.
\end{prop}

\begin{proof}
By \cite[Theorem~3.3~(b)]{Fulton}, the Chow group $\CH^d(\P^d_R\times \P^d_R)$ is generated by the classes of $\P^i_R\times \P^{d-i}_R$ for $0\leq i\leq d$ . The class of the diagonal $\Delta_{\P^d_R}$ may then be determined by computing the intersection numbers $[\Delta_{\P^d_R}]\cdot[\P^i_R\times \P^{d-i}_R]=1$:
\begin{equation}
\label{correspdiag}
[\Delta_{\P^d_R}]=\sum_{i=0}^d[\P^i_R\times \P^{d-i}_R]\in \CH^d(\P^d_R\times \P^d_R).
\end{equation}
Let $\alpha\in H^{2k}_G(\P^d(C),\Z(k))_0$.
We let the correspondence (\ref{correspdiag}) act on $\alpha$.
 For $i<k$, one has $[\P^i_R\times \P^{d-i}_R]_*\alpha=0$. Indeed, $H^{2k}_G(\P^i(C),\Z(k)) \to H^{2k}_G(\P^i(R),\Z(k))$
is injective by~\textsection\ref{subsubsec:cohomological dimension}, so that $H^{2k}_G(\P^i(C),\Z(k))_0=0$. For $i>k$, one also has $[\P^i_R\times \P^{d-i}_R]_*\alpha=0$ in view of the vanishing of $H^{2k-2i}_G(\P^{d-i}(C),\Z(k-i))$ in this range. We deduce:
$$\alpha=[\Delta_{\P^d_R}]_*\alpha=[\P^k_R\times \P^{d-k}_R]_*\alpha.$$
The latter class is algebraic because the whole group $H^{2k}_G(\P^k(R),\Z(k))_0$ consists of algebraic classes by Proposition \ref{prop:hodgereel0cycles}.
\end{proof}

\subsection{Main question}

The results of Voisin~\cite{voisinthreefolds,voisinremarks}
on the integral Hodge conjecture for complex varieties
(discussed in~\textsection\ref{par:complexIHC})
and the good basic properties of the real integral Hodge conjecture
(see~\textsection\ref{subsec:first positive results})
lead us to raise the following question:

\begin{question}
\label{mainquestion}
Let~$X$ be a smooth, proper and geometrically irreducible
variety over~$\R$.
If the underlying complex variety~$X_\C$ is a uniruled threefold, a Calabi--Yau threefold, or a rationally connected variety,
does~$X$ satisfy the real integral Hodge conjecture for $1$\nobreakdash-cycles?
\end{question}

In~\textsection\ref{sec:onecycles} and~\textsection\ref{sec:blochogus} below,
we explore the meaning and implications of a positive answer to this question.
In~\cite{bwpartie2}, we shall provide positive answers for various classes of uniruled threefolds (and, in higher dimensions, for iterated conic bundles over these threefolds, or over arbitrary surfaces).
As far as we know, the answer to Question~\ref{mainquestion} may be
in the affirmative in full generality.

For Calabi--Yau threefolds and for rationally connected varieties,
one might wish to extend
Question~\ref{mainquestion} to arbitrary real closed fields,
using Definition~\ref{def:ihc real closed field}.
However, we shall give in~\cite{bwpartie2} various examples
which show that the resulting question would have a negative
answer,
at least for rationally connected threefolds with a rational point
and for (simply connected) Calabi--Yau threefolds with no rational point.  (It is still
conceivable that
over an arbitrary
real closed field,
rationally
connected varieties with no rational point
may satisfy the real integral Hodge
conjecture for $1$\nobreakdash-cycles.)

\section{Cycles of dimension~\texorpdfstring{$1$}{1}}
\label{sec:onecycles}

The goal of this section is to investigate the relationship between the real integral Hodge conjecture for $1$\nobreakdash-cycles, on the one hand,
and the complex integral Hodge conjecture for $1$\nobreakdash-cycles,
the study of the Borel--Haefliger cycle class map,
and the question of the existence of curves of even geometric genus in a given variety, on the other hand.
For the whole of~\textsection\ref{sec:onecycles},
we fix a real closed field~$R$
and a smooth, proper and geometrically irreducible variety~$X$ over~$R$,
of dimension~$d$.

\subsection{Summary}

Before examining the integral Hodge conjecture for $1$\nobreakdash-cycles,
we must study the target $H^{2d-2}_G(X(C),\Z(d-1))_0$ of the equivariant cycle class map.

A key role will be played by two homomorphisms
\begin{align*}
\psi:H^{2d-2}_G(X(C),\Z(d-1))_0\to \Mpsi
\qquad\text{ and }\qquad
\psi':\Pic(X_C)^G[2^\infty]\to \Mpsi^*\rlap{,}
\end{align*}
where~$\Mpsi$ is a $2$\nobreakdash-torsion abelian group and $\Mpsi^*=\Hom(\Mpsi,\Z/2\Z)$.
Their definitions, given in~\textsection\ref{subsec:psipsiprime}, are of a purely topological nature.
Theorem~\ref{th:image psi},
whose proof relies on the self-dual long exact sequence of Theorem~\ref{th:selfduality},
asserts that the images of~$\psi$ and~$\psi'$ are exact orthogonal complements
under the natural pairing $\Mpsi\times\Mpsi^*\to\Z/2\Z$;
in particular, the map~$\psi$ is surjective if~$\Pic(X_C)[2]=0$.
We show, in~\textsection\ref{subsec:cokernel norm map},
that under the same assumption,
the kernel of~$\psi$ coincides with the image of the norm map
\begin{align*}
N_{C/R}:H^{2d-2}(X(C),\Z(d-1))\to H^{2d-2}_G(X(C),\Z(d-1))\rlap{.}
\end{align*}
Thus, we obtain, in this case, a complete description of $H^{2d-2}_G(X(C),\Z(d-1))_0$.

To define~$\psi$ and~$\psi'$ and to prove these results,
we have to distinguish between the two cases $X(R)=\emptyset$ and $X(R)\neq\emptyset$.
The definitions and arguments, in these two cases, are entirely distinct, but run parallel.

In~\textsection\ref{subsec:effect of psi}, we study the map
$\phi=\psi\circ \cl : \CH_1(X) \to \Mpsi$.
It turns out that~$\phi$ coincides with
the Borel--Haefliger cycle class map when~$X(R)\neq\emptyset$,
and that when~$X(R)=\emptyset$, this map computes the geometric genus modulo~$2$
of the curves lying on~$X$.

We gather several consequences of these results in~\textsection\ref{subsection:obstructions}.
In the case of surfaces~$X$ with $H^2(X,\sO_X)=0$, we obtain a statement
of independent interest on the image of the
Borel--Haefliger cycle class map and on the existence of curves of even geometric
genus
(Corollary~\ref{cor:surfaces}), which generalises, in particular, the work of Mangolte and van~Hamel~\cite{mangoltevanhamel}
on Enriques surfaces.
In the case of simply connected Calabi--Yau or rationally connected threefolds, we obtain, by combining these results with the work of Voisin~\cite{voisinthreefolds},
an equivalence between the real integral Hodge conjecture and the surjectivity of the Borel--Haefliger cycle class map and the existence of a curve of even geometric genus
(Corollary~\ref{cor:IHC pour solides RC ou CY}).

\subsection{The map \texorpdfstring{$\psi$}{ψ} and its image}
\label{subsec:psipsiprime}

We first define the $2$\nobreakdash-torsion abelian group~$\Mpsi$ and the homomorphisms~$\psi$ and~$\psi'$.
Once~$\Mpsi$ is defined, we set $\Mpsi^*=\Hom(\Mpsi,\Z/2\Z)$.

\begin{defn}
\label{def:psi no point}
If $X(R)=\emptyset$,
we let $\Mpsi=\Z/2\Z$, we define
\begin{align*}
\psi:H^{2d-2}_G(X(C),\Z(d-1))_0 \to \Mpsi
\end{align*}
by $\psi(x)=x \smile \omega^2$ for any~$x$
(see~\textsection\ref{subsubsec:omega}
for the definition of $\omega^2 \in H^2_G(X(C),\Z(2))$),
which makes sense as $H^{2d}_G(X(C),\Z(d+1))=\Z/2\Z$ in view of
Proposition~\ref{prop:poincare duality equivariant} and Remark~\ref{rks:equivariant poincare}~(ii),
and we define
\begin{align*}
\psi':\Pic(X_C)^G[2^\infty] \to \Mpsi^*
\end{align*}
as the restriction to $\Pic(X_C)^G[2^\infty]$ of the map $\Pic(X_C)^G \to \Br(R)=\Z/2\Z=\Mpsi^*$
which measures the obstruction to
representing a $G$\nobreakdash-invariant divisor class on~$X_C$ by a $G$\nobreakdash-invariant divisor
(see~\cite[III, \textsection5.4]{grothbrauer}).
\end{defn}

\begin{defn}
\label{def:psi with point}
If $X(R)\neq\emptyset$,
we let $\Mpsi=H^{d-1}(X(R),\Z/2\Z)$, we define
\begin{align*}
\psi:H^{2d-2}_G(X(C),\Z(d-1))_0 \to \Mpsi
\end{align*}
to be the map~\eqref{eq:psi with point},
and, noting that $\Mpsi^*=H^1(X(R),\Z/2\Z)$ (see~\eqref{eq:semi-algebraic poincare duality})
and that $\Pic(X_C)^G=\Pic(X)$ as $X(R)\neq\emptyset$ (see~\cite[8.1/4]{neronmodels}), we define
\begin{align*}
\psi':\Pic(X_C)^G[2^\infty] \to \Mpsi^*
\end{align*}
as the restriction to $\Pic(X_C)^G[2^\infty]=\Pic(X)[2^\infty]$ of the
Borel--Haefliger cycle class map $\Pic(X)\to H^1(X(R),\Z/2\Z)$
(see~\eqref{eq:borel-haefliger}).
\end{defn}

\begin{thm}
\label{th:image psi}
For any smooth, proper and geometrically irreducible variety~$X$,
of dimension~$d$, over a real closed field~$R$,
the images of~$\psi$ and of~$\psi'$ are exact orthogonal complements
under the natural perfect pairing $\Mpsi \times \Mpsi^* \to \Z/2\Z$.
\end{thm}

\begin{proof}
We first prove the theorem under the assumption that $X(R)=\emptyset$.
In this case,
we have to show that~$\psi$ is surjective if and only if~$\psi'=0$.
Let us consider the commutative square
\begin{align*}
\xymatrix{
\mkern-66mu\Z/2\Z=H^0_G(X(C),\Q_2/\Z_2(-1)) \ar[r] & H^2_G(X(C),\Q_2/\Z_2(1)) \\
H^0(G,\Q_2/\Z_2(-1)) \ar[r]^\sim \ar[u]_(.47)\wr & H^2(G,\Q_2/\Z_2(1))\rlap{\text{,}} \ar[u]
}
\end{align*}
in which the horizontal maps are the cup product with~$\omega^2$ and the vertical maps
come from the Hochschild--Serre spectral sequence~\eqref{eq:hochschild-serre}.
Applying~\eqref{eq:real-complex sequence 01}
and~\eqref{eq:real-complex sequence 10},
we see that the bottom horizontal map is an isomorphism.
On the other hand,
by Proposition~\ref{prop:poincare duality equivariant}
and Remark~\ref{rks:equivariant poincare}~(ii),
the top horizontal map composed with the inclusion $H^2_G(X(C),\Q_2/\Z_2(1)) \subseteq H^2_G(X(C),\Q/\Z(1))$ is dual to~$\psi$.
Hence~$\psi$ is surjective if and only if
the vertical map on the right is injective.
By~\eqref{eq:hochschild-serre}, this amounts
to the surjectivity of the natural map
$H^1_G(X(C),\Q_2/\Z_2(1)) \to H^1(X(C),\Q_2/\Z_2(1))^G$,
or equivalently, of
 $H^1_\et(X,\Q_2/\Z_2(1)) \to H^1_\et(X_C,\Q_2/\Z_2(1))^G$.
The Kummer exact sequence in \'etale cohomology identifies the latter map with $\Pic(X)[2^\infty] \to \Pic(X_C)^G[2^\infty]$,
which is surjective if and only if~$\psi'=0$.

We henceforth assume that $X(R)\neq\emptyset$.  Let us state two lemmas.

As $H^0(G,\Q/\Z(1))=\Z/2\Z$ and $H^1(G,\Q/\Z(1))=0$,
the spectral sequence~\eqref{eq:first spectral sequence}
induces a canonical isomorphism $\Mpsi^*=H^1(X(R),\Z/2\Z) \isoto H^1_G(X(R),\Q/\Z(1))$.

\begin{lem}
\label{lem:image of psiprime}
Through this identification of~$\Mpsi^*$ with $H^1_G(X(R),\Q/\Z(1))$,
 the image of~$\psi'$ is
the image of the restriction
map $H^1_G(X(C),\Q/\Z(1)) \to H^1_G(X(R),\Q/\Z(1))$.
\end{lem}

\begin{proof}
The boundary maps of the exact sequence $0\to \Z(1) \to \Q(1) \to \Q/\Z(1) \to 0$
and the isomorphism $\Pic(X)_{\mathrm{tors}}=H^1_\et(X,\Q/\Z(1))=H^1_G(X(C),\Q/\Z(1))$
stemming from the Kummer sequence
fit into a diagram
\begin{align}
\begin{aligned}
\label{eq:image of psiprime big diagram}
\xymatrix@R=3ex{
\Pic(X)_{\mathrm{tors}} \ar@{=}[r] \ar@{}|{\bigcap}[d] & H^1_G(X(C),\Q/\Z(1)) \ar[r] \ar[d] & \rlap{$H^1_G(X(R),\Q/\Z(1))=\Mpsi^*$}\phantom{H^1_G(X(R),\Q/\Z(1))} \ar[d]^(.45)\wr \\
\Pic(X) \ar[r]^(.39)\cl & H^2_G(X(C),\Z(1)) \ar[r] & \rlap{$H^2_G(X(R),\Z(1))=\Mpsi^*$\text{,}}\phantom{H^2_G(X(R),\Z(1))}
}
\end{aligned}
\end{align}
where we identify $H^2_G(X(R),\Z(1))$ with~$\Mpsi^*$ by means of~\eqref{eq:natural projection}.
In terms of these identifications, the rightmost vertical map of~\eqref{eq:image of psiprime big diagram} is the identity map
of~$\Mpsi^*$, as follows from the commutative square
\begin{align}
\begin{aligned}
\xymatrix@R=3ex{
H^1(X(R),H^0(G,\Q/\Z(1))) \ar[r]^(.56)\sim \ar[d] & H^1_G(X(R),\Q/\Z(1)) \ar[d] \\
H^1(X(R),H^1(G,\Z(1))) \ar[r]^(.56)\sim & H^2_G(X(R),\Z(1))\rlap{\text{,}}
}
\end{aligned}
\end{align}
whose horizontal maps are the edge homomorphisms of the spectral sequences~\eqref{eq:first spectral sequence}.
According to \cite[(3.8) and~\textsection3.3.1]{ctcimetrento},
the left square of~\eqref{eq:image of psiprime big diagram}
commutes after tensoring with~$\Zl$ for every prime~$\ell$;
therefore~\eqref{eq:image of psiprime big diagram} is a commutative diagram.
By Theorem~\ref{th:conditions de krasnov},
the image of $\Pic(X)[2^\infty]$ in~$\Mpsi^*$, via~\eqref{eq:image of psiprime big diagram},
is the image of~$\psi$ (see~\eqref{eq:psi with point}).
It is also the image of~$\Pic(X)_{\mathrm{tors}}$
since $2\Mpsi^*=0$. The lemma follows.
\end{proof}

\begin{lem}
\label{lem:topological condition = restriction}
The restriction map induces an isomorphism
\begin{align*}
H^{2d-2}_G(X(C),\Z(d-1))_0\isoto H^{2d-2}_G(X(C) \setminus X(R),\Z(d-1))\rlap{\text{.}}
\end{align*}
\end{lem}

\begin{proof}
We recall that a canonical retraction
of the forgetful map
\begin{align}
\label{eq:forgetfulmap degree 2d-2}
H^{2d-2}_{G,X(R)}(X(C),\Z(d-1)) \to H^{2d-2}_G(X(C),\Z(d-1))
\end{align}
was constructed in~\textsection\ref{subsubsec:with integral coeff}.
By Proposition~\ref{prop:short exact sequences integral},
the lemma will follow if we check that
the kernel of this retraction
is $H^{2d-2}_G(X(C),\Z(d-1))_0$.
Let us take up, from~\textsection\ref{subsubsec:normal bundle with mod 2 coefficients},
the notation $H^p=H^p(X(R),\Z/2\Z)$.
Let $\alpha \in H^{2d-2}_G(X(C),\Z(d-1))$.
As in~\textsection\ref{subsubsec:topological constraints},
we denote by $\alpha_p \in H^p$ the image of~$\alpha$
by~\eqref{eq:composition restriction decomposition}
for $\congru{p}{d-1}{2}$.
Unravelling the definition of the canonical retraction of~\eqref{eq:forgetfulmap degree 2d-2}
constructed in~\textsection\ref{subsubsec:with integral coeff},
while taking~\textsection\ref{subsubsec:reduction modulo 2} into account,
shows that~$\alpha$ belongs to its kernel if and only if
the component in~$H^p$ of the class
\begin{align}
\label{eq:class verification retraction}
\sum_{\substack{0\leq i\leq d-3 \\ \vphantom{X^X}i \mkern1mu\equiv\mkern1mu d-1 \text{ mod } 2}}\mkern-10mu \left(\alpha_i+\beta_\Z(\alpha_i)\right)\gamma^{-1} \in H^0\oplus\cdots\oplus H^{d-2}
\end{align}
is trivial for each $\congru{p}{d-1}{2}$ such that $p\leq d-3$.  Here,
the product takes place in the ring $H^0\oplus\cdots\oplus H^{d-2}$ and $\beta_\Z:H^i\to H^{i+1}$ denotes
the ordinary Bockstein homomorphism (see~\textsection\ref{subsubsec:reduction modulo 2}).
On the other hand, we have
$\alpha \in H^{2d-2}_G(X(C),\Z(d-1))_0$ if and only if $\alpha_p=0$ for each $p\leq d-3$
such that $\congru{p}{d-1}{2}$.
To see that these two conditions are equivalent, we note that
if $\alpha_p\neq 0$ for some $\congru{p}{d-1}{2}$
and if~$p$ denotes the smallest such integer,
then the component in~$H^p$ of~\eqref{eq:class verification retraction} is~$\alpha_p$.
\end{proof}

Let us resume the proof of Theorem~\ref{th:image psi} under the assumption
that $X(R)\neq\emptyset$.
The map~$u_2$ appearing in~\eqref{eq:selfdual sequence}
can be inserted into a diagram
\begin{align}
\label{diag:image of psiprime}
\owrepositiontag{{\raise 11pt}{%
\xymatrix@C=1em@R=3ex{
&&\mkern-92muH^0\oplus H^1=H^1_G(X(R),\Z/2\Z) \ar[d]^{\delta_2} \ar[r]
& \rlap{$H^1_G(X(R),\Q/\Z(1))=H^1$}\phantom{H^1_G(X(R),\Q/\Z(1))}\mkern13mu \ar@<-.2em>[d] \\
H^1 \ar@<.3em>[ur]!(11.5,0) \ar[r] & H^{2-d} \oplus \cdots \oplus H^1 \ar[r]^(.42){u_2} & H^2_G(X(C),X(R),\Z/2\Z) \ar[r] & H^2_G(X(C),X(R),\Q/\Z(1))\rlap{\text{,}}
}}}
\end{align}
in which the leftmost arrows are the obvious inclusions,
the vertical maps are the connecting
homomorphisms of the localisation exact sequences,
and the canonical isomorphisms of the first row are the decomposition~\eqref{eq:canonical decomposition mod 2} and the identification used in the statement of Lemma~\ref{lem:image of psiprime}.
A glance at the description of~$u_2$ given after~\eqref{eq:selfdual sequence}
shows that this diagram commutes.
In addition, the map from the bottom left~$H^1$ to the top right~$H^1$
is the identity map.
It follows,
in view of Lemma~\ref{lem:image of psiprime},
that the kernel of the composition $\theta':H^1 \to H^2_G(X(C),X(R),\Q/\Z(1))$
of all of the maps of the bottom row of~\eqref{diag:image of psiprime}
coincides with the image of~$\psi'$.

Now, let
$\theta:H^{2d-2}_G(X(C)\setminus X(R),\Z(d-1))\to H^{d-1}$
denote the composition of
the natural map
$H^{2d-2}_G(X(C)\setminus X(R),\Z(d-1))\to H^{2d-2}_G(X(C)\setminus X(R),\Z/2\Z)$,
of the map $w_{2d-2}: H^{2d-2}_G(X(C)\setminus X(R),\Z/2\Z) \to H^{d-1}\oplus\cdots\oplus H^{2d-2}$
appearing in~\eqref{eq:selfdual sequence},
and of the projection map
$H^{d-1}\oplus\cdots\oplus H^{2d-2}\to H^{d-1}$.
The description of~$w_{2d-2}$ given after~\eqref{eq:selfdual sequence}
shows,
in view of~\textsection\ref{subsubsec:reduction modulo 2},
that~$\psi$ coincides with the composition of~$\theta$
with the isomorphism of
Lemma~\ref{lem:topological condition = restriction}.
Hence~$\psi$ and~$\theta$ have the same image.

By Theorem~\ref{th:selfduality}, Proposition~\ref{prop:poincare duality equivariant},
and Remark~\ref{rks:equivariant poincare}~(ii),
the Pontrjagin dual of~$\theta$ is~$\theta'$.
The kernel of~$\theta'$ and the image of~$\theta$ are therefore
exact orthogonal
complements.
\end{proof}

\subsection{The effect of \texorpdfstring{$\psi$}{ψ} on cycle classes}
\label{subsec:effect of psi}

We set $\phi=\psi\circ \cl : \CH_1(X) \to \Mpsi$.

\begin{thm}
\label{th:phi}
Let~$X$ be a smooth, proper and geometrically irreducible variety
over a real closed field~$R$.  If~$X(R)\neq\emptyset$, the map~$\phi$ is the Borel--Haefliger
cycle class map.  If $X(R)=\emptyset$, the map~$\phi$ is characterised by the property that
$$
\phi(B)=\left\{\begin{array}{ll}
1 & \text{if $B$ is geometrically irreducible and has even geometric genus,}\\
0 & \text{otherwise}
\end{array}
\right.
$$
for any integral proper curve $B \subseteq X$.
\end{thm}

\begin{proof}
When $X(R)\neq\emptyset$, this is Theorem~\ref{th:conditions de krasnov} applied with $i=0$.
Let us assume that $X(R)=\emptyset$. If~$B$ is not geometrically integral, then, as a $1$\nobreakdash-cycle on~$X$,
it is the norm of a $1$\nobreakdash-cycle on~$X_C$, so that~$\cl(B)$ belongs to the image of the norm map
\begin{align*}
H^{2d-2}(X(C),\Z(d-1))\to H^{2d-2}_G(X(C),\Z(d-1))\rlap{\text{.}}
\end{align*}
In view of the real-complex long exact sequence~\eqref{eq:real-complex long 10}, it follows
that $\cl(B) \smile \omega=0$,
hence $\phi(B)=0$.
We may therefore assume that~$B$ is geometrically irreducible.
Let~$B'$ denote its normalisation
and $\pi:B'\to X$ the natural morphism.
The push-forward map $\pi_*:H^2_G(B'(C),\Z(2)) \to H^{2d}_G(X(C),\Z(d+1))$
is an isomorphism since its Pontrjagin dual
 $\pi^*:H^0_G(X(C),\Q/\Z(1))\to H^0_G(B'(C),\Q/\Z(1))$
is one
(see~\eqref{eq:projection formula equivariant}
and Proposition~\ref{prop:poincare duality equivariant}).
As a consequence, if $\psi_{B'}:H^0_G(B'(C),\Z) \to \Z/2\Z$
denotes the map associated with~$B'$
by Definition~\ref{def:psi no point},
the square
\begin{align*}
\xymatrix@R=3ex{
\Z=H^0_G(B'(C),\Z) \ar@<.7em>[d]^{\pi_*} \ar[r]^(.69){\psi_{B'}} &\Z/2\Z\ar@{=}[d] \\
H^{2d-2}_G(X(C),\Z(d-1)) \ar[r]^(.69){\psi} & \Z/2\Z
}
\end{align*}
is commutative.
By Theorem~\ref{th:image psi}
and Lemma~\ref{lem:geyer} below,
the map~$\psi_{B'}$ is surjective
if and only if~$B'$ has even genus.
As the vertical arrow on the left maps~$1$ to~$\cl(B)$,
this completes the proof
of Theorem~\ref{th:phi}.
\end{proof}

\begin{lem}[Geyer]
\label{lem:geyer}
Let~$B$ be a smooth, proper and geometrically integral curve over a real closed field~$R$.
If $B(R)=\emptyset$,
the natural map $\Pic(B)[2^{\infty}]\to \Pic(B_C)^G[2^\infty]$ is
surjective
if and only if~$B$ has even genus.
\end{lem}

\begin{proof}
This follows from the work of Geyer
(see~\cite[p.~91]{geyer}, extended to real closed fields
in~\cite[\textsection10]{knebusch2};
see also~\cite[\textsection2]{grossharris}).
We sketch a complete argument
for the reader's convenience.
First, as $B(R)=\emptyset$ and~$B$ is a curve, by a theorem of Witt~\cite{wittreel}
extended by Pfister to real closed fields (see~\cite[Theorem~4.1]{knebusch1} and the
references therein), one can write~$-1$ as a sum
of two squares in~$R(B)$.
In other words, the natural map $\Br(R)\to \Br(B)$ vanishes.
Hence, the exact sequence of low degree terms of the Hochschild--Serre spectral sequence takes the form
\begin{align}
\label{eq:suite pic pic br}
\xymatrix{
0 \ar[r] & \Pic(B) \ar[r] & \Pic(B_C)^G \ar[r] & \Br(R) \ar[r] & 0\rlap{\text{.}}
}
\end{align}

Let~$g$ denote the genus of~$B$.

\begin{sublem}\label{sublem:geyer}
Let $L \in \Pic(B_C)^G$.
If $\deg(L)\equiv g \!\!\mod 2$,
then $L \in \Pic(B)$.
\end{sublem}

\begin{proof}
After tensoring~$L$ with a large power of an ample line bundle,
we may assume that $\deg(L)>2g-2$.
The Riemann--Roch theorem then implies that $\dim_C H^0(B_C,L)$ is odd.
On the other hand,
this integer kills the image of~$L$ in~$\Br(R)$
(see~\cite[Theorem~6]{lichtenbaumduality} or \cite[p.~160]{grossharris}).
\end{proof}

The sublemma implies the lemma when~$g$ is even.
Suppose now~$g$ is odd. By~\eqref{eq:suite pic pic br} there exists $L \in \Pic(B_C)^G$
with $L \notin \Pic(B)$.  By the sublemma, its degree must be even.
By adding to~$L$ a suitable multiple of the class of a closed point of~$B$,
we may assume that $\deg(L)=0$.
As $\Pic^0(B_C)$ is divisible, there
exists $L_0 \in \Pic^0(B_C)$ such that $L=2L_0$.  Then $L-N_{C/R}(L_0)$ is an element
of $\Pic(B_C)^G[2]$ which does not belong to $\Pic(B)[2]$, thus completing the proof
of the lemma.
\end{proof}

\subsection{Connection with the first intermediate index}
\label{subsec:elw}

According to Theorem~\ref{th:phi},
the map $\phi:\CH_1(X)\to \Mpsi$ is surjective if and only if
\begin{enumerate}[(i)]
\item the map $\cl_R:\CH_1(X)\to H^{d-1}(X(R),\Z/2\Z)$ is surjective, if $X(R)\neq\emptyset$;
\item $X$ contains a geometrically integral curve of even geometric genus, if $X(R)=\emptyset$.
\end{enumerate}

The first condition is a classical one in real algebraic geometry.
Since the second one is not, and will play a prominent role in what follows when $X(R)=\emptyset$, we devote this section to explaining a few general facts about it.

In the sequel, we shall encounter many examples
of varieties, over the reals,
which do not contain a curve of even geometric genus
(see~\textsection\ref{section:examples} and \cite[\textsection\ref*{BW2-par:cexnonarch}]{bwpartie2}).
The following proposition illustrates the sharpness of such examples.

\begin{prop}
\label{prop:parity of genus in irrelevant situations}
Let~$X$ be a smooth, proper and geometrically irreducible variety, of dimension~$\geq 2$, over~$R$
(or, more generally, over an infinite perfect field).
Then~$X$ contains a geometrically irreducible curve of odd geometric genus.
If $X(R)\neq\emptyset$, then~$X$ contains a geometrically irreducible curve of even geometric genus.
\end{prop}

\begin{proof}
We first assume that~$X$ is a surface.
In this case, if~$H$ denotes a very ample divisor on~$X$,
a general member of the linear system $|4H|$ is a geometrically irreducible curve with odd geometric genus,
as follows from the adjunction formula.
If $X(R)\neq\emptyset$, let $\pi:X' \to X$ denote the blowing-up of a point of~$X(R)$,
with exceptional divisor~$E$, and let~$H'$ denote an ample divisor on~$X'$.
The image, by~$\pi$, of a general member of the linear system $|4nH'+E|$,
for $n\gg 0$,
is a geometrically irreducible curve with even geometric genus,
according to the adjunction formula.
The general case of Proposition~\ref{prop:parity of genus in irrelevant situations}
can be proved by applying these arguments to the desingularisation of an appropriate
surface lying in~$X$.
\end{proof}

The existence of a curve of even geometric genus is conveniently expressed in
terms of the intermediate indices of~$X$, introduced by Koll\'ar~\cite[Definition~1]{kollarelw}
following~\cite{elw}.
We recall that for any~$i$,
the \emph{$i$th intermediate index} $\elw_i(X)$ is, by definition,
the gcd of the integers $\chi(X,E)$ when~$E$ ranges over the coherent sheaves on~$X$ supported
on a closed subset of dimension~$\leq i$.  Clearly
\begin{align*}
\elw_d(X) \mid \elw_{d-1}(X) \mid \cdots \mid \elw_0(X)
\end{align*}
and $\elw_0(X)$ is the gcd of the degrees of the closed points of~$X$.
The next statement, a consequence of
standard results on Grothendieck groups of coherent sheaves
(see \cite[\textsection8]{borelserre}, \cite{berthelotsga6}),
summarises some basic properties of these indices.

\newcommand{\citekollarpropq}{\cite[Proposition~4]{kollarelw}}
\begin{prop}[\citekollarpropq]
\label{prop:basic properties elw}
For any proper variety~$X$ and any~$i$,
we have $\elw_i(X)=\gcd \chi(Z,\sO_Z)=\gcd \chi(Z',\sO_{Z'})$,
where~$Z$ ranges over the integral closed subvarieties of~$X$ of dimension~$\leq i$ and~$Z'$ denotes the normalisation of~$Z$.
\end{prop}

We can now relate~$\elw_1(X)$ to curves of even genus.
Note that if~$X$ is a nonempty proper variety over~$R$, then $\elw_0(X)=1$ if $X(R)\neq\emptyset$
and $\elw_0(X)=2$ otherwise.

\begin{cor}
\label{cor:what is elw1}
Let~$X$ be a nonempty proper variety over~$R$.
If~$X(R)\neq\emptyset$ or if~$X$ contains a geometrically irreducible curve of even geometric genus,
then $\elw_1(X)=1$.  Otherwise $\elw_1(X)=2$.
\end{cor}

As another consequence of Proposition~\ref{prop:basic properties elw}, we note that
for any proper and geometrically irreducible curve~$B$ over~$R$ such that $B(R)=\emptyset$,
the geometric genus and the arithmetic genus
of~$B$ have the same parity (both are congruent to $1-\elw_1(B)$ modulo~$2$).
Replacing the word ``geometric'' with ``arithmetic''
in the statements of Theorem~\ref{th:phi}
and Corollary~\ref{cor:what is elw1}
would therefore make no difference.

Finally, we recall how the parity of the genus behaves in a cover of curves.

\begin{prop}
\label{prop:genusundermorphism}
Let $f:E \to F$ be a finite morphism between smooth, proper and geometrically irreducible curves over~$R$.
Assume that $F(R)=\emptyset$ and that~$F$ has even geometric genus.  Then~$E$ has even geometric
genus if and only if~$f$ has odd degree.
\end{prop}

\begin{proof}
Apply the degree formula $\congru{\chi(E,\sO_E)}{\deg(f)\chi(F,\sO_F)}{2}$
(see~\cite{hautiondegreeformula} or \cite[Lemma~5~(2)]{kollarelw}).
\end{proof}

\begin{rmk}
Let~$X$ be a smooth and proper variety, defined over an arbitrary field.
Let~$i$ be an integer.
Associating, with an integral dimension~$i$ cycle on~$X$,
the Euler characteristic of its structure sheaf
determines a morphism $\CH_i(X)\to\Z/\elw_{i-1}(X)\Z$
which factors through algebraic equivalence
(see~\cite[Proposition~14]{kollarelw}).
When the ground field is real closed and $i=1$, more is true:
this morphism even factors through homological equivalence,
according to Theorem~\ref{th:phi}.

We do not know whether this assertion remains valid for all values of $i$.
For $i=1$,
it does not extend to arbitrary fields,
as the following example shows.
Let~$X$ denote the smooth projective quadric threefold
over $\C((t))((u))((v))$
defined by the anisotropic quadratic form
$\langle 1,t,u,tu,v \rangle$.
Using the Hochschild--Serre spectral sequence, one
checks that $H^4_\et(X,\Z_2(2))$ is torsion-free.
On the other hand, it follows from \cite[Theorem~5.3 and Theorem~3.8]{Karpenko}
that $\CH_1(X)_\tors$ has order~$2$ and is
generated by a class whose image in $\Z/\elw_0(X)\Z=\Z/2\Z$ does not vanish.
\end{rmk}

\subsection{The cokernel of the norm map}
\label{subsec:cokernel norm map}

\begin{prop}
\label{prop:conoyau norme}
Let~$X$ be a smooth, proper and geometrically irreducible variety,
of dimension~$d$, over a real closed field~$R$.  If $\Pic(X_C)[2]=0$, then~$\psi$ fits into an exact sequence
\begin{align}
\label{eq:conoyau de la norme}
\xymatrix@C=1.3em{
H^{2d-2}(X(C),\Z(d-1))\ar[r] & H^{2d-2}_G(X(C),\Z(d-1))_0 \ar[r]^(.79){\smash[t]{\psi}} & \Mpsi \ar[r] & 0\rlap{\text{,}}
}
\end{align}
where the first map is the norm map from~\eqref{eq:real-complex long 10}.
\end{prop}

It is part of the assertion of Proposition~\ref{prop:conoyau norme} that the norm map takes its values
in the subgroup $H^{2d-2}_G(X(C),\Z(d-1))_0 \subseteq H^{2d-2}_G(X(C),\Z(d-1))$.

\begin{proof}
The following lemma immediately
implies Proposition~\ref{prop:conoyau norme} when $X(R)=\emptyset$.

\begin{lem}
\label{lem:proof of conoyau norme}
Under the assumptions of the proposition, the group $H^{2d-1}(X(C),\Z)$ is finite of odd order.  Cup product with $\omega \in H^1_G(X(C),\Z(1))$ induces a surjection
\begin{align}
\label{eq:cupproduct with omega 1}
H^{2d-2}_G(X(C),\Z(d-1)) \twoheadrightarrow H^{2d-1}_G(X(C),\Z(d))[2]
\end{align}
whose kernel is the image of the norm map and it induces an isomorphism
\begin{align}
\label{eq:cupproduct with omega 2}
H^{2d-1}_G(X(C),\Z(d))[2] \isoto H^{2d}_G(X(C),\Z(d+1))\rlap{\text{.}}
\end{align}
If $X(R)\neq\emptyset$, it also induces an isomorphism
\begin{align}
\label{eq:cupproduct with omega 3}
H^{2d}_G(X(C),\Z(d+1))\isoto H^{2d+1}_G(X(C),\Z(d+2))\rlap{\text{.}}
\end{align}
\end{lem}

To deduce Proposition~\ref{prop:conoyau norme} from Lemma~\ref{lem:proof of conoyau norme} when $X(R)\neq\emptyset$, we consider the commutative diagram
\begin{align*}
\xymatrix@R=-.5ex@C=3.5em{
&H^{2d-2}_G(X(C),\Z(d-1)) \ar@{->>}[r] \ar[d] & H^{2d+1}_G(X(C),\Z(d+2)) \ar[d]^(.45)\wr &
\vphantom{\displaystyle\bigoplus_{\substack{0\leq p\leq d-1 \\ p\mkern1mu\equiv\mkern1mu d-1 \text{ mod } 2}}} \\
&H^{2d-2}_G(X(R),\Z(d-1)) \ar[d] \ar[r] & H^{2d+1}_G(X(R),\Z(d+2)) \ar[d]^(.45)\wr &
\vphantom{\displaystyle\bigoplus_{\substack{0\leq p\leq d-1 \\ p\mkern1mu\equiv\mkern1mu d-1 \text{ mod } 2}}} \\
&\mkern-60mu\displaystyle\bigoplus_{\substack{0\leq p\leq d-1 \\ p\mkern1mu\equiv\mkern1mu d-1 \text{ mod } 2}}\mkern-25muH^p(X(R),\Z/2\Z) \ar@{=}[r]!(-24.1,0) &
\mkern-60mu\displaystyle\bigoplus_{\substack{0\leq p\leq d-1 \\ p\mkern1mu\equiv\mkern1mu d-1 \text{ mod } 2}}\mkern-25muH^p(X(R),\Z/2\Z)\rlap{\text{,}}
}
\end{align*}
in which the horizontal arrows are the cup product with~$\omega^3$
and the lower vertical maps are the maps~\eqref{eq:natural projection}.
The commutativity of the lower square follows
from~\textsection\ref{subsubsec:reduction modulo 2}
and~\textsection\ref{subsubsec:effect}.
By Lemma~\ref{lem:proof of conoyau norme},
the top horizontal map is surjective and its kernel is the image of the norm map.
The upper vertical map on the right is an isomorphism
by~\textsection\ref{subsubsec:cohomological dimension}.
Hence the composition of the vertical maps on the left is surjective and its kernel is the image of the
norm map.
Proposition~\ref{prop:conoyau norme} is now established.
\end{proof}

\begin{proof}[Proof of Lemma~\ref{lem:proof of conoyau norme}]
As $\Pic(X_C)[2]=0$, the $2$\nobreakdash-torsion subgroup of
$H^1_\et(X_C,\Q_2/\Z_2)$ is trivial.  By Poincar\'e duality,
the group $H^{2d-1}_\et(X_C,\Z_2)=H^{2d-1}(X(C),\Z) \otimes_\Z \Z_2$ is therefore
$2$\nobreakdash-divisible.  The group $H^{2d-1}(X(C),\Z)$ being finitely
generated, it must then be finite of odd order.
On the other hand,
as $2\omega=0$, the cohomology class $x \smile \omega$ is
$2$\nobreakdash-torsion for any~$x$, hence~\eqref{eq:cupproduct with omega 1}
is well-defined.
It follows from these remarks and from the real-complex long exact sequence~\eqref{eq:real-complex long 10}
that~\eqref{eq:cupproduct with omega 1} is surjective, with kernel the image
of the norm map.

The norm map $H^{2d}(X(C),\Z(d)) \to H^{2d}_G(X(C),\Z(d))$ is injective as
its composition with the natural map $H^{2d}_G(X(C),\Z(d)) \to H^{2d}(X(C),\Z(d))$
is multiplication by~$2$ on the group $H^{2d}(X(C),\Z(d))=\Z$
(see Lemma~\ref{lem:H2dreel}~(i)).
Putting together this injectivity,
the remarks that $H^{2d-1}(X(C),\Z)$ has odd order
and that $2\omega=0$,
and the exact sequence~\eqref{eq:real-complex long 10}, we deduce
that~\eqref{eq:cupproduct with omega 2} is an isomorphism.

If $X(R)\neq\emptyset$, the natural map
$H^{2d}_G(X(C),\Z(d)) \to H^{2d}(X(C),\Z(d))$ is surjective
(see Lemma~\ref{lem:H2dreel}~(i)).
On the other hand, we have $H^{2d+1}(X(C),\Z(d))=0$
as~$X(C)$ has cohomological dimension~$2d$
(see \cite[Chapter~II, Lemma~9.1]{delfshomology}).  Hence the exact sequence~\eqref{eq:real-complex long 10}
implies that~\eqref{eq:cupproduct with omega 3} is an isomorphism in this case.
\end{proof}

\subsection{Wrapping up}
\label{subsection:obstructions}

We now combine the contents
of~\textsection\textsection\ref{subsec:psipsiprime}--\ref{subsec:cokernel
norm map} and deduce various results on the Borel--Haefliger cycle class
map for $1$\nobreakdash-cycles, on the existence of curves of even genus,
and on the real integral Hodge conjecture for $1$\nobreakdash-cycles.

In accordance with common usage,
for a smooth and proper variety~$X$,
we set $H_1(X(R),\Z/2\Z)=H^{d-1}(X(R),\Z/2\Z)$
and write $H_1^\alg(X(R),\Z/2\Z)=\cl_R(\CH_1(X))$
and $H^1_\alg(X(R),\Z/2\Z)=\cl_R(\Pic(X))$,
where~$\cl_R$ denotes the Borel--Haefliger cycle class maps
for curves or for divisors.

\subsubsection{Nomenclature}
\label{nomenclature}

We have seen in Theorem~\ref{th:phi} that
$\phi=\psi\circ\cl:\CH_1(X) \to \Mpsi$ detects the genus modulo~$2$
if $X(R)=\emptyset$ and is the Borel--Haefliger
cycle class map if $X(R)\neq\emptyset$.
Let us analyse the various ways in which~$\phi$ can fail to be surjective.
First, if~$\psi$ is not surjective,
we say that there is a \emph{topological obstruction} to the surjectivity of~$\phi$.
When $R=\R$, one can further factor~$\phi$ as
\begin{align*}
\CH_1(X)\to \Hdg^{2d-2}_G(X(C),\Z(d-1))_0\subseteq H^{2d-2}_G(X(C),\Z(d-1))_0 \xrightarrow{\psi} \Mpsi\rlap{\text{.}}
\end{align*}
If $R=\R$,
we say that there is a \emph{Hodge-theoretic obstruction}
to the surjectivity of~$\phi$
if~$\psi$ is surjective but its restriction to~$\Hdg^{2d-2}_G(X(C),\Z(d-1))_0$
is not.
Finally, if $R=\R$ (resp., if $H^2(X,\sO_X)=0$),
we speak of a \emph{cycle-theoretic obstruction} to the surjectivity of~$\phi$
if the Hodge-theoretic (resp., topological) obstruction vanishes but~$\phi$ still fails to be surjective.

There can be a topological obstruction
only if $\Pic(X_C)[2]\neq 0$,
by Theorem~\ref{th:image psi};
there can be a Hodge-theoretic obstruction
only if $H^2(X,\sO_X)\neq 0$;
and there can be a cycle-theoretic obstruction
only if the real integral Hodge conjecture for $1$\nobreakdash-cycles fails for~$X$.
Examples illustrating all of these obstructions will be given in~\textsection\ref{section:examples}.

\subsubsection{Varieties with $H^2(X,\sO_X)=0$}
\label{subsec:varieties with h20=0}

Assume that $H^2(X,\sO_X)=0$ (``no Hodge-theoretic obstruction'')
and that~$X$ satisfies the real integral Hodge conjecture for $1$\nobreakdash-cycles
(``no cycle-theoretic obstruction'').
Then, by Theorem~\ref{th:image psi}, the image of $\phi=\psi\circ\cl:\Pic(X)\to\Mpsi$ is the orthogonal complement
of the image of~$\psi'$ (``the topological obstruction controls the image of~$\phi$'').
Combining this with Theorem~\ref{th:phi}
and Proposition~\ref{prop:parity of genus in irrelevant situations}, we have now established
the following statement, in which~$\cl_R$ denotes the two
Borel--Haefliger
cycle class maps
$\Pic(X) \to H^1(X(R),\Z/2\Z)$
and $\CH_1(X) \to H^{d-1}(X(R),\Z/2\Z)$.

\begin{thm}
\label{th:nohodgetheoreticob}
Let~$X$ be a smooth, proper and geometrically irreducible variety over a
real closed field~$R$.
Assume
that $\dim(X)>0$,
that $H^2(X,\sO_X)=0$
and that~$X$ satisfies the real integral Hodge conjecture for $1$\nobreakdash-cycles.
\begin{enumerate}[(i)]
\item
The subgroups $\cl_R(\CH_1(X))$ and
$\cl_R(\Pic(X)[2^\infty])$ are exact orthogonal
complements under the Poincar\'e duality pairing.
\item
There exists a geometrically irreducible curve of even geometric
genus in~$X$ if and only if the natural map $\Pic(X)[2^{\infty}]\to
\Pic(X_C)^G[2^{\infty}]$ is onto.
\end{enumerate}
In particular, if $\Pic(X_C)[2]=0$, then $H_1^\alg(X(R),\Z/2\Z)=H_1(X(R),\Z/2\Z)$ and~$X$
contains a geometrically irreducible curve of even geometric genus.
\end{thm}

According to Proposition~\ref{prop:real(1,1)nonarch},
the hypotheses of Theorem~\ref{th:nohodgetheoreticob}
are met for surfaces of geometric genus zero.

\begin{cor}
\label{cor:surfaces}
Let~$X$ be a smooth, proper and geometrically irreducible surface over a
real closed field~$R$, such that $H^2(X,\sO_X)=0$.
\begin{enumerate}[(i)]
\item
The subgroups $\cl_R(\Pic(X))$ and $\cl_R(\Pic(X)[2^\infty])$ of $H^1(X(R),\Z/2\Z)$ are exact orthogonal
complements under the Poincar\'e duality pairing.
\item There exists a geometrically irreducible curve of even geometric
genus in~$X$ if and only if the natural map $\Pic(X)[2^{\infty}]\to
\Pic(X_C)^G[2^{\infty}]$ is onto.
\end{enumerate}
In particular, if $\Pic(X_C)[2]=0$, then $H^1_\alg(X(R),\Z/2\Z)=H^1(X(R),\Z/2\Z)$ and~$X$
contains a geometrically irreducible curve of even geometric genus.
\end{cor}

\begin{rmks}
\label{rmk:surfacespg=0}
(i)
At least when $R=\R$,
the particular case of Corollary~\ref{cor:surfaces}~(i)
when $\Pic(X_C)[2]=0$ was known
to Silhol \cite[Theorem~III.3.4]{silhol} for geometrically rational surfaces and to van~Hamel
\cite[Chapter~IV, Corollary~4.4 and Chapter~III, Lemma~8.9]{vanhamelthese} in general;
see also \cite[Th\'eor\`eme~3.7.18]{mangoltelivre}.
Corollary~\ref{cor:surfaces}~(ii), on the other hand, is new even when $\Pic(X_C)[2]=0$.

(ii)
Assume that $R=\R$. 
According to \cite[Th\'eor\`eme~4]{kahnchern},
the map~$\cl_{\R}$ sends the isomorphism class of a
line bundle on~$X$ to the first Stiefel--Whitney class of the
line bundle it induces on~$X(\R)$.
In particular,
if~$K_X$ denotes the canonical divisor class of~$X$
and
$w_1(X(\R))$
the first Stiefel--Whitney class of the tangent bundle of~$X(\R)$,
then $\cl_{\R}(K_X)=w_1(X(\R))$.
On the other hand, for an Enriques surface~$X$, the group $\Pic(X)[2^\infty]$ has order~$2$ and
is generated by~$K_X$.
Corollary~\ref{cor:surfaces}~(i) therefore
recovers
the theorem of Mangolte and van~Hamel~\cite[Theorem~4.4]{mangoltevanhamel}
according to which if~$X$ is an Enriques surface,
the subgroup $H^1_\alg(X(\R),\Z/2\Z)$ is the orthogonal complement of~$w_1(X(\R))$
(so that the equality $H^1_\alg(X(\R),\Z/2\Z)=H^1(X(\R),\Z/2\Z)$ holds
if and only if every connected component of~$X(\R)$ is orientable).
Corollary~\ref{cor:surfaces}~(i) may be viewed as a generalisation
of this result to all surfaces of geometric genus zero and all real closed fields.

(iii)
Corollary~\ref{cor:surfaces}~(i)
was known to Kucharz when~$R=\R$
and~$X_\C$ is birationally ruled
(see \cite[Proposition~1.6]{kucharzalgeq};
note that $\NS(X)$ is torsion-free in this case,
so that $\cl_\R(\Pic(X)[2^\infty])=\cl_\R(\Pic^0(X)[2^\infty])=\cl_\R(\Pic^0(X))$).

(iv)
Both assertions of Corollary~\ref{cor:surfaces} fail
if we drop the assumption $H^2(X,\sO_X)=0$ (see Example~\ref{ex:quartique avec points}
and Example~\ref{ex:kollar quartique sans point}; these are~$K3$ surfaces).
\end{rmks}

\begin{example}
\label{ex:Enriques sans point}
By Corollary~\ref{cor:surfaces}~(ii),
every Enriques surface over a real closed field contains a geometrically irreducible
curve of even geometric genus.  Indeed, if~$X$ is
such a surface, then $H^2(X,\sO_X)=0$ and $\Pic(X_C)[2^\infty]$ is generated by the
canonical class~$K_{X_C}$, which comes, by pull-back, from the class~$K_X$
in $\Pic(X)[2]$.
\end{example}

As a consequence of Corollary~\ref{cor:surfaces}~(i), we obtain
a lower bound on the size of the subgroup $H^1_\alg(X(R),\Z/2\Z) \subseteq H^1(X(R),\Z/2\Z)$, yielding a positive result concerning the question raised in \cite[Remark 5.2]{BenedettiTognoli}.

\begin{cor}
\label{cor:lower bound h1}
Let~$X$ be a smooth, proper and geometrically irreducible surface over a
real closed field~$R$.  Assume that $H^2(X,\sO_X)=0$. Then
\begin{align*}
\dim_{\Z/2\Z} H^1_\alg(X(R),\Z/2\Z) \geq \frac{1}{2}\dim_{\Z/2\Z} H^1(X(R),\Z/2\Z)\rlap{\text{.}}
\end{align*}
In particular, if $H^1(X(R),\Z/2\Z)\neq 0$, then $H^1_\alg(X(R),\Z/2\Z)\neq 0$.
\end{cor}

\begin{proof}
By Corollary~\ref{cor:surfaces}~(i), we
have $H^1_\alg(X(R),\Z/2\Z)^\perp \subseteq H^1_\alg(X(R),\Z/2\Z)$.
\end{proof}

\begin{rmks}
(i)
This bound is sharp:
if~$E$ is an elliptic curve over~$\R$,
if $X \to E$ is a conic bundle surface
with
smooth fibres over~$E(\R)$, and if~$X(\R)$ is connected and nonempty while~$E(\R)$
has two connected components,
it is an exercise to check that $\dim_{\Z/2\Z}H^1(X(\R),\Z/2\Z)=2$
while $\dim_{\Z/2\Z}H^1_\alg(X(\R),\Z/2\Z)=1$.

(ii)
There are surfaces~$X$ over~$\R$ with $H^2(X,\sO_X)\neq 0$ (\emph{e.g.}, $K3$
surfaces) such that $H^1(X(\R),\Z/2\Z)\neq 0$
but $H^1_\alg(X(\R),\Z/2\Z)=0$ (see \cite[Exemple~4.5.9]{mangoltelivre}).
\end{rmks}

\subsubsection{A criterion for the real integral Hodge conjecture}

In the next statements,
the real integral Hodge conjecture over an arbitrary real
closed field is meant in the sense of Definition~\ref{def:ihc real closed field}.
Similarly, when $R\neq \R$, ``the complex integral Hodge conjecture for $1$\nobreakdash-cycles on~$X_C$'' when $H^2(X,\sO_X)=0$
simply refers to the surjectivity of the cycle class map $\cl_C:\CH_1(X_C)\to H^{2d-2}(X(C),\Z(d-1))$.

\begin{thm}
\label{thm:relation ihc phi}
Let $X$ be a smooth, proper and geometrically irreducible variety over a real closed field~$R$.
Assume that $\Pic(X_C)[2]=0$, that $H^2(X,\sO_X)=0$,
and that~$X_C$ satisfies the complex integral Hodge conjecture for $1$\nobreakdash-cycles.
Then~$X$ satisfies the real integral Hodge conjecture for $1$\nobreakdash-cycles
if and only if the following hold:
\begin{enumerate}
\item[(i)] if $X(R)\neq\emptyset$, then $H_1(X(R),\Z/2\Z)=H_1^\alg(X(R),\Z/2\Z)$;
\item[(ii)] if \mbox{$X(R)=\emptyset$, then $X$ contains a geometrically irreducible curve of even} geometric genus.
\end{enumerate}
\end{thm}

\begin{proof}
The exact sequence of Proposition~\ref{prop:conoyau norme} fits into a commutative diagram
\begin{align*}
\xymatrix@R=3ex{
H^{2d-2}(X(C),\Z(d-1)) \ar[r] & H^{2d-2}_G(X(C),\Z(d-1))_0 \ar[r]^(.78)\psi & \Mpsi \ar[r] & 0\\
\CH_1(X_C)\ar[u]_{\cl_C} \ar[r] & \CH_1(X) \ar[u]_\cl \ar@<-.45em>[ru]_(.55)\phi
}
\end{align*}
whose bottom horizontal map is the proper push-forward.
The map~$\cl_C$ is surjective by assumption.
Hence the surjectivity of~$\cl$ is equivalent to that of~$\phi$.
The latter is, in turn,
equivalent to~(i)--(ii),
by Theorem~\ref{th:phi}.
\end{proof}

For smooth proper threefolds $X$ that are rationally connected, or Calabi-Yau (in the sense that $K_X\simeq \sO_X$ and $H^1(X,\sO_X)=H^2(X,\sO_X)=0$), the complex
integral Hodge conjecture for $1$\nobreakdash-cycles
was proved by Voisin \cite[Theorem~2]{voisinthreefolds} when $R=\R$. Using comparison with \'etale cohomology and the Lefschetz principle, the same holds for an arbitrary real closed field $R$. One deduces:

\begin{cor}
\label{cor:IHC pour solides RC ou CY}
Let $X$ be a smooth and proper threefold
over a real closed field~$R$.
Assume that~$X$ is
 rationally connected or is simply connected Calabi--Yau.
Then the real integral Hodge conjecture for~$X$ is equivalent
to the equality $H_1^\alg(X(R),\Z/2\Z)=H_1(X(R),\Z/2\Z)$, if $X(R)\neq\emptyset$,
or to the existence of a geometrically irreducible curve of even geometric genus, if $X(R)=\emptyset$.
\end{cor}

\section{Examples}
\label{section:examples}

We now provide various examples
of smooth, proper and geometrically irreducible varieties~$X$ over~$\R$
such that one of the two equalities
$\elw_1(X)=1$ (when $X(\R)=\emptyset$)
or
$H_1(X(\R),\Z/2\Z)=H_1^\alg(X(\R),\Z/2\Z)$ (when $X(\R)\neq\emptyset$)
fails.
We recall that when $X(\R)=\emptyset$,
the condition $\elw_1(X)=1$ is equivalent to the existence
of a curve of even genus in~$X$
(see Corollary~\ref{cor:what is elw1}).
All of these examples illustrate the obstructions
to the surjectivity of $\phi:\CH_1(X)\to\Mpsi$
described in~\textsection\ref{nomenclature}.

\subsection{Topological obstructions}
\label{subsec:topological examples}

The easiest examples are curves of positive genus.
For a curve,
only a topological obstruction can prevent~$\phi$ from being surjective,
since the map
 $\cl:\CH_1(X) \to H^{2d-2}_G(X(C),\Z(d-1))_0$ is an isomorphism
when $d=1$.

\begin{example}[with real points]
If~$X$ is a curve and $X(\R)$ has at least two connected components,
then $H_1(X(\R),\Z/2\Z)\neq H_1^\alg(X(\R),\Z/2\Z)$.
\end{example}

\begin{example}[with no real point]
If~$X$ is a curve of odd genus and $X(\R)=\emptyset$, then $\elw_1(X)=2$
(see Corollary~\ref{cor:what is elw1}).
\end{example}

By Theorem~\ref{th:image psi},
topological obstructions can only occur if $\Pic(X_{\C})[2]\neq 0$.
Curves of positive genus satisfy the stronger
property that the abelian variety $\Pic^0(X_{\C})$ is non-zero.
Let us now show that surfaces with $\Pic^0(X_\C)=0$
can also carry topological obstructions.
The surfaces used in the next two examples even satisfy $H^i(X,\sO_X)=0$
for all $i>0$:
they illustrate the sharpness of the last statement of Corollary~\ref{cor:surfaces}.

\begin{example}[with real points]
\label{ex:enriques mangolte van hamel}
Let~$X$ be a real Enriques surface such that~$X(\R)$ is non-orientable
(such surfaces exist; see, \emph{e.g.}, \cite[Theorem~2.2]{degtyarevkharlamov}).
By the theorem of Mangolte and van~Hamel \cite[Theorem~1.1]{mangoltevanhamel}
recalled in Remark~\ref{rmk:surfacespg=0}~(ii),
we then have $H_1(X(\R),\Z/2\Z)\neq H_1^\alg(X(\R),\Z/2\Z)$.
This is explained by a topological obstruction
(see
 Remark~\ref{rmk:surfacespg=0}~(ii)
and the discussion in~\textsection\ref{subsec:varieties with h20=0}).
\end{example}

As we have seen
in Example~\ref{ex:Enriques sans point},
the map~$\phi$ is always surjective
for an Enriques surface with no real point.
To provide an analogue of Example~\ref{ex:enriques mangolte van hamel}
in the case with no real point,
we resort to Campedelli surfaces instead.
These are minimal surfaces of general type
with $H^i(X,\sO_X)=0$ for $i>0$ and $K_X^2=2$.
The precise surfaces we use below
 were constructed, over~$\C$, by
Godeaux~\cite[\textsection6]{godeaux}, and are also described
in~\cite[\textsection2.1]{reid}.
We endow them with an appropriate real structure.

\begin{example}[with no real point]
\label{ex:campedelli}
Let~$\zeta$ be a primitive eighth root of unity.
Let $f:\P^6(\C)\to\P^6(\C)$ be the map defined by
\begin{align*}
f([x_0:\dots:x_6])=[\zeta\petitoverline{x_1}:\petitoverline{x_0}:\zeta^2\petitoverline{x_3}:\petitoverline{x_2}:\zeta^3\petitoverline{x_5}:\petitoverline{x_4}:\petitoverline{x_6}]\rlap{.}
\end{align*}
Let~$H$ denote the group of diffeomorphisms of~$\P^6(\C)$ generated by~$f$.
One checks that $H=\Z/16\Z$ and
that the index~$2$ subgroup $K=\Z/8\Z\subset H$ acts holomorphically on $\P^6(\C)$ while the other elements of~$H$ act antiholomorphically.
Let us identify
$(f^2)^* \sO(1)$ with~$\sO(1)$
by mapping~$x_0$ to~$\zeta x_0$
and let us
consider the
endomorphism
of the complex vector space $H^0(\P^6(\C),\sO(2))$
induced by~$f^2$
and by this identification.
The family $(x_i x_j)_{0\leq i\leq j\leq 6}$ forms an eigenbasis.
Let~$\Lambda_i$
denote the eigenspace associated with~$\zeta^{2i}$.
For $(Q_0,\dots,Q_3) \in \prod_{i=0}^3 \Lambda_i$,
let $Y \subset \P^6(\C)$
denote the subvariety
defined by $Q_0=\dots=Q_3=0$.
By an explicit computation based on the Bertini theorem for linear systems,
one checks that if $Q_0,\dots,Q_3$ are general,
then~$Y$
is a smooth surface
that does not meet the fixed locus of~$f^8$, hence does not meet the fixed locus of any nontrivial element of $H$.
Moreover, one checks that there is a Zariski dense set of quadruples $(Q_0,\dots,Q_3)$
for which~$Y$ is stable under~$f$.
As a consequence, we can choose $(Q_0,\dots,Q_3)$ such
that~$Y$ is a smooth surface on which~$H$ acts freely.
The quotient $S=Y/K$ is then a smooth projective complex surface
on which~$f$ induces a fixed-point free antiholomorphic involution~$\sigma$.
There exist a real projective surface~$X$ and an isomorphism $X(\C)\simeq S$ through which complex conjugation
corresponds to~$\sigma$
(see~\cite[Proposition~I.1.4]{silhol}).
We then have $X(\R)=\emptyset$.
As~$Y$ is a simply connected \'etale cover of both~$X_\C$ and~$X$,
we have $\pi_1^{\et}(X_{\C})=K$ and $\pi_1^{\et}(X)=H$,
hence $H^1_\et(X_\C,\Z/2\Z)=\Z/2\Z$ and $H^1_\et(X,\Z/2\Z)=\Z/2\Z$.
By the Kummer exact sequence,
we deduce that $\Pic(X_\C)[2]=\Z/2\Z$ and $\Pic(X)[2]=0$.
Corollary~\ref{cor:surfaces}~(ii) now implies that~$\elw_1(X)=2$.
\end{example}

\subsection{Hodge-theoretic obstructions}
\label{subsec:hodge-theoretic examples}

The next two examples are~$K3$ surfaces.
As such surfaces are simply connected, they cannot carry
a topological obstruction; however,
they may carry a Hodge-theoretic obstruction.  These examples illustrate
the importance of the hypothesis $H^2(X,\sO_X)=0$ in Corollary~\ref{cor:surfaces}.

\begin{example}[with real points]
\label{ex:quartique avec points}
Examples of quartic surfaces $X \subset \P^3_\R$ such that $H^1(X(\R),\Z/2\Z)\neq H^1_\alg(X(\R),\Z/2\Z)$
are given in~\cite[Example~3.4~(c) and~(d)]{bochnakkucharzalgebraicmodels} and~\cite[Exemple~4.5.9]{mangoltelivre}.
As a further example, take~$X$ to be a very general small real deformation of a smooth quartic surface containing a real line. The cohomology class of the line deforms to a class $\alpha\in H^1(X(\R),\Z/2\Z)$ such that $\deg(\alpha\smile\cl_{\R}(\sO_X(1)))\neq 0\in\Z/2\Z$ by Ehresmann's theorem, and $H^1_\alg(X(\R),\Z/2\Z)$ is generated by $\cl_{\R}(\sO_X(1))$ by the Noether--Lefschetz theorem. Since $\deg(\cl_{\R}(\sO_X(1))\smile\cl_{\R}(\sO_X(1)))=0 \in \Z/2\Z$, one has $\alpha\notin H^1_\alg(X(\R),\Z/2\Z)$.
\end{example}

\begin{example}[with no real point]
\label{ex:kollar quartique sans point}
Let $X \subset \P^3_\R$ be a very general quartic surface
such that $X(\R)=\emptyset$.
As~$\CH_1(X)$ is generated by the class of a hyperplane section,
which has genus~$3$,
Theorem~\ref{th:phi}
immediately implies that~$X$
does not contain any geometrically irreducible curve of even geometric genus.
This example was first noted by Koll\'ar, see \cite[p.~15,
Remarque~(3)]{colliotthelenemadore}.
\end{example}

\subsection{\texorpdfstring{Cycle-theoretic obstructions (failures of the real integral Hodge conjecture)}{Cycle-theoretic obstructions}}
\label{subsec:cycle theoretic obs}

We now give examples, with or without real points,
of cycle-theoretic obstructions to the surjectivity of~$\phi$.
In view of Proposition \ref{prop:real(1,1)},
such obstructions
cannot occur on curves or surfaces.
The examples we give
are simply connected threefolds which satisfy $H^2(X,\sO_X)=0$.   As such,
they cannot carry a topological or Hodge-theoretic obstruction.
Their construction relies on Koll\'ar's specialisation method to provide
counterexamples to the integral Hodge conjecture (see~\cite{trentoexamples}),
as implemented by Totaro~\cite{totarocontreexemples}.
To the best of our knowledge, Example~\ref{ex:a la totaro avec points} constitutes
a new kind of example of a real variety~$X$
such that $H_1(X(\R),\Z/2\Z)\neq H_1^\alg(X(\R),\Z/2\Z)$
(see Remark~\ref{rk:covers hknotalg}~(i)).
Here, the underlying phenomenon
is a defect of the complex integral Hodge conjecture.

\begin{example}[with real points]
\label{ex:a la totaro avec points}
The example will be a hypersurface $X \subset \P^1_\Q \times \P^3_\Q$ of bidegree~$(4,4)$.
Let us first fix a prime number~$p$ (for instance $p=2$)
and choose $f_p \in \Z[u,v,x_0,\dots,x_3]$,
bihomogeneous of bidegree~$(4,4)$ in $(u,v)$, $(x_0,\dots,x_3)$,
whose zero locus $X_{\mkern-1muf_p} \subset \P^1_\Q \times \P^3_\Q$ is smooth
and whose reduction modulo~$p$ is equal to
$v(u^3x_0^4+u^2vx_1^4+uv^2x_2^4+v^3x_3^4)$.
Let us choose $f_\R \in \Q[u,v,x_0,\dots,x_3]$,
bihomogeneous of bidegree~$(4,4)$,
such that $X_{\mkern-1muf_\R} \subset \P^1_\Q\times\P^3_\Q$ is a smooth hypersurface
containing
$\P^1_{\Q}\times\{[1:0:0:0]\}$.
Such~$f_\R$ exist
as a consequence of the refined Bertini theorem of~\cite[Theorem~7]{kleimanaltman}.
Finally, let $f \in \Q[u,v,x_0,\dots,x_3]$ be arbitrarily close to~$f_\R$ in the real topology and to~$f_p$ in the $p$\nobreakdash-adic topology,
and let $X=X_{\mkern-1muf}$.
As~$f$ is $p$\nobreakdash-adically close to~$f_p$,
the argument used in the proof of \cite[Theorem~3.1]{totarocontreexemples}
shows that every curve in~$X_\C$ has even degree over~$\P^1_\C$.
A~fortiori, every curve in~$X$ has even degree over~$\P^1_\R$:
the push-forward map $H_1(X(\R),\Z/2\Z) \to H_1(\P^1(\R),\Z/2\Z)$
must vanish on $H_1^\alg(X(\R),\Z/2\Z)$.
On the other hand, as~$f$ is close to~$f_\R$,
this push-forward map is surjective. (Its surjectivity
is indeed invariant under small real deformations, by Ehresmann's theorem.)
Hence $H_1(X(\R),\Z/2\Z)\neq H_1^\alg(X(\R),\Z/2\Z)$.
\end{example}

\begin{example}[with no real point]
\label{ex:a la totaro sans point}

Let~$B$ be a smooth, projective conic over~$\Q$ such that $B(\R)=\emptyset$.
Let~$p$ be a prime of good reduction for~$B$,
\emph{i.e.},
such that there exists a smooth $\Z_{(p)}$\nobreakdash-scheme~$\sB$ with generic
fiber~$B$ and special fiber~$\P^1_{\F_p}$.
Let $\sY = \sB \times \P^3$,
$Y = B \times \P^3$,
and $\sY_0 = \sY \otimes \F_p = \P^1_{\F_p} \times \P^3_{\F_p}$.
Let $\sO_{\sY}(4,4)$ be the dual of $\omega_{\sB}^{\otimes 2} \boxtimes \omega_{\P^3}$.
As $H^1(\sY,\sO_\sY(4,4))=0$, the restriction
map $H^0(\sY,\sO_\sY(4,4))\to H^0(\sY_0,\sO_{\sY_0}(4,4))$ is onto.
In addition, its kernel is Zariski dense in $H^0(Y,\sO_Y(4,4))$ (viewed as an affine
space over~$\Q$).
We can therefore choose a section $f \in H^0(\sY,\sO_\sY(4,4))$
which reduces to
$v(u^3x_0^4+u^2vx_1^4+uv^2x_2^4+v^3x_3^4) \in H^0(\sY_0,\sO_{\sY_0}(4,4))$
modulo~$p$,
where $[u:v]$ and $[x_0:x_1:x_2:x_3]$ respectively denote the homogeneous
coordinates of~$\P^1_{\F_p}$ and of~$\P^3_{\F_p}$,
and
such that the zero locus $X_{\mkern-1muf} \subset B \times \P^3$ is smooth over~$\Q$.
Let $X = X_{\mkern-1muf} \otimes_\Q \R$.
The argument used in the proof of \cite[Theorem~3.1]{totarocontreexemples}
shows that every curve in~$X_\C$ has even degree over~$B_\C$.
A~fortiori, every curve in~$X$ has even degree over~$B_\R$.
By Proposition~\ref{prop:genusundermorphism}, it follows that every curve in~$X$ has odd genus,
so that $\elw_1(X)=2$.
\end{example}

In view of Examples \ref{ex:a la totaro avec points} and \ref{ex:a la totaro sans point}, it is natural to ask:

\begin{question}
Does there exist a smooth, proper and geometrically irreducible variety~$X$ over $\R$ such that $X_\C$ satisfies the complex integral Hodge conjecture for $1$\nobreakdash-cycles, but such that $X$ has a cycle-theoretic obstruction to the surjectivity of~$\phi$?
\end{question}

Although we do not know any variety over $\R$ with these properties, we will exhibit such examples over non-archimedean real closed fields in \cite[\textsection\ref*{BW2-par:cexnonarch}]{bwpartie2}.

\section{Bloch--Ogus theory and torsion \texorpdfstring{$1$}{1}-cycles}
\label{sec:blochogus}

In this section, we apply Bloch--Ogus theory to investigate the
consequences of the real integral Hodge conjecture for the
study of the group $\CH_1(X)_\tors$ and of its image in $H_1(X(R),\Z/2\Z)$ by the
Borel--Haefliger cycle class map.
The main result is Theorem~\ref{th:ch1torsion}.
We illustrate its applicability in~\textsection\ref{subsec:example torsion for real quartic}
by computing $\CH_1(X)_\tors$ for a smooth quartic threefold~$X$ over
a real closed field~$R$, under the assumption that~$X$ satisfies the real integral Hodge conjecture
and that $X(R)=\emptyset$.
As a preliminary step, we prove, in~\textsection\ref{subsec:coniveau and torsion cycles},
that the defect of the real integral Hodge conjecture, for a threefold,
can be interpreted in terms of unramified cohomology; this is
a real analogue of a theorem of Colliot-Th\'el\`ene and Voisin~\cite{ctvoisin}.
Along the way, we also obtain some new technical results of independent interest
in~\textsection\ref{subsec:complements on bo}, such as the vanishing of
the upper differentials in the coniveau spectral sequence for the equivariant cohomology of~$X(C)$
with coefficients in an arbitrary $G$\nobreakdash-module~$M$
(see Proposition~\ref{prop:various bo}~(vi)),
an assertion which was previously known only when~$M$ is one of $\Z/2\Z$, $\Q/\Z$, $\Q/\Z(1)$
(due to the work
of Colliot-Th\'el\`ene, Parimala, Scheiderer, van~Hamel;
see~\cite[\textsection3.1]{ctparimala}, \cite[Proposition~19.8]{scheiderer},
\cite[\textsection3]{ctscheiderer},
\cite[\textsection2]{vanhamelabeljacobi}).

\subsection{Complements on Bloch--Ogus theory}
\label{subsec:complements on bo}

We collect, in~\textsection\ref{subsec:complements on bo}, some results
on Bloch--Ogus theory in the context of equivariant semi-algebraic cohomology.
General references for this theory are~\cite{blochogus}
and~\cite{blochogusgabber}.
It was applied
in real algebraic geometry
in \cite{ctparimala},
\cite[Chapter~19]{scheiderer},
\cite{ctscheiderer},
\cite{vanhamelabeljacobi},
\cite{hellervoineagu}
with $\Z/2\Z$ coefficients
and
in~\cite[\textsection2]{Hilbert17}
with~$\Z_2$ coefficients.

We fix a $G$\nobreakdash-module~$M$.
Applying \cite[Remark~5.1.3~(3)]{blochogusgabber}
to the cohomology theory with supports
$(X,Z)\mapsto H^*_{G,Z(C)}(X(C),M)$,
we obtain
the \emph{coniveau spectral sequence}, which, in view of~\eqref{eq:equivariant purity subvariety},
takes the form
\begin{align}
\label{eq:coniveau spectral sequence}
E_1^{p,q} = \bigoplus_{Z \subseteq X} \varinjlim_{U \subseteq Z} H^{q-p}_G(U(C),M(-p))
\Rightarrow H^{p+q}_G(X(C),M)
\end{align}
for any smooth variety~$X$ over~$R$;
here,
the direct sum ranges over the irreducible Zariski closed subsets $Z \subseteq X$
of codimension~$p$ and the direct limit ranges over the dense Zariski open
subsets $U \subseteq Z$.  Clearly $E_r^{p,q}=0$ for all $r\geq 1$ whenever $p>q$.

This cohomology theory with supports satisfies the \'etale excision
and homotopy invariance axioms
of \cite[\textsection5.1, \textsection5.3]{blochogusgabber},
as a consequence of the semi-algebraic implicit function theorem
(see~\cite[Example~5.1]{delfsknebuschintrolocallysemialg}),
for \'etale excision,
and of~\eqref{eq:finite dimensional borel} and \cite[Corollary~4.5]{delfshomotopyaxiom},
for homotopy invariance.
By \cite[Corollary~5.1.11 and Proposition~5.3.2]{blochogusgabber},
we deduce,
for any $p,q$,
a canonical isomorphism
\begin{align}
\label{eq:bloch-ogus iso}
E_2^{p,q}=H^p(X,\sH^q_X(M))
\end{align}
if $\sH^q_X(M)$ denotes the Zariski sheaf
associated with the presheaf $U \mapsto H^q_G(U(C),M)$.

We denote by $N^pH^i_G(X(C),M) \subseteq H^i_G(X(C),M)$
the subgroup
of those classes~$\alpha$ for which there exists
a Zariski closed
subset $Z \subseteq X$ of codimension~$\geq p$
such that~$\alpha$ is supported on~$Z(C)$.
The filtration thus defined, called the coniveau filtration, is the one determined
by~\eqref{eq:coniveau spectral sequence}.

The preceding discussion also applies with equivariant cohomology replaced by
cohomology (or we may simply apply it to~$X_C$ viewed as a geometrically reducible variety over~$R$ in the naive way).
We let $\pi:X_C\to X$ denote the projection
map and
$\sH^q_{X_C}(M)$ the Zariski sheaf, on~$X_C$,
associated with the presheaf $U \mapsto H^q(U(C),M)$.
The next proposition summarises
the analogues, for semi-algebraic cohomology with coefficients in~$\Z$,
of some of the statements of \cite[\textsection2]{Hilbert17} (assertions~(i), (ii), (iii))
and of \cite[Chapter~19]{scheiderer}
and \cite[\textsection1.7--\textsection2]{vanhamelabeljacobi}
(assertions~(iv), (v), (vi)).

\begin{prop}
\label{prop:various bo}
Let~$X$ be a smooth variety over~$R$, of dimension~$d$.
\begin{itemize}
\item[(i)]
Let $M$ be a $G$\nobreakdash-module.
For every $p,q\geq 0$,
there are canonical isomorphisms $\sH^q_X(M[G])=\pi_*\sH^q_{X_C}(M)$
and  $H^p(X,\pi_*\sH^q_{X_C}(M))=H^p(X_C,\sH^q_{X_C}(M))$.
\item[(ii)] For every $q\geq 0$, the sheaves $\sH^q_X(\Z(q-1))$ and $\sH^q_X(\Z[G])$ are torsion-free.
\item[(iii)]
The real-complex exact sequence~\eqref{eq:real-complex sequence 01}
induces an exact sequence
\begin{align*}
\xymatrix@R=0.5ex@C=1em{
&*!<2.6em,0ex>\entrybox{
0 \to
\sH_X^q(\Z(q-1))
\to
\pi_* \sH_{X_C}^q(\Z)
\to
\sH_X^q(\Z(q))
}\ar`r/6pt[d]`[l]`[dl]`[d][d] \\
&*!<.35em,0ex>\entrybox{
 \sH_X^{q+1}(\Z(q-1))\to \pi_* \sH_{X_C}^{q+1}(\Z)\to \sH_X^{q+1}(\Z(q))\to 0
}
}
\end{align*}
for every $q\geq 0$.
\item[(iv)]
Let $q>d$.
Let~$M$ be a $G$\nobreakdash-module.
Let us denote by $\iota:X(R) \to X$ the natural morphism of sites
from the semi-algebraic site of~$X(R)$ to the small Zariski site of~$X$.
The natural map $\sH^q_X(M) \to \iota_*H^q(G,M)$
is an isomorphism.
In addition, one has $H^p(X,\sH^q_X(M))=H^p(X(R),H^q(G,M))$
for all $p\geq 0$.
\item[(v)]
For $j\in\Z$
and $i \geq 0$,
the restriction map
$H^i_G(X(C),\Z(j))\to H^i_G(X(R),\Z(j))$
and the decomposition~\eqref{eq:canonical decomposition}
induce an isomorphism
\begin{align*}
H^i_G(X(C),\Z(j))/N^{i-d}H^i_G(X(C),\Z(j))
 \mkern3mu\isoto \mkern-20mu\bigoplus_{\substack{0\leq p<i-d \\ p\mkern1mu\equiv\mkern1mu i-j \text{ mod } 2}}\mkern-20muH^p(X(R),\Z/2\Z)\rlap{\text{.}}
\end{align*}
\item[(vi)]
The differential $E_r^{p,q}\to E_r^{p+r,q-r+1}$
of the coniveau spectral sequence~\eqref{eq:coniveau spectral sequence}
vanishes for all $p,q,r$ such that $r\geq 2$ and $q>d$
and for any $G$\nobreakdash-module~$M$.
\end{itemize}
\end{prop}

\begin{proof}
The first isomorphism of~(i) follows from~\eqref{eq:cohoeqnoneq}.
The second one
is obtained in~\cite[Lemma 2.2.1~(a)]{ctscheiderer} with~$\Z/2\Z$ coefficients and in \cite[Proposition~2.1]{Hilbert17}
in the setting of $2$\nobreakdash-adic cohomology.
The arguments given there apply verbatim with equivariant semi-algebraic cohomology with coefficients in~$M$.

For any prime number~$\ell$,
the sheaf associated with $U\mapsto H^q_\et(U,\Z_\ell(q-1))$
is torsion-free (see~\cite[Proposition~2.2]{Hilbert17},
where the assumption that $\ell=2$ is not used;
the underlying argument, which rests on the Bloch--Kato conjecture,
goes back to \cite[Proof of Theorem~1]{blochsrinivas}
and to \cite[Th\'eor\`eme~3.1]{ctvoisin}
and does not depend on
the nature of the ground field~$R$).
By the comparison between
equivariant semi-algebraic cohomology and $\ell$\nobreakdash-adic cohomology,
it follows that the sheaf $\sH^q_X(\Z(q-1))$ is torsion-free.
Applying this to~$X_C$
and noting that $\sH^q_X(\Z[G])=\pi_*\sH^q_{X_C}(\Z(q-1))$,
we see that $\sH^q_X(\Z[G])$ is torsion-free as well.
The proof of~(ii) is complete.

Assertion~(iii)
follows from~(i) and~(ii).  This is observed in~\cite[Proposition~2.5]{Hilbert17}
for $2$\nobreakdash-adic cohomology and the same proof applies here.

Assertion~(iv) for torsion $G$\nobreakdash-modules~$M$
is due to Scheiderer \cite[Corollary~6.9.1 and Corollary~19.5]{scheiderer}.
We adapt his arguments to an arbitrary $G$\nobreakdash-module~$M$
as follows.
The restriction map
$H^q_G(U(C),M) \to H^q_G(U(R),M)$
is an isomorphism
for any $q>d$
and any affine open subset $U \subseteq X$,
by Lemma~\ref{lem:restrmapinjective}.
In addition,
the map $H^q_G(U(R),M) \to H^0(U(R),H^q(G,M))$
induced by the spectral sequence
\begin{align*}
E_2^{a,b}(U)=H^{a}(U(R),H^{b}(G,M)) \Rightarrow H^{a+b}_G(U(R),M)
\end{align*}
(see~\eqref{eq:first spectral sequence})
becomes an isomorphism after sheafification with respect to~$U$,
since the sheaf associated with the presheaf $U \mapsto E_2^{a,b}(U)$ vanishes when $a>0$
according to \cite[Proposition~19.2.1]{scheiderer}
(see also the proof of \cite[Lemma~1.2]{vanhamelabeljacobi}).
All in all, we obtain the desired isomorphism $\sH^q_X(M)\isoto \iota_*H^q(G,M)$.
On the other hand, the functor~$\iota_*$ is exact,
by \cite[Theorem~19.2]{scheiderer}.
By the Leray spectral sequence for~$\iota$, assertion~(iv) follows.

Assertion~(vi)
for $M=\Z/2\Z$
is due to van~Hamel
\cite[Theorem~2.1]{vanhamelabeljacobi}.
Based on Proposition~\ref{prop:truncated projection integral coeff},
we extend his arguments to integral coefficients.
In fact, we shall prove~(v) and~(vi) for $M=\Z(j)$ simultaneously.
To this end, we first note that by purity for semi-algebraic cohomology,
the group $H^p_{Z(R)}(X(R),\Z/2\Z)$ vanishes for any closed subset $Z \subseteq X$ of codimension~$>p$.
(See~\eqref{eq:purity mod 2} and the proof of \cite[Chapter~VI, Lemma~9.1]{milneet}.)
Hence the restriction map induces a map
\begin{align}
\label{eq:map assertion v}
H^i_G(X(C),\Z(j))/N^{i-d}H^i_G(X(C),\Z(j))
 \mkern3mu\to \mkern-20mu\bigoplus_{\substack{0\leq p<i-d \\ p\mkern1mu\equiv\mkern1mu i-j \text{ mod } 2}}\mkern-20muH^p(X(R),\Z/2\Z)\rlap{\text{,}}
\end{align}
which is surjective by Proposition~\ref{prop:truncated projection integral coeff}.
Let us now consider the coniveau spectral sequence~\eqref{eq:coniveau spectral sequence}
associated with $M=\Z(j)$.
We have $E_2^{p,q}=H^p(X(R),H^q(G,\Z(j)))$ whenever $q>d$,
according to~(iv)
and to~\eqref{eq:bloch-ogus iso}.
The target of~\eqref{eq:map assertion v}
can therefore be rewritten as $\bigoplus_{0 \leq p<i-d} E_2^{p,i-p}$.
On the other hand, the domain of~\eqref{eq:map assertion v} has the same cardinality
as $\bigoplus_{0 \leq p<i-d} E_\infty^{p,i-p}$, which is finite.
As $E_\infty^{p,i-p}$ is a subquotient of $E_2^{p,i-p}$ for all~$p$ and as~\eqref{eq:map assertion v}
is surjective, it follows that~\eqref{eq:map assertion v} is an isomorphism
and that $E_\infty^{p,i-p}=E_2^{p,i-p}$ for $p<i-d$.
We have thus established~(v), as well as~(vi) for $M=\Z(j)$ (with $q=i-p$).

It remains to check~(vi) for an arbitrary $G$\nobreakdash-module~$M$, which we now fix.
The map $H^q(G,N(j)) \to H^q(G,M(j))$
induced by
the natural morphism of $G$\nobreakdash-modules $N=\Z^{(H^0(G,M))} \oplus \Z(1)^{(H^0(G,M(1)))} \to M$
is surjective for $q=0$ and any~$j$, therefore also for any~$q$ and any~$j$
since the boundary map $H^{q-1}(G,M(j+1))\to H^q(G,M(j))$ of~\eqref{eq:real-complex sequence 01}
is surjective when $q>0$ (see \cite[Chapter~III, Corollary~5.7, Proposition~5.9, Corollary~6.6]{brown}).
For $q>0$, this is a surjection between $\Z/2\Z$\nobreakdash-vector spaces, hence it is a split surjection.
We deduce that the group $H^p(X(R),H^q(G,N))$ surjects onto $H^p(X(R),H^q(G,M))$
for all~$p\geq 0$ and all $q>0$.
This remark, together with assertion~(iv) for the $G$\nobreakdash-modules~$M$ and~$N$
and with assertion~(vi) for the $G$\nobreakdash-module~$N$
(which we have already shown to hold) implies assertion~(vi) for~$M$.
\end{proof}

\subsection{Coniveau and the real integral Hodge conjecture for \texorpdfstring{$1$}{1}-cycles}
\label{subsec:coniveau and torsion cycles}

If~$X$ is a smooth, proper and irreducible complex variety of dimension~$d$
such that $\CH_0(X)$ is supported on a surface (\emph{i.e.}, such that there exists
a closed subvariety $Y\subseteq X$ of dimension~$\leq 2$ such that $\CH_0(Y) \twoheadrightarrow \CH_0(X)$),
Colliot-Th\'el\`ene and Voisin~\cite[Corollaire~3.12]{ctvoisin}
have shown that the defect of the integral Hodge conjecture for $1$\nobreakdash-cycles on~$X$
is measured by the group $H^{d-3}(X,\sH^d_X(\Q/\Z(d-1)))$.
We establish, in Proposition~\ref{prop:ihc defect and bo} below, an analogue of this statement
for real varieties.

For a variety~$X$ defined over a real closed field~$R$, a $G$\nobreakdash-module~$M$
and any $q\geq 0$,
we set $\sH^q_X(M)_0=\Ker\left(\sH^q_X(M) \to \iota_*H^q(G,M)\right)$
(notation as in Proposition~\ref{prop:various bo}~(iv)).

\begin{prop}
\label{prop:ihc defect and bo}
Let~$X$ be a smooth, proper and geometrically irreducible variety,
of dimension~$d$, over a real closed field~$R$.
If $\CH_0(X_{C'})$ is supported on a surface
for every algebraically closed field~$C'$ containing~$R$, there is a canonical isomorphism
\begin{align*}
H^{d-3}(X,\sH^d_X(\Q/\Z(d-1))_0) = \Coker\left(\CH_1(X) \to H^{2d-2}_G(X(C),\Z(d-1))_0\right)_{\tors}\rlap{.}
\end{align*}
If moreover $R=\R$ or $H^2(X,\sO_X)=0$, the real integral Hodge conjecture (see~\textsection\ref{par:realIHC})
holds for $1$\nobreakdash-cycles on~$X$ if and only if $H^{d-3}(X,\sH^d_X(\Q/\Z(d-1))_0)=0$.
\end{prop}

\begin{proof}
Let us consider
the commutative square
\begin{align}
\begin{aligned}
\label{diag:comm square hd0}
\xymatrix@R=3ex{
\sH^d_X(\Q/\Z(d-1)) \ar[r] \ar[d] & \sH^{d+1}_X(\Z(d-1)) \ar[d]^(.45)\wr \\
\iota_*H^d(G,\Q/\Z(d-1)) \ar[r]^\sim & \iota_*H^{d+1}(G,\Z(d-1))\rlap{\text{,}}
}
\end{aligned}
\end{align}
whose horizontal maps come from the short exact sequence $0\to\Z\to\Q\to\Q/\Z\to 0$.
The vertical map on the right is an isomorphism,
by Proposition~\ref{prop:various bo}~(iv).
Hence the kernel of the top horizontal map is $\sH^d_X(\Q/\Z(d-1))_0$ and we therefore obtain,
in view of Proposition~\ref{prop:various bo}~(ii),
a short exact sequence
\begin{align}
\label{eq:sHd z q qz}
\xymatrix@C=1.5em{
0 \ar[r] & \sH^d_X(\Z(d-1)) \ar[r] & \sH^d_X(\Q(d-1)) \ar[r] & \sH^d_X(\Q/\Z(d-1))_0 \ar[r] & 0\rlap{\text{.}}
}
\end{align}
The sheaf $\sH^d_X(\Q(d-1))$ is a direct summand of $\pi_* \sH^d_{X_C}(\Q)$.
By Proposition~\ref{prop:various bo}~(i), it follows
that $H^{d-3}(X,\sH^d_X(\Q(d-1)))$ injects into $H^{d-3}(X_C,\sH^d_{X_C}(\Q))$.
The latter group vanishes since~$\CH_0(X_{C'})$
is supported on a surface for every algebraically closed fields~$C'$ containing~$C$
(see \cite[Proposition~3.3~(ii)]{ctvoisin},
whose proof goes through over an arbitrary algebraically closed field of characteristic~$0$
provided one replaces the hypothesis on~$\CH_0(X_C)$ with the same hypothesis on $\CH_0(X_{C'})$
for all~$C'$).
In view of~\eqref{eq:sHd z q qz},
we deduce that
\begin{align}
\label{eq:hd-3 hd-2shd}
H^{d-3}(X,\sH^d_X(\Q/\Z(d-1))_0)=H^{d-2}(X,\sH^d_X(\Z(d-1)))_\tors\rlap{\text{.}}
\end{align}
Let us now consider the coniveau spectral sequence~\eqref{eq:coniveau spectral sequence}
for $M=\Z(d-1)$.
We have $H^{d-2}(X,\sH^d_X(\Z(d-1)))=E_2^{d-2,d}$
(see~\eqref{eq:bloch-ogus iso}).
By Proposition~\ref{prop:various bo}~(vi) and in view of the fact that $E_1^{p,q}=0$ for $p>q$,
we also have $E_\infty^{d-2,d}=E_2^{d-2,d}$.
On the other hand,
we have
\begin{align}
\label{eq:nd-1}
N^{d-1} H^{2d-2}_G(X(C),\Z(d-1)) = \Im\left(\CH_1(X) \to H^{2d-2}_G(X(C),\Z(d-1))\right)
\end{align}
as a consequence of equivariant purity (see~\eqref{eq:equivariant purity subvariety})
and
\begin{align}
\label{eq:nd-2}
N^{d-2} H^{2d-2}_G(X(C),\Z(d-1))=H^{2d-2}_G(X(C),\Z(d-1))_0
\end{align}
according to Proposition~\ref{prop:various bo}~(v) applied with $i=2d-2$, $j=d-1$.
As the quotient of~\eqref{eq:nd-2} by~\eqref{eq:nd-1} is~$E_\infty^{d-2,d}$,
we conclude that
\begin{align}
\label{eq:coker cycle map tors}
H^{d-2}(X,\sH^d_X(\Z(d-1)))=
\Coker\left(\CH_1(X) \to H^{2d-2}_G(X(C),\Z(d-1))_0\right)\mkern-3mu\rlap{\text{.}}
\end{align}
Together with (\ref{eq:hd-3 hd-2shd}), this proves the first statement. If $R=\R$ or $H^2(X,\sO_X)=0$, the cycle class map
$\CH_1(X_C) \otimes_\Z\Q \to \Hdg^{2d-2}(X(C),\Q(d-1))$
is onto (we use the convention that all classes are Hodge if $H^2(X,\sO_X)=0$): by the Lefschetz principle one reduces to the case $C=\C$, for which see~\cite[p.~91]{lewissurvey}.
It follows, by a trace argument, that the torsion subgroup of the
right-hand side of~\eqref{eq:coker cycle map tors}
is canonically isomorphic to
$\Coker\left(\CH_1(X) \to \Hdg^{2d-2}_G(X(C),\Z(d-1))_0\right)$, as desired.
\end{proof}

\begin{rmks}
\label{rks:ch0 supported on surface}
(i) According to the generalised Bloch conjecture,
the hypothesis that~$\CH_0(X_{C'})$ is supported on a surface
for every algebraically closed field~$C'$ containing~$R$
should be equivalent to
the vanishing of $H^i(X,\sO_X)$ for all~$i\geq 3$
(see \cite[Conjecture~1.11]{voisinweyl},
\cite[\textsection3]{jannsenmotivicsheaves},
\cite[Th\'eor\`eme~22.17]{voisinbook}).

(ii) If~$R$ has infinite transcendence degree over~$\Q$
(for instance, if $R=\R$),
the group $\CH_0(X_{C'})$ is supported on a surface
for every algebraically closed field~$C'$ containing~$R$
if and only if it is so for $C'=C$
(by an argument known as decomposition of the diagonal, see \cite[Appendix to Lecture~1]{blochlectures}).

(iii)
As a consequence of Proposition~\ref{prop:ihc defect and bo},
if~$X$ is a real threefold such that $\CH_0(X_\C)$ is supported on a surface (for instance, a
uniruled threefold),
the defect of the real integral Hodge conjecture for~$X$ is measured by
the subgroup $H^3_\nr(X,\Q/\Z(2))_0$
of the unramified cohomology group
$H^3_\nr(X,\Q/\Z(2))=H^0(X,\sH^3_X(\Q/\Z(2)))$
consisting of those classes whose restriction to any real point of~$X$ vanishes.
\end{rmks}

\subsection{Torsion \texorpdfstring{$1$}{1}-cycles}

Thanks to Proposition~\ref{prop:ihc defect and bo},
we are now in a position to derive consequences of the real integral Hodge conjecture on the
study of $\CH_1(X)_\tors$.

If~$X$ is a smooth, proper and irreducible variety of dimension~$d$ over~$R$,
we denote by
$\lambda:\CH_1(X)_\tors \to H^{2d-3}_\et(X,\Q/\Z(d-1))$
Bloch's Abel--Jacobi map.
(See~\cite{blochtorsion} for its construction
over an algebraically closed field;
the construction goes through over a real closed field,
as explained in \cite[Theorem~3.1]{vanhamelabeljacobi}.)
The map~$\lambda$ takes its values in the inverse image
$H^{2d-3}_G(X(C),\Q/\Z(d-1))_0$
of $H^{2d-2}_G(X(C),\Z(d-1))_0$
by the boundary map
\begin{align}
\label{eq:boundary map h2d-3}
H^{2d-3}_G(X(C),\Q/\Z(d-1))\to H^{2d-2}_G(X(C),\Z(d-1))
\end{align}
since the composition of~$\lambda$ with this boundary map coincides with the equivariant cycle
class map (see \cite[Corollaire~1]{ctsso} and Theorem~\ref{th:conditions de krasnov}).

In the next statements, we denote by~$\cl_R$ the two
Borel--Haefliger
cycle class maps
$\Pic(X) \to H^1(X(R),\Z/2\Z)$
and $\CH_1(X) \to H^{d-1}(X(R),\Z/2\Z)$.

\begin{thm}
\label{th:ch1torsion}
Let~$X$ be a smooth, proper and geometrically irreducible variety,
of dimension~$d$, over a real closed field~$R$.
Assume that $R=\R$ or $H^2(X,\sO_X)=0$,
and that~$X$ satisfies the real integral Hodge conjecture
for $1$\nobreakdash-cycles.
Finally, assume that for every algebraically closed field~$C'$ containing~$R$,
the group $\CH_0(X_{C'})$ is supported on a surface.
\begin{itemize}
\item[(i)] Bloch's Abel--Jacobi map induces a surjection
\begin{align*}
\lambda: \CH_1(X)_\tors \twoheadrightarrow H^{2d-3}_G(X(C),\Q/\Z(d-1))_0
\end{align*}
(even an isomorphism if $d\leq 3$).
\item[(ii)]
The subgroup $\cl_R(\CH_1(X)[2^\infty])$
and the image of the map
\begin{align*}
H^2_G(X(C),\Z(1))\to H^2_G(X(R),\Z(1))=H^1(X(R),\Z/2\Z)
\end{align*}
obtained by composing the restriction map
and the decomposition~\eqref{eq:canonical decomposition}
are exact orthogonal complements under the Poincar\'e duality pairing.
\end{itemize}
\end{thm}

In view of Proposition~\ref{prop:real(1,1)nonarch}
and Theorem~\ref{th:conditions de krasnov}, Theorem \ref{th:ch1torsion} (ii) has the following corollary. It is the twin of Theorem~\ref{th:nohodgetheoreticob}~(i); when $d=2$,
the two are equivalent.

\begin{cor}
\label{cor:ch1torsduality}
Let $X$ be as in Theorem \ref{th:ch1torsion}. If $H^2(X,\sO_X)=0$,
the subgroups $\cl_R(\CH_1(X)[2^\infty])$ and
$\cl_R(\Pic(X))$ are exact orthogonal
complements under the Poincar\'e duality pairing.
\end{cor}

\begin{proof}[Proof of Theorem \ref{th:ch1torsion}]
We start with~(i).
It is well known that~$\lambda$ is injective if~$d\leq 3$
(see \cite[Corollaire~1]{ctsso}).
It is also a general fact that $\Im(\lambda)=N^{d-2}H^{2d-3}_G(X(C),\Q/\Z(d-1))$
(see \cite[Theorem~3.1]{vanhamelabeljacobi}).
These remarks, together with~\eqref{eq:nd-2}
and with the next lemma, imply~(i).

\begin{lem}
The inverse image of $N^{d-2}H^{2d-2}_G(X(C),\Z(d-1))$
by the boundary map~\eqref{eq:boundary map h2d-3}
is equal to
$N^{d-2}H^{2d-3}_G(X(C),\Q/\Z(d-1))$.
\end{lem}

\begin{proof}
Let us denote
by~$A_r^{p,q}$ (resp., $B_r^{p,q}$)
the term~$E_r^{p,q}$
(resp., $E_r^{p,q+1}$)
of the coniveau spectral sequence~\eqref{eq:coniveau spectral sequence}
associated with $M=\Q/\Z(d-1)$
(resp., $M=\Z(d-1)$).
The short exact sequence $0\to\Z\to\Q\to\Q/\Z\to 0$
induces a morphism
of cohomology theories with supports
$H^*_{G,Z(C)}(X(C),\Q/\Z(d-1)) \to H^{*+1}_{G,Z(C)}(X(C),\Z(d-1))$
in the sense of \cite[\textsection5.1]{blochogusgabber}
and, hence,
a morphism of spectral sequences
\begin{align}
\begin{aligned}
\xymatrix@R=2ex{
\vphantom{A_G^{p+1}}A_1^{p,q} \ar@{=>}[r] \ar[d] & \rlap{$H^{p+q}_G(X(C),\Q/\Z(d-1))$}\phantom{H^{p+q+1}_G(X(C),\Z(d-1))} \ar[d] \\
\vphantom{B_G^{p+1}}B_1^{p,q} \ar@{=>}[r] & H^{p+q+1}_G(X(C),\Z(d-1))\rlap{\text{.}}
}
\end{aligned}
\end{align}
The sheaf $\sH^q_X(\Q(d-1))$ vanishes for $q>d$,
being a direct summand of $\pi_*\sH^q_{X_C}(\Q)$;
hence,
the boundary map $\sH^q_X(\Q/\Z(d-1)) \to \sH^{q+1}_X(\Z(d-1))$
is an isomorphism for $q>d$
and is surjective for $q=d$.
We deduce, on the one hand, that $A_2^{p,q} \isoto B_2^{p,q}$ for $q>d$
(thanks to~\eqref{eq:bloch-ogus iso}), and, on the other hand,
that the sequence
\begin{align}
\label{eq:shdqzd-1}
\xymatrix@C=1.5em{
0 \ar[r] & \sH^d_X(\Q/\Z(d-1))_0 \ar[r] & \sH^d_X(\Q/\Z(d-1)) \ar[r] & \sH^{d+1}_X(\Z(d-1)) \ar[r] & 0
}
\end{align}
is exact (as we have already observed its exactness on the left, in~\eqref{diag:comm square hd0}).
Now, our hypothesis that~$X$ satisfies the real integral Hodge conjecture
for $1$\nobreakdash-cycles
implies,
by Proposition~\ref{prop:ihc defect and bo},
that $H^{d-3}(X,\sH^d_X(\Q/\Z(d-1))_0)=0$.
It follows, in view of~\eqref{eq:shdqzd-1}
and~\eqref{eq:bloch-ogus iso}, that $A_2^{d-3,d} \hookrightarrow B_2^{d-3,d}$.
Thus, the map $A_2^{p,q}\to B_2^{p,q}$ is injective for all $p,q$ such that $p+q=2d-3$
and $q\geq d$.
As the differentials $B_r^{p-r,q+r-1}\to B_r^{p,q}$
vanish for $r\geq 2$ and $q\geq d$ (indeed, even for $q\geq d+1-r$, see Proposition~\ref{prop:various bo}~(vi)), we conclude that $A_\infty^{p,q} \hookrightarrow B_\infty^{p,q}$
for all $p,q$ such that $p+q=2d-3$ and $q\geq d$.
\end{proof}

We now deduce~(ii) from~(i).  By~(i),
by Theorem~\ref{th:conditions de krasnov},
and  by the compatibility between~$\lambda$
and the equivariant cycle class map (see \cite[Corollaire~1]{ctsso}),
we have
\begin{align*}
\cl_R(\CH_1(X)[2^\infty])=
\cl_R(\CH_1(X)_\tors)=
\psi\left(\left(H^{2d-2}_G(X(C),\Z(d-1))_0\right){}_{\mkern-4mu\tors}\right)\mkern-3mu\rlap{\text{,}}
\end{align*}
where~$\psi$ is as in Definition~\ref{def:psi with point}.

From this point on, we proceed as in the proof of Theorem~\ref{th:image psi}.
Let us take up, from~\textsection\ref{subsubsec:normal bundle with mod 2 coefficients},
the notation $H^p=H^p(X(R),\Z/2\Z)$.
The map~$u_2$ appearing in~\eqref{eq:selfdual sequence}
can be inserted into a diagram
\begin{align}
\label{diag:image of psiprimetors}
\owrepositiontag{{\raise 11pt}{%
\xymatrix@C=1em@R=3ex{
&&\mkern-92muH^0\oplus H^1=H^1_G(X(R),\Z/2\Z) \ar[d]^{\delta_2} \ar[r]
& \rlap{$H^2_G(X(R),\Z(1))=H^1$}\phantom{H^2_G(X(R),\Z(1))}\mkern13mu \ar@<-.2em>[d] \\
H^1 \ar@<.3em>[ur]!(11.5,0) \ar[r] & H^{2-d} \oplus \cdots \oplus H^1 \ar[r]^(.42){u_2} & H^2_G(X(C),X(R),\Z/2\Z) \ar[r] & H^3_G(X(C),X(R),\Z(1))\rlap{\text{,}}
}}}
\end{align}
in which the leftmost arrows are the obvious inclusions,
the vertical maps are the connecting
homomorphisms of the localisation exact sequences,
and the canonical isomorphisms of the first row are the decompositions~\eqref{eq:canonical decomposition mod 2}
and~\eqref{eq:canonical decomposition}.
By the description of~$u_2$ given after~\eqref{eq:selfdual sequence},
this diagram commutes.
In addition, the map from the bottom left~$H^1$ to the top right~$H^1$
is the identity map.
It follows
that the kernel of the composition $\theta'_1:H^1 \to H^3_G(X(C),X(R),\Z(1))$
of all of the maps of the bottom row of~\eqref{diag:image of psiprimetors}
coincides with the image of the map appearing in the statement of Theorem~\ref{th:ch1torsion}~(ii).

Let $\theta:H^{2d-2}_G(X(C)\setminus X(R),\Z(d-1))\to H^{d-1}$
denote the map defined at the end of  the proof of Theorem~\ref{th:image psi}
and let $\theta_1 : H^{2d-3}_G(X(C)\setminus X(R),\Q/\Z(d-1)) \to H^{d-1}$
denote its composition with the boundary map arising from the short
exact sequence $0\to\Z\to\Q\to\Q/\Z\to 0$.
As remarked during the proof of Theorem~\ref{th:image psi},
the map~$\psi$ coincides with the composition of~$\theta$
with the isomorphism of
Lemma~\ref{lem:topological condition = restriction}.
Hence $\psi\left(\left(H^{2d-2}_G(X(C),\Z(d-1))_0\right){}_{\mkern-4mu\tors}\right)$
is equal to the image of~$\theta_1$.

Unravelling the definitions of~$\theta_1$ and of~$\theta'_1$ and applying Theorem~\ref{th:selfduality}, Proposition~\ref{prop:poincare duality equivariant},
and Remark~\ref{rks:equivariant poincare}~(ii),
we see that the Pontrjagin dual of~$\theta_1'$ is~$\theta_1$.
The kernel of~$\theta_1'$ and the image of~$\theta_1$ are therefore
exact orthogonal complements, which completes the proof of~(ii).
\end{proof}

\subsection{An example: torsion \texorpdfstring{$1$-cycles}{1-cycles} on real quartic threefolds}
\label{subsec:example torsion for real quartic}

To illustrate the contents of~\textsection\ref{sec:blochogus}, we now
determine the torsion subgroup of the Chow group of $1$\nobreakdash-cycles
of a real quartic threefold with no real point.

\begin{prop}
\label{prop:example quartic threefold torsion}
Let~$R$ be a real closed field and $X \subset \P^4_R$ a smooth quartic threefold
such that $X(R)=\emptyset$.
If~$X$ satisfies the
real integral Hodge conjecture for $1$\nobreakdash-cycles,
then there exists an isomorphism of
abelian groups $\CH_1(X)_\tors \simeq \Z/2\Z \oplus (\Q/\Z)^{30}$.
\end{prop}

We shall prove, in \cite[\textsection\ref*{BW2-sec:Fano}]{bwpartie2},
that if~$R=\R$, then~$X$ does
satisfy the real integral Hodge conjecture for
$1$\nobreakdash-cycles (an assertion which is equivalent to the existence of a geometrically
irreducible curve
of even geometric genus in~$X$, by Theorem~\ref{thm:relation ihc phi}),
so that the conclusion of Proposition~\ref{prop:example quartic threefold torsion}
holds unconditionally in this case.
We do not know whether the conclusion of Proposition~\ref{prop:example quartic threefold torsion} holds
with no assumption on~$R$.

\begin{proof}[Proof of Proposition~\ref{prop:example quartic threefold torsion}]
Applying \cite[Theorem~6]{roitmanrateqv} or the rational connectedness of Fano varieties shows that
the group $\CH_0(X_{C'})$ is supported on a point
for every algebraically closed field~$C'$ containing~$R$.
We can therefore apply Theorem~\ref{th:ch1torsion}~(i)
and conclude that Bloch's Abel--Jacobi map induces an isomorphism
$\CH_1(X)_\tors=H^3_G(X(C),\Q/\Z(2))$.
The latter group is an extension of $H^4_G(X(C),\Z(2))_\tors$ by
$H^3_G(X(C),\Z(2)) \otimes_\Z \Q/\Z$.
The next two lemmas now imply the proposition.
\end{proof}

\begin{lem}
\label{lem:computation cohomology real quartic}
There is a canonical isomorphism $H^4_G(X(C),\Z(2))=\Z\oplus\Z/2\Z$,
the first summand being generated by $\cl(L+\bar L)$,
where $L\subset X_C$ denotes a line and~$\bar L$ its conjugate.
\end{lem}

\begin{proof}
The class $\cl(L+\bar L) \in H^4_G(X(C),\Z(2))$
generates
a subgroup of index~$2$,
by Proposition~\ref{prop:conoyau norme}.
In addition, this class is not divisible by~$2$.
Indeed, if it were,
one would deduce, by taking the cup product with the class of~$\sO_X(1)$,
the surjectivity of the natural map $H^6_G(X(C),\Z(3)) \to H^6(X(C),\Z(3))=\Z$,
which would contradict Proposition~\ref{prop:hodgereel0cycles}
since $X(R)=\emptyset$.  The lemma follows.
\end{proof}

\begin{lem}
\label{lem:rkH3}
The finitely generated abelian group $H^3_G(X(C),\Z(2))$ has rank~$30$.
\end{lem}

\begin{proof}
Let us apply the Lefschetz fixed-point theorem
to the complex conjugation involution of~$X(C)$
(over an arbitrary real closed field, see~\cite{brumfielfixedpoint}).
As this involution has no fixed point,
we deduce from the canonical $G$\nobreakdash-equivariant isomorphisms
$H^{2i}(X(C),\Q)=\Q(-i)$
for $i\in \{0,1,2,3\}$
and from the vanishing of $H^1(X(C),\Q)$ and of $H^5(X(C),\Q)$
that the generator of~$G$ acts on $H^3(X(C),\Q)$
with trace~$0$.
On the other hand, the vector space $H^3(X(C),\Q)$ has dimension~$60$
(see \cite[Example~5.24]{eisenbudharris3264}).
This implies the lemma, since $H^3_G(X(C),\Q)=H^3(X(C),\Q)^G$.
\end{proof}

\begin{rmks}
\label{rmks:chow group of quartic threefold}
(i)
When $R=\R$,
one can exploit the structure of the group of real points of the intermediate Jacobian
of~$X$ to verify that
in the situation of
Proposition~\ref{prop:example quartic threefold torsion},
Theorem~\ref{th:ch1torsion}~(i)
allows one to produce,
at the price of a significantly more involved computation,
an isomorphism $\CH_1(X) \simeq \Z \oplus \Z/2\Z \oplus (\R/\Z)^{30}$.

(ii) 
It is possible to perform the computations of Proposition~\ref{prop:example quartic threefold torsion} for some smooth quartic threefolds with real points.
In this setting, Lemma~\ref{lem:rkH3} still holds,
by \cite[Chapter~I, p.~12, (2.5)]{silhol}, and \cite[Proposition~\ref*{BW2-prop:preuve Fano 1}]{bwpartie2}
allows us to prove, when $R=\R$, that at least some of these varieties still satisfy the real integral Hodge conjecture.
As an example, we have verified
that $\CH_1(X)_{\tors}\simeq\Z/2\Z\oplus(\Q/\Z)^{30}$ if~$X$ has homogeneous equation $x_0^4+x_1^4=x_2^4+x_3^4+x_4^4$.
\end{rmks}

\bibliographystyle{myamsalpha}
\bibliography{hodgereel}
\end{document}